\documentclass[10pt]{article}
\usepackage{a4wide}
\usepackage{amsmath}
\usepackage{amssymb}
\usepackage{latexsym}
\usepackage[pdftex]{graphicx}
\usepackage{enumerate}
\usepackage{url}

  \usepackage{tikz}
  \usetikzlibrary{calc,patterns}
  \usetikzlibrary{positioning,arrows}

  \usepackage{color}
  \definecolor{brown}{cmyk}{0, 0.72, 1, 0.45}
\definecolor{grey}{gray}{0.5}
\def\red{\color{black}}
\def\blue{\color{black}}
\def\grn{\color{black}}

\def\orange{\color{black}}
\newcommand{\mnote}[1]{}
\newcommand{\gnote}[1]{}
\newcommand{\onote}[1]{}

  \newenvironment{proof}{\vspace{1ex}\noindent{\bf Proof:}}{\hspace*{\fill}
  $\blacksquare$\vspace{1ex}}
  \newenvironment{proofof}[1]{\vspace{1ex}\noindent{\bf Proof of #1:}}{\hspace*{\fill}
  $\blacksquare$\vspace{1ex}}
  \def\noproof{\hspace*{\fill}$\blacksquare$}

  \newtheorem{theorem}{Theorem}[section]
  \newtheorem{lemma} [theorem] {Lemma}
  \newtheorem{corollary} [theorem] {Corollary}
\newtheorem{question} [theorem] {Question}

\newcommand{\Bcal}[0]{\ensuremath{{\mathcal B}}}
\newcommand{\Ccal}[0]{\ensuremath{{\mathcal C}}}
\newcommand{\Dcal}[0]{\ensuremath{{\mathcal D}}}

\newcommand{\Fcal}[0]{\ensuremath{{\mathcal F}}}
\newcommand{\Gcal}[0]{\ensuremath{{\mathcal G}}}

\newcommand{\Ical}[0]{\ensuremath{{\mathcal I}}}
\newcommand{\Jcal}[0]{\ensuremath{{\mathcal J}}}

\newcommand{\Pcal}[0]{\ensuremath{{\mathcal P}}}

\newcommand{\Tcal}[0]{\ensuremath{{\mathcal T}}}

\newcommand{\Vcal}[0]{\ensuremath{{\mathcal V}}}

\newcommand{\Xcal}[0]{\ensuremath{{\mathcal X}}}

\newcommand{\eR}[0]{\ensuremath{ \mathbb R}}

\newcommand{\eN}[0]{\ensuremath{ \mathbb N}}
\newcommand{\Zed}[0]{\ensuremath{ \mathbb Z}}

\newcommand{\Ztil}[0]{\tilde{Z}}

\newcommand{\norm}[1]{\ensuremath{\|#1\|}}

\newcommand{\Pee}[0]{\ensuremath{{\mathbb P}}}
\newcommand{\Ee}[0]{\ensuremath{{\mathbb E}}}

\newcommand{\isd}[0]{\hspace{.2ex} \raisebox{-.1ex}{$=$} \hspace{-1.5ex}
\raisebox{1ex}{{$\scriptstyle d$}} \hspace{.8ex} }
\newcommand{\convd}[0]{\hspace{.2ex} \raisebox{-.1ex}{$\rightarrow$}
\hspace{-1.5ex}
\raisebox{1ex}{{$\scriptstyle d$}} \hspace{.8ex} }

 \newcommand{\eps}{\varepsilon}
\newcommand{\dtv}{d_{\text{TV}}}

\DeclareMathOperator{\diam}{diam}

\DeclareMathOperator{\area}{area}
\DeclareMathOperator{\Po}{Po}
\DeclareMathOperator{\Bi}{Bi}
\DeclareMathOperator{\dd}{d}
\DeclareMathOperator{\Var}{Var}

\newcommand{\pr}[1]{{\bf(pr-#1)}}
\newcommand{\PP}[1]{{\bf(\Pcal\Pcal-#1)}}
\newcommand{\str}[1]{{\bf(str-#1)}}
\newcommand{\strs}[1]{{\bf({str0-#1)}}}
\newcommand{\cruc}[1]{{\bf(cruc-#1)}}

\newcommand{\GPo}[0]{\ensuremath{G_{\Pcal}}}

\newcommand{\Rcnr}[0]{\ensuremath{R_{\text{cnr}}}}
\newcommand{\Rsde}[0]{\ensuremath{R_{\text{sde}}}}
\newcommand{\Rmdl}[0]{\ensuremath{R_{\text{mdl}}}}
\newcommand{\Wcnr}[0]{\ensuremath{W_{\text{cnr}}}}
\newcommand{\Wsde}[0]{\ensuremath{W_{\text{sde}}}}
\newcommand{\Wmdl}[0]{\ensuremath{W_{\text{mdl}}}}
\newcommand{\Zcnr}[0]{\ensuremath{Z_{\text{cnr}}}}
\newcommand{\Zsde}[0]{\ensuremath{Z_{\text{sde}}}}
\newcommand{\Zmdl}[0]{\ensuremath{Z_{\text{mdl}}}}
\newcommand{\mucnr}[0]{\ensuremath{\mu_{\text{cnr}}}}
\newcommand{\musde}[0]{\ensuremath{\mu_{\text{sde}}}}
\newcommand{\mumdl}[0]{\ensuremath{\mu_{\text{mdl}}}}

\newcommand{\chvatal}[0]{Chv{\'a}tal}

\newcommand{\side}[0]{q}

 \title{Maker-Breaker games on random geometric graphs}
  \author{Andrew Beveridge\thanks{Department of Mathematics, Statistics and
  Computer Science, Macalester College, Saint Paul, MN. E-mail:
  \texttt{abeverid@macalester.edu}}\,,
Andrzej Dudek\thanks{Department of Mathematics, Western Michigan University,
Kalamazoo, MI. E-mail: \texttt{andrzej.dudek@wmich.edu}. {\orange Research supported in part by Simons Foundation Grant \#244712}} \,,
Alan Frieze\thanks{Department of Mathematical Sciences, Carnegie Mellon
University,
Pittsburgh, PA. E-mail: \texttt{alan@random.math.cmu.edu}.
Supported in part by NSF Grant CCF1013110}\,,
Tobias M\"uller\thanks{Mathematical Institute, Utrecht University, Utrecht,
the
Netherlands. E-mail: \texttt{t.muller@uu.nl}.
Part of this work was done while this author was supported by a VENI grant
from Netherlands Organisation for Scientific Research (NWO)}\,,
Milo\v{s} Stojakovi\'{c}\thanks{Department of Mathematics and Informatics, University of Novi Sad,
Serbia. Email: \texttt{milos.stojakovic@dmi.uns.ac.rs}.
Partly supported by Ministry of Science and Technological Development,
Republic of Serbia, and Provincial Secretariat for Science, Province of
Vojvodina.}}

\begin{document}

  \maketitle

\begin{abstract}
In a Maker-Breaker game on a graph $G$, Breaker and Maker alternately claim
edges of $G$. Maker wins if, after all edges have been claimed, the graph
induced by his edges has some desired property. We consider four 
Maker-Breaker games played on random geometric graphs.
For each of our four games we show that if we add edges between $n$ points chosen 
uniformly at random in the unit square by order of increasing edge-length
then, with probability tending to one as $n\to\infty$, the
graph becomes Maker-win the very moment it satisfies a simple
necessary condition.
In particular, with high probability,
Maker wins the connectivity game as soon as the minimum degree is at least two; Maker wins the Hamilton cycle 
game as soon as the minimum degree is at least four; Maker wins the perfect matching game as soon as the
minimum degree is at least two and every edge has at least
three neighbouring vertices; and Maker wins the $H$-game as soon as there is a subgraph from a finite list of
``minimal graphs".
These results also allow us to give precise expressions for the limiting probability that $G(n,r)$ is Maker-win
in each case, where
$G(n,r)$ is the graph on $n$ points chosen uniformly at random on the unit square with an 
edge between two points if and only if their distance is at most $r$.
\end{abstract}


\section{Introduction}

Let $H = (X, \Fcal)$ be a hypergraph. That is, $X$ is a finite set and
$\Fcal \subseteq 2^X$ is
a collection of subsets of $X$.
The {\em Maker-Breaker game} on $H$ is played as follows.
There are two players, Maker and Breaker, that take turns claiming {\grn unclaimed} elements
of $X$, with Breaker moving first.
Maker wins if, after all the elements of $X$ have been claimed, the set
$M \subseteq X$ of elements
claimed by him contains an element of $\Fcal$ (i.e.~if $F \subseteq M$
for some $F \in \Fcal$).
Otherwise Breaker wins.

Maker-Breaker games have considerable history going back to \chvatal\ and
Erd\H{o}s~\cite{ChvatalErdos78}
and have become increasingly popular over the past decade or so.
Often the case is considered where $X = E(G)$ is the edge set of
some graph $G$ and $\Fcal \subseteq 2^{E(G)}$ is
a collection of graph theoretic structures of interest, such as the set
of all spanning trees,
the set of all perfect matchings, the set of all Hamilton cycles, or the set
of all subgraphs isomorphic to
a given graph $H$.
In these cases we speak of, respectively, the connectivity game,
the perfect matching game, the Hamilton cycle game, the $H$-game.

{\grn
Already in~\cite{ChvatalErdos78}, an interesting connection between the positional games on the complete graph and the corresponding properties of the random graph was pointed out, roughly noting that the course of a game between the two (smart) players often resembles a purely random process. This many time repeated and enriched paradigm in positional game theory later came to be known as 
\emph{Erd\H{o}s's probabilistic intuition}, and although still lacking the form of a precise statement, it has proved a valuable source of inspiration.
See~\cite{BeckBook} for an overview of the theory of positional games.

Combining two related concepts, Maker-Breaker games on the Erd\H{o}s-R\'enyi model of random graph were first introduced and studied in~\cite{StojakovicSzabo05}. Several works followed, including~\cite{HKSS09, BFHK11}, resulting in precise descriptions of the limiting probabilities for Maker-win in games of connectivity, perfect matching and Hamilton \gnote{Replaced "Hamiltonicity game" w. "Hamilton cycle game" everywhere.} cycle, all played on the edges of a random graph $G\sim G(n,p)$. As for the $H$-game, not that much is known. The limiting probability is known precisely only for several special classes of graphs~\cite{MuSt}. Recently, the order of the threshold probability was determined for all graphs $H$ that are not trees, and whose maximum 2-density is not determined by a $K_3$ as a 
subgraph~\cite{NeStSt}. When it comes to other models of random graphs, some positional games on random regular graphs were studied 
in~\cite{BKS11}.
}

In the current paper we add to this line of research, by considering
Maker-Breaker games played on the random geometric graph on the unit square.
Given points $x_1,\dots,x_n \in \eR^2$ and $r>0$ we define
the geometric graph $G(x_1,\dots,x_n;r)$ as follows.
It has vertex set $V = \{ x_1,\dots,x_n\}$, and an edge
$x_ix_j$ if and only if $\norm{x_i-x_j} \leq r$, {\orange where $\norm{.}$ is the Euclidean norm.}
The random geometric graph $G(n,r) = G(X_1,\dots,X_n;r)$
is defined by taking $X_1,\dots,X_n$ i.i.d.~uniform at random
on the unit square $[0,1]^2$.
The random geometric graph essentially goes back to Gilbert~\cite{Gilbert61} who defined a very similar model in 
1961. For this reason the random geometric graph is sometimes also called the {\em Gilbert model}.
Random geometric graphs have been the subject of a considerable research effort in the last two decades or so. 
As a result, detailed information is
now known on aspects such as ($k$-)connectivity~\cite{PenroseMST1, Penrosekconn}, the largest component~\cite{PenroseBook},
the chromatic number and clique number~\cite{twopoint, McDiarmidMuller11} and the simple random walks on the graph~\cite{CooperFriezeCoverRgg}.
A good overview of the results prior to 2003 can be found in the monograph~\cite{PenroseBook}.
It is of course possible to define the random geometric graph in dimensions other than two, using a probability measure other
than the uniform and using a metric other than the euclidean norm, and in fact several authors do this.
In the present work we have decided to stick with the two-dimensional, uniform, euclidean setting because this
is the most natural choice in our view, and this setting is already challenging enough.

Recall that, formally, a graph property $\Pcal$ is a collection of graphs
that is closed under isomorphism.
We call a graph property $\Pcal$  {\em increasing} if it is preserved
under addition of edges (i.e.~$G \in \Pcal$ implies $G\cup e \in \Pcal$
for all $e \in {V(G)\choose 2}$).
Examples of increasing properties are being connected, being non-planar,
containing a Hamilton cycle,
or being Maker-win in any of the games mentioned above.
We define the {\em hitting radius} of an increasing property $\Pcal$ as:
\[
 \rho_n(\Pcal) := \inf\{ r\geq 0 : G(X_1,\dots,X_n;r) \text{ satisfies }
\Pcal \}.
\]
Here we keep the locations of the points $X_1,\dots,X_n$ fixed as we take the
infimum. 

{\blue
We give explicit descriptions of the hitting radius for three different games, namely the connectivity game, the Hamilton cycle game and the perfect matching game. For each game, we have a very satisfying  characterization: the hitting radius for $G(n,r)$ to be Maker-win coincides exactly with a simple, necessary minimum degree condition.   Each characterization engenders an extremely precise description of the behavior at the threshold value for the radius. We note that these results require some very technical lemmas that explicitly catalog the structure of the graph around its sparsest regions. Finally, we also state a general  theorem for the $H$-game for a fixed graph $H$. The hitting radius obeys a similar behavior, and  can be determined by finding the smallest $k$ for which the $H$-game is Maker-win on the $k$-clique. {\grn Upper bounds for such $k$ are available, see e.g.~\cite{Beck1}, so determining $k$ for a given $H$ is essentially a finite problem.}

We state these four results below, as couplets consisting of a theorem recognizing the coincidence of the hitting radii,  followed by a corollary that describes the behavior  around that critical radius.  We say that event $A_n$ holds  \emph{whp} (short for \emph{with high probability}) to mean that $\Pee(A_n) = 1 - o(1)$ as $n \rightarrow \infty$.
}

\begin{theorem}\label{thm:connhit}
The random geometric graph process satisfies \onote{Added ``emph'' to \emph{whp} in the whole paper}
\[
 \rho_n(\text{Maker wins the connectivity game}) =
\rho_n(\text{minimum degree $\geq 2$}) \quad \emph{whp}.
\]
\end{theorem}

Theorem~\ref{thm:connhit} allows us to derive an expression for the limiting
probability
that Maker wins the connectivity game on $G(n,r)$, as follows.
Let $n \geq 2$ be an integer and $r> 0$ be arbitrary.
Since having minimum degree at least two is a necessary condition
for Maker winning the connectivity game, we have:

\begin{equation}\label{eq:connprobconn1}
\begin{array}{rcl}
\Pee( \delta(G(n,r)) \geq 2 )
& \geq &
\Pee( \text{Maker wins on } G(n,r) )\\
& = &
\Pee( \text{Maker wins on } G(n,r) \text{ and } \delta(G(n,r)) \geq 2 ) \\
& = &
\Pee( \delta(G(n,r)) \geq 2 ) \\
& &
- \Pee( \text{Maker loses on } G(n,r) \text{ and } \delta(G(n,r)) \geq 2 ) \\
& \geq &
\Pee( \delta(G(n,r)) \geq 2 ) \\
& & -
\Pee( \rho_n(\text{Maker wins}) \neq
\rho_n(\text{minimum degree $\geq 2$}) ).
\end{array}
\end{equation}

\noindent
Combining this with Theorem~\ref{thm:connhit} we see that, for any
sequence $(r_n)_n$ of positive numbers:

\begin{equation}\label{eq:connprobconn2}
\Pee( \text{Maker wins the connectivity game on } G(n,r_n) )
=
\Pee( \delta(G(n,r_n)) \geq 2 ) - o(1).
\end{equation}

\noindent
Combining this with a result of Penrose (repeated as
Theorem~\ref{thm:penrosemindeg} below), we
find that:

\begin{corollary}\label{cor:connprob}
Let $(r_n)_n$ be a sequence of positive numbers and let us write
\[
x_n := \pi n r_n^2 -\ln n -\ln\ln n.
\]
Then the following holds in the random geometric graph $G(n,r_n)$:
\[
\lim_{n\to\infty} \Pee(\text{Maker wins the connectivity game})
=
\left\{\begin{array}{cl}
        1 & \text{ if } x_n \to +\infty, \\
e^{-(e^{-x}+\sqrt{\pi e^{-x}})} & \text{ if } x_n \to x \in\eR, \\
0 & \text{ if } x_n \to -\infty.
       \end{array}
\right.
\]
\end{corollary}

Let us define $N(v)$ to be the set of neighbors of vertex $v$, and  the
edge-degree of an edge $e=uv \in E(G)$ as
$d(e) = |(N(v)\cup N(u))\setminus \{u,v\}|$.

\begin{theorem}\label{thm:perfecthit}
The random geometric graph process satisfies, for
$n$ even: \onote{Added an extra space after min. and applied this to the whole paper.}
\[
 \rho_n(\text{Maker wins the perfect matching game}) =
\rho_n(\text{{\orange min.\,deg.} $\geq 2$ and {\orange min.\,edge-deg.} $\geq 3$}) \quad \emph{whp.}
\]
\end{theorem}

{\grn A similar argument to the one we used for the connectivity game yields the following.}
\begin{corollary}\label{cor:perfecthitprob}
Let $(r_n)_n$ be a sequence of positive numbers and
let us write
\[
x_n := \pi n r_{\grn n}^2 -\ln n -\ln\ln n.
\]
Then the following holds in the random geometric graph $G(n,r_n)$:
\[
\lim_{n\to\infty, \atop n \text{ even}} \Pee(\text{Maker wins the
perfect matching game})
=
\left\{\begin{array}{cl}
        1 & \text{ if } x_n \to +\infty, \\
e^{-((1+\pi^2/8)e^{-x} + \sqrt{\pi}(1+\pi) e^{-x/2})} &
\text{ if } x_n \to x \in\eR, \\
0 & \text{ if } x_n \to -\infty.
       \end{array}
\right.
\]
\end{corollary}

\begin{theorem}\label{thm:hamhit}
The random geometric graph process satisfies
\[
 \rho_n(\text{Maker wins the Hamilton cycle game}) =
\rho_n(\text{minimum degree $\geq 4$}) \quad \emph{whp.}
\]
\end{theorem}

{\grn Again, we can make use of this result to obtain the precise relation between the radius $r$ and the probability of Maker-win in the Hamilton cycle game.}

\mnote{Updated final constant: $2 \ln ((k-1)! )= 2 \ln (6)$.}
\begin{corollary}\label{cor:hamhitprob}
Let $(r_n)_n$ be a sequence of positive numbers and
let us write
\[
x_n := \frac12\left(\pi n r_{\grn n}^2 - (\ln n+ 5\ln\ln n - 2\ln(6))\right).
\]
Then the following holds in the random geometric graph $G(n,r_n)$:
\[
\lim_{n\to\infty} \Pee(\text{Maker wins the Hamilton cycle game})
=
\left\{\begin{array}{cl}
        1 & \text{ if } x_n \to +\infty, \\
e^{-e^{-x}} & \text{ if } x_n \to x \in\eR, \\
0 & \text{ if } x_n \to -\infty.
       \end{array}
\right.
\]
\end{corollary}

Our next result, on the $H$ game, is a bit different from our results for the other games.
In particular, the ``action" now takes place long before the connectivity threshold, when the average
degree $\pi n r^2$ decays polynomially in $n$.

\begin{theorem}\label{thm:Hgame}
Let $H$ be any fixed graph and let $k_H$ denote the
least $k$ for which the $H$-game is Makers
win on a $k$-clique, and let $\Fcal_H$ denote the
family of all graphs on $k_H$ vertices for which
the game is Maker-win.
Then
\[
\rho_n(\text{Maker wins the $H$-game}) =
\rho_n(\text{the graph contains a subgraph $\in\Fcal_H$})
\quad \emph{whp.}
\]
\end{theorem}

Again the theorem on the hitting radius allows us to determine the hitting probability.
This time we make use of a theorem of Penrose~\cite{PenroseBook} (stated as Theorem~\ref{thm:smallsubgraphs} below)
on the appearance of small subgraphs to obtain:

\begin{corollary}\label{cor:Hprob}
Let $H$ be any fixed graph, and let $k_H$ denote the smallest
$k$ for which the $H$-game is Maker-win on a $k$-clique.
If $r = c \cdot n^{-1/2+1/2{\orange (k-1)}}$ with $c\in\eR$ fixed then \onote{Replaced $k$ by $k-1$}
\[
\Pee(\text{Maker wins the $H$-game}) \mapsto f(c),
\]
\noindent
where $0<f(c)<1$ (is an expression which can be computed explicitly in
principle and)
satisfies $f(c)\to 0$ as $c\to-\infty$; and $f(c)\to 1$ as $c\to\infty$.
\end{corollary}

\subsection{Overview}

\mnote{It looks to me like the $H$-game is a different (and simpler) argument than the rest. Does it make sense to prove this first to avoid reviewer fatigue?}

{\blue
The proofs of the main theorems leverage a characterization of the local structure of the graph $G(n,r)$. In particular, we meticulously describe the graph around its sparser regions. Section \ref{sec:prelim} contains our background results. We start with a short list of known results about Maker-Breaker games, followed by a selection of results about random geometric graphs, drawn or adapted from \cite{PenroseBook}. Next, we give some geometric preliminaries, most of which provide approximations of small areas defined by intersecting disks in $[0,1]^2$.

Section \ref{sec:structure} contains the technical lemmas that chart a detailed cartography of $G(n,r)$. We  \emph{dissect} the unit square $[0,1]^2$ into \emph{cells}, which are squares of side length roughly $\eta r$ where $\eta >0$ is very small. We fix a large constant $T >0$ (made explicit later). We say that a cell $c$ is \emph{good} if it is dense, meaning that there are at least $T$ vertices in $c$. The sparse cells are $\emph{bad}$. We show that bad cells come in clusters of very small diameter, which we call \emph{obstructions}. Moreover, these obstructions are well-separated from one another. These results are collected in the \emph {Dissection Lemma \ref{lem:structure}}. Next, we prove the \emph{Obstruction Lemma \ref{lem:crucial}}, which shows that for each obstruction, there are enough points in nearby good cells to allow Maker to overcome these bottlenecks.

Section \ref{sec:connect} contains the straight-forward proof for the  connectivity game. A classic result \cite{Lehman64} states that the connectivity game on a graph $G$ is Maker-win if and only if $G$ has two disjoint spanning trees. Our characterization of obstructions quickly reveal that such a pair of trees exist.

The argument for the Hamilton cycle game, found in Section \ref{sec:hamilton},  is far more delicate. In order to win the game on $G(n,r_n)$, Maker plays lots of local games. The games that are played  in and around the obstructions require the most judicious play: this is where Maker must directly play against Breaker to keep his desired structure viable at a local level. Ultimately, he is able to construct long paths that span the obstructions and have endpoints in nearby good cells. Meanwhile, Maker plays a different kind of game in each good cell. Therein, he creates a family of flexible \emph{blob cycles}, which consist of cycles that also contain a fairly large clique. In addition, Maker  plays to claim half the edges between the vertices in nearby local games. Once the edge claiming is over, Maker stitches together his desired structure. The soup of blob cycles in each cell gives Maker the flexibility to connect together the local games. Each local game is absorbed into a nearby blob cycle. This process is repeated by merging the current blob cycles in a good cell into one blob cycle in the cell. Along the way, we also absorb other vertices that are not yet attached, and also connect the blob cycles in nearby good cells, following a pre-determined tree structure. {\grn The final result is a Hamilton cycle.}

Section \ref{sec:perfect} considers the perfect matching game. The argument and game play is similar to the Hamilton cycle game. The local games in and around the obstructions are played head-to-head with Breaker, creating a matching that saturates each obstruction (and uses some nearby vertices in good cells). Meanwhile, Maker creates a Hamilton cycle through the rest of the vertices (as in the previous game) and then takes every other edge to get the perfect matching.

In Section \ref{sec:H-game}, we handle the $H$-game with a straight-forward argument. Once we reach the threshold for the appearance of a clique on which Maker can win the $H$-game, we use a Poisson argument for independent copies of such a clique appearing in well-separated regions of the graph.

Finally, we conclude in Section \ref{sec:conclusion} and list some directions for future research.

}

\section{Preliminaries}
\label{sec:prelim}

\mnote{Re-organized the prelims, adding two subsections.}

In this section, we present some preliminary results that will be useful in the sequel. We start with a modest collection of previous results on Maker-Breaker games.
The Maker-Breaker connectivity game is also known as the {\em Shannon switching game}.
A classical result by Lehman~\cite{Lehman64} states:

\begin{theorem}[\cite{Lehman64}] \label{t:lehman}
The connectivity game played on $G$ is Maker-win
if and only if $G$ admits two disjoint spanning trees.
\end{theorem}

This has the immediate corollary:

\begin{corollary}\label{cor:connectivityKn}
The connectivity game is Maker-win on $K_n$ if and only if
$n\geq 4$.
\end{corollary}

The following is a standard result.

\begin{theorem}[\cite{Beck1}] \label{thm:Hgamedet}
Let $H$ be a finite graph. There is a $N = N(H)$ such that
Maker can win the $H$-game on $K_s$ for all $s\geq N$.
\end{theorem}

It turns out that the Hamilton cycle
game, as well as several other standard games on graphs, are easy wins for
Maker when played on a sufficiently large complete graph.
To make the game more balanced, \chvatal\ and Erd\H os~\cite{ChvatalErdos78}
introduced {\em biased games}.
In the $(1:b)$ biased game, Maker claims a single edge in each move, as before,
but Breaker claims $b$ edges. The parameter $b$ is called the bias (towards
Breaker). {\grn Due to the ``bias monotonicity" of Maker-Breaker games,} it is straightforward to conclude that for any positional game there is some
value $b(n)$ such that Maker wins the game for all $b < b(n)$, while Breaker
wins for $b \geq b(n)$. We call $b(n)$ the {\em {\grn threshold} bias} for that game.
Later on,  we  make use of a recent break-through result of
Krivelevich on the {\grn threshold} bias of the Hamilton cycle game on $K_n$.

\begin{theorem}[Krivelevich~\cite{Krivelevich10}]\label{thm:kriv}
The {\grn threshold} bias of the Hamilton cycle game on $K_n$ satisfies
$b(n) = (1+o(1)) n/\ln n$.
\end{theorem}

\subsection{Probabilistic preliminaries}

Throughout this paper, $\Po(\lambda)$ will denote the Poisson distribution
with parameter $\lambda$, and $\Bi(n,p)$ will denote the binomial
distribution with parameters $n, p$.
Recall that the $\Bi(1,p)$-distribution is also called
the {\em Bernoulli} distribution.
We will make use of the following incarnation of the Chernoff bounds.
A proof can for instance be found in Chapter 1 of~\cite{PenroseBook}.

\begin{lemma}\label{lem:chernoff}
Let $Z$ be either Poisson or Binomially distributed, and write
$\mu := \Ee Z$.
\begin{enumerate}
 \item For all $k \geq \mu$ we have
\[
\Pee( Z \geq k ) \leq e^{-\mu H(k/\mu)},
\]
\item For all $k \leq \mu$ we have
\[
\Pee( Z \leq k ) \leq e^{-\mu H(k/\mu)},
\]
\end{enumerate}
where $H(x) := x\ln x - x + 1$.\noproof
\end{lemma}

An easy inequality on the Poisson distribution that we will use below  is as follows.
For completeness we spell out the short proof.

\begin{lemma}\label{lem:Poissontriviality}
Let $Z$ be a Poisson random variable.
Then $\Pee( Z \geq k ) \leq (\Ee Z)^k$, for all $k \in \eN$ 
\end{lemma}

\begin{proof}
Using Markov's inequality we have

\[ \Pee( Z \geq k ) 
= \Pee( Z(Z-1)\cdots(Z-k+1) \geq 1 ) \leq \Ee{\big[}Z(Z-1)\dots(Z-k+1){\big]} = (\Ee Z)^k, \]

\noindent
where we also used the well-known, elementary fact that the $k$-th factorial moment of the Poisson equals the
$k$-th power of its first moment.
\end{proof}

The {\em total variational distance} between two integer-valued random variables
$X,Y$ is defined as:

\[
\dtv( X, Y ) = \sup_{A\subseteq\Zed} |\Pee( X \in A ) - \Pee( Y \in A )|.
\]

Let $g : \eR^d \to [0,\infty)$ be a bounded, measurable function.
A Poisson process on $\eR^d$ with intensity $g$ is a random set of points
$\Pcal \subseteq \eR^d$ with the properties that:
\begin{enumerate}
\item[\PP{1}] For every measurable $A \subseteq \eR^d$, we have
$\Pcal(A) \isd \Po\left( \int_A g(x){\dd}x \right)$, where
$\Pcal(A) := |A \cap \Pcal|$ denotes the number of
points of $\Pcal$ that fall into $A$;
\item[\PP{2}] If $A_1,\dots, A_n$ are disjoint and measurable, then
$\Pcal(A_1), \dots, \Pcal(A_n)$ are independent.
\end{enumerate}

An important special case is when $\int_{\eR^d} g(x){\dd}x < \infty$.
In this case we can write $g = \lambda \cdot f$, with $f$
(the probability density of) an absolute continuous
probability measure.
In this case the {\em Poisson process} $\Pcal$ can also be generated as follows.
Let $X_1,X_2,\ldots \in \eR^d$ be an infinite supply of random
points, i.i.d.~distributed
with probability density $f$; and let $N \isd \Po(\lambda)$ be independent
of $X_1,X_2,\dots$.
Then $\Pcal \isd \{ X_1, \dots, X_N \}$.
A proof of this folklore result and more background on Poisson point processes
can for instance be found in~\cite{KingmanBook}.

It is often useful to consider a {\em Poissonized} version of the
random geometric graph.
Here we mean the following.
Let us take $X_1, X_2, \dots$ i.i.d.~uniform at random on the unit square,
and we let $N \isd \Po(n)$ be independent of $X_1, X_2,\dots$.
Following Penrose~\cite{PenroseBook}, we will write

\begin{equation}\label{eq:Pdef}
\Pcal_n := \{ X_1,\dots, X_N\}.
\end{equation}

\noindent
(Thus $\Pcal_n$ is a Poisson process with intensity $n$ on the unit square
and intensity 0 elsewhere.)
The {\em Poisson random geometric graph} is defined as

\[
\GPo(n,r) := G(\Pcal_n; r).
\]

\noindent
The properties~\PP{1} and \PP{2} above make $\GPo(n,r)$ often
slightly easier to deal
with than the ordinary random geometric graph.
For notational convenience (and again following
Penrose~\cite{PenroseBook}) we set:

\begin{equation}\label{eq:Xdef}
\Xcal_n := \{ X_1,\dots, X_n \}.
\end{equation}

The usual random geometric graph $G(n,r) = G(\Xcal_n;r)$ is sometimes
also called
the {\em binomial random geometric graph}.
By defining both $G(n,r)$ and $\GPo(n,r)$ on the same set of points
$X_1,X_2,\dots$ we get an explicit
coupling which often helps to transfer results from the Poissonized
setting to the original setting.

The next theorem is especially useful for dealing with the subgraph counts and
counts of other ``small substructures'' in the Poissonized random geometric
graph.
The statement and its proof are almost identical to
Theorem 1.6 of~\cite{PenroseBook},
but for completeness
we include a  proof in Appendix~\ref{sec:Palmapp}.

\begin{theorem}\label{thm:palm}
Let $\Pcal_n$ be as in~\eqref{eq:Pdef}, and let $h(a_1,\dots,a_k;A)$ be a
bounded
measurable function defined
on all tuples $(a_1,\dots,a_k;A)$ with
$A \subseteq \eR^2$ finite and $a_1,\dots,a_k \in A$. Let us write
\[
 Z := \sum_{a_1,\dots,a_k \in \Pcal_n, \atop a_1,\dots, a_k \text{ distinct }}
 h(a_1,\dots,a_k;\Pcal_n).
\]
Then
\[
\Ee Z
= n^k \cdot \Ee h( Y_1,\dots,Y_k; \{Y_1, \dots, Y_k\} \cup \Pcal_n ),
\]
where $Y_1,\dots, Y_k$ are i.i.d.~uniform on the unit square, and are
independent of
$\Pcal_n$.
\end{theorem}

The previous theorem can be used to count ``small substructures" in the Poisson setting.
We are primarily interested in the binomial random graph.
The next lemma is useful for transferring results from the Poisson random geometric graph to the
binomial random graph.
The proof is relatively standard.
For completeness we provide the proof in Appendix~\ref{sec:configurations}, since
it is not available in the literature as far as we are aware.

\begin{lemma}\label{lem:configurations}
Let $h_n(v_1,\dots,v_k;V)$ be a sequence of $\{0,1\}$-valued, measurable functions defined
on all
tuples $(v_1,\dots,v_k;V)$ with $V \subseteq \eR^2, v_1,\dots,v_k \in V$ and
set 
\[
Z_n := \sum_{v_1,\dots,v_k \in \Pcal_n} h_n(v_1,\dots,v_k;\Pcal_n), \quad
\Ztil_n := \sum_{v_1,\dots,v_k \in \Xcal_n} h_n(v_1,\dots,v_k;\Xcal_n),
\]
with $\Pcal_n$ as in~\eqref{eq:Pdef} and $\Xcal_n$ as in~\eqref{eq:Xdef}.
Suppose that $\Ee Z_n = O(1)$ and that there exists a sequence $(r_n)_n$ such that 
$\pi n r_n^2 = o( \sqrt{n} )$ and the value of $h_n(v_1,\dots,v_k;V)$ does not depend on
$V \setminus (B(v_1;r_n)\cup\dots\cup B(v_k;r_n))$.
Then $Z_n = \Ztil_n$ whp.
\end{lemma}

We will need two results of Penrose~\cite{PenroseBook} that were proved using
Poissonization.
The first of these two results is on the occurrence of small subgraphs. \onote{Removed an extra ``denote''}
For $H$ a graph, we shall denote by $N(H) = N(H;n,r)$ the number of
{\em induced} subgraphs of $G(n,r)$ that are
isomorphic to $H$.
For $H$ a connected geometric graph on $k$ vertices, let us denote

\begin{equation}\label{eq:muGdef}
\mu(H) := \frac{1}{k!} \int_{\eR^2} \dots \int_{\eR^2}
1_{\{G(0,x_1,\dots,x_{k-1};1) \cong H\}} {\dd}x_1\dots{\dd}x_{k-1}
\end{equation}

\noindent
(Here $G\cong H$ means that $G$ and $H$ are isomorphic and $1_A$ is
the indicator function of the set $A$, in our case
the set of all $(x_1,\dots, x_{k-1})\in \left(\eR^2\right)^{k-1}$ that satisfy
$G(0,x_1,\dots,x_{k-1};1) \cong {\grn H}$.)
It can be seen that, since $H$ is a connected geometric graph,
$0 < \mu(H) < \infty$.
The following is a restriction of Corollary 3.6 in~\cite{PenroseBook}
to the special case
of the uniform distribution on the unit square and the Euclidean norm.

\begin{theorem}[Penrose~\cite{PenroseBook}]\label{thm:smallsubgraphs}
For $k \in \eN$, let $H_1,\dots, H_m$ be connected, non-isomorphic
geometric graphs on $k \geq 2$ vertices.
Let $(r_n)_n$ be a sequence of positive numbers satisfying
$r_n = \alpha \cdot n^{-\frac{k}{2(k-1)}}$ for some constant $\alpha > 0$.
Then
\[
 (N(H_1),\dots,N(H_m)) \convd (Z_1,\dots,Z_m),
\]
where $Z_1,\dots,Z_m$ are independent Poisson random variables with means
$\Ee Z_i = \alpha^{2(k-1)} \cdot \mu(H_i)$.
\end{theorem}

\noindent
We shall also need the a result on the minimum degree of the
random geometric graphs.
The following is a reformulation of Theorem 8.4 in~\cite{PenroseBook},
restricted to
the case of the Euclidean metric in two dimensions.

\begin{theorem}[Penrose~\cite{PenroseBook}]\label{thm:penrosemindeg}
Let $(r_n)_n$ be a sequence of positive numbers.
The following hold for the random geometric graph $G(n,r_n)$:
\begin{enumerate}
\item If $\pi n r_n^2 = \ln n + x + o(1)$ for some fixed $x\in\eR$ then
\[
\lim_{n\to\infty} \Pee[G(n,r_n)\text{ has min.\,deg.} \geq 1]
=
e^{-e^{-x}}.
\]
\item If $\pi n r_n^2 = \ln n + \ln\ln n + x + o(1)$ for some fixed
$x \in\eR$ then
\[
\lim_{n\to\infty} \Pee[G(n,r_n)\text{ has min.\,deg.} \geq 2]
=
e^{-(e^{-x}+\sqrt{\pi e^{-x}})}.
\]
\item If $\pi n r_n^2 = \ln n + (2k-3)\ln\ln n + 2\ln((k-1)!) + 2x + o(1)$
for some fixed $x\in\eR$ and $k > 2$ then
\[
\lim_{n\to\infty} \Pee[G(n,r_n)\text{ has min.\,deg.} \geq k]
= e^{-e^{-x}}.
\]
\end{enumerate}
\end{theorem}

Later on we will do some reverse-engineering of Theorem~\ref{thm:penrosemindeg}.
For this purpose it is convenient
to state also the following intermediate result that was part of the proof of
Theorem~\ref{thm:penrosemindeg}.

\begin{lemma}[\cite{PenroseBook}]\label{lem:penrosemindeg2}
Let $(r_n)_n$ be such that $\pi n r_n^2 = \ln n + \ln\ln n + x
+ o(1)$ for some $x \in \eR$, and let
$W_n$ denote the number of vertices of degree exactly one in $G(n, r_n)$. Then
\[
E W_n \to e^{-x} + \sqrt{ \pi e^{-x} }.
\]
\end{lemma}

\subsection{Geometric preliminaries}

This section begins with two elementary results about geometric graphs. The remainder of the section is devoted to  approximating  the area of intersecting regions in  $[0,1]^2$.
We start with a standard elementary result. Because we are not aware of a proof anywhere in the literature, we provide a  proof in Appendix~\ref{sec:spanapp}.

\begin{lemma}\label{lem:span}
Let $G$ be a connected geometric graph. Then $G$ has
a spanning tree of maximum degree at most five.
\end{lemma}

\noindent
Note  that Lemma~\ref{lem:span} is best possible since $K_{1,5}$
is a connected
geometric graph. We also need the following observation.
We leave the straightforward proof to the reader.

\begin{lemma}\label{lem:cross}
Let $G = (x_1,\dots,x_n;r)$ be a geometric graph
and suppose that $x_ix_j, x_ax_b \in E(G)$ are two edges that do not
share endpoints.
If the line segments $[x_i,x_j]$ and $[x_a,x_b]$ cross then
at least one of the edges $x_ix_a, x_ix_b, x_jx_a, x_jx_b$ is
also in $E(G)$. \noproof
\end{lemma}

We now turn to approximating areas in $[0,1]^2$. First, we give an expression
for the area of the difference between two disks of the same radius.
We leave the proof, which is straightforward trigonometry, to the reader.

\begin{lemma}\label{lem:areadiff}
For $x,y \in \eR^2$ we have, provided $d := \norm{x-y} \leq 2r$:
\[
\area( B(x;r) \setminus B(y;r) ) = \pi r^2 - 2 r^2 \arccos( d/2r ) +
dr\sqrt{1-(d/2r)^2}.
\]
\hfill\noproof
\end{lemma}


\noindent
Using the fact that $\frac{\pi}{2}(1 - x) \leq \arccos(x) \leq
\frac{\pi}{2} - x$ for
$0 \leq x \leq 1$, and
the Taylor approximations $\arccos(x) = \frac{\pi}{2} - x + O(x^2)$ and
$\sqrt{1-x} = 1 - x/2 + O(x^2)$,
we get the following straightforward consequence of Lemma~\ref{lem:areadiff}
that will be useful to us later:

\begin{corollary}\label{cor:areadiff2}
For $x,y \in \eR^2$ we have, provided $d := \norm{x-y} \leq 2r$:
\begin{equation}\label{eq:areadiff21}
dr \leq \area( B(x;r) \setminus B(y;r) ) \leq  4dr,
\end{equation}
and
\begin{equation}\label{eq:areadiff22}
\area( B(x;r) \setminus B(y;r) ) = 2dr - O(d^2),
\end{equation}
as $d\downarrow 0$. \noproof
\end{corollary}

\noindent
We also need expressions for the area of the intersection of a
disk of
radius $r$ with the unit square. Again the proof is straightforward
trigonometry that we
leave to the interested reader.

\begin{lemma}\label{lem:areaIntSq}
Suppose that $r < \frac12$ and $x\in [0,r)\times(r,1-r)$.
Writing $h$ for the first coordinate of $x$, we have:

\[
\area( B(x;r) \cap [0,1]^2 )
= \pi r^2 - \arccos(h/r) \cdot r^2 + hr\sqrt{1-(h/r)^2}.
\]
\hfill\noproof
\end{lemma}

\noindent
Again using $\frac{\pi}{2}(1 - x) \leq \arccos(x) \leq \frac{\pi}{2} - x$ for
$0 \leq x \leq 1$, and
the Taylor approximations $\arccos(x) = \frac{\pi}{2} - x + O(x^2)$ and
$\sqrt{1-x} = 1 - x/2 + O(x^2)$,
we get:

\begin{corollary}\label{cor:areaIntSq}
Suppose that $r < \frac12$ and $x\in[0,r)\times(r,1-r)$.
Writing $h$ for the first coordinate of $x$, we have:
\begin{equation}\label{eq:areaIntSq1}
\frac{\pi}{2}r^2 + hr \leq \area( B(x;r) \cap [0,1]^2 ) \leq  \frac{\pi}{2}
r^2 + 2hr,
\end{equation}%
for $0\leq h \leq r$, and
\begin{equation}\label{eq:areaIntSq2}
\area( B(x;r) \cap [0,1]^2 ) = \frac{\pi}{2} r^2 + 2hr - O(h^2),
\end{equation}%
as $h\downarrow 0$. \noproof
\end{corollary}

We need one more geometric approximation, that combines the bounds
from~\eqref{eq:areadiff22} and~\eqref{eq:areaIntSq2}.

\begin{lemma}\label{lem:areafinal}
Suppose that $x \in [0,r)\times(2r,1-2r), y \in B(x;r)$, with $y$ to the
right of $x$.
Let $h$ denote the first coordinate of $x$, and let $\alpha, d$ be defined by
$v := y-x = (d\cos\alpha, d\sin\alpha)$ (see Figure~\ref{fig:dhalpha}).
Then
\begin{equation}\label{eq:areafinal1}
\area([0,1]^2\cap(B(x;r)\cup B(y;r))) = \frac{\pi}{2} r^2 + 2hr +
(1+\cos\alpha)dr + O( (d+h)^2),
\end{equation}
and
\begin{equation}\label{eq:areafinal2}
\area( [0,1]^2\cap (B(y;r) \setminus B(x;r) ) )
=
(1+\cos\alpha)dr + O(d(d+h))
\end{equation}
\noindent
where the error terms are uniform over all $-\pi/2 \leq \alpha \leq \pi/2$.
\end{lemma}

\begin{figure}[h!]
 \begin{center}
  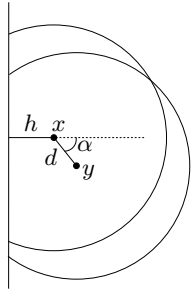
 \end{center}
\caption{Computing the area of $(B(x;r)\cup B(y;r)) \cap [0,1]^2$.
\label{fig:dhalpha}}
\end{figure}

\begin{proof}
Let $\ell_1$ denote the vertical line through $x$, and let $\ell_2$
denote the vertical line
through the midpoint of $[x,y]$.
Let $A_1$ denote the part of $B(x;r)\cup B(y;r)$ between the $y$-axis and
$\ell_1$;
let $A_2$ denote the part of $B(x;r)\cup B(y;r)$ between $\ell_1$ and
$\ell_2$; and let
$A_3$ denote the part of $B(x;r)\cup B(y;r)$ to the {\red right} of $\ell_2$.

By symmetry, and~\eqref{eq:areadiff22}, we have.
\[
\area(A_{\red 3}) = \frac12\area( B(x;r)\cup B(y;r) ) = \frac{\pi}{2}r^2 +
dr + O(d^2).
\]
Observe that $A_2$ is contained in a rectangle of sides
$d \cdot \frac12 \cos(\alpha)$ and
$2r + d$. Hence, $\area(A_2) \leq dr \cos(\alpha) + O(d^2)$.
On the other hand, $A_2$ contains the portion of $B(x;r)$ between $\ell_1$ and
$\ell_2$.
Thus, using~\eqref{eq:areaIntSq2} we see that in fact
\[
\area(A_2) = dr\cos\alpha + O(d^2).
\]
Similarly, $A_{\red 1}$ is contained inside a rectangle of sides $h$ and
$2r + d$; and
it contains the part of $B(x;r)$ between the $y$-axis and $\ell_1$, giving:
\[
 \area(A_{\red 1}) = 2hr + O( h(d+h) ).
\]
Combining the three expressions proves~\eqref{eq:areafinal1}.

Now notice that, if $B^+(x;r)$ denotes the portion of $B(x;r)$ to the right of
$\ell_1$:

\[ \begin{array}{rcl}
\area( [0,1]^2 \cap (B(y;r) \setminus B(x;r)) )
& \geq &
\area( (A_{\red 2}\cup A_{\red 3})\setminus B^+(x;r)) \\
& = & \area(A_{\red 2}\cup A_{\red 3}) - \frac{\pi}{2} r^2 \\
& = & (1+\cos\alpha)dr + O(d^2).
\end{array} \]%
On the other, hand it is not hard to see that the portion of
$B(y;r) \setminus B(x;r)$ that lies between the $y$-axis and $\ell_1$ has area
at most
$h \cdot d$, so that
\[ \begin{array}{rcl}
\area( [0,1]^2 \cap (B(y;r) \setminus B(x;r)) )
& \leq &
\area( (A_{\red 2}\cup A_{\red 3})\setminus B^+(x;r)) + h\cdot d \\
& = & (1+\cos\alpha)dr + O(d(d+h)),
\end{array} \]%
proving~\eqref{eq:areafinal2}.
\end{proof}

\section{The  structure of $G(n, r_n)$ near the connectivity threshold}
\label{sec:structure}

{\blue This section contains a number of observations of varying technical difficulty that
 describe the structure of the
random geometric graph  $G(n, r_n)$ when $r_n$ is such that the probability of
Maker-win
for one of the games under consideration is nontrivial. Intuitively speaking, we characterize regions of $G(n,r_n)$ as being dense or sparse. Most of the graph is dense. The sparse regions are of small diameter, and they are well-separated. Equally as important, we show that the dense region surrounding a sparse region contains enough points to enable Maker to overcome this local bottleneck. Let us begin.  }

Let $G=(V,r)$ be a geometric graph, where  $V = \{ x_1, \ldots , x_n \}
\subset [0,1]^2$. We consider the structure of $G$ with respect to a partition
of $[0,1]^2$ into small squares. We  introduce a vocabulary for describing the
density of vertices each the square. We also categorize the the vertices
themselves, depending on whether they are in dense squares or sparse squares.
In addition, we will pay special attention to vertices in dense squares that
are also close to vertices in sparse squares.

Let $\eta >0$ be arbitrarily small and let $m \in \eN$ {\blue be such that $\side(m) := 1/m  = \eta r$}. Let $\Dcal =
\Dcal(m)$ denote the \emph{dissection} of $[0,1]^2$ into
squares of side length $\side(m)$. %
We will call these squares {\em cells}.
For $K\in \eN$, we define a {\em $K\times K$ block of cells} in the obvious way,
see Figure~\ref{fig:D10}.
Given $T > 0$ and $V \subseteq [0,1]^2$, we call a cell $c\in\Dcal$
{\em good} with respect to
$T, V$ if $|c\cap V| \geq T$ and {\em bad} otherwise.
When the choice of $T$ and $V$ is clear from the context we will just speak
of good and bad.
\mnote{Moved the definition of $\Gamma$ up to here.}
Let $\Gamma = \Gamma(V, m, T, r)$ denote the  graph whose
vertices are the good cells of $\Dcal(m)$, with an
edge $cc' \in E(\Gamma)$ if and only if the lower left
corners of $c,c'$ have distance at most $r - \side \sqrt{2}$.
(Note that this way, any $x\in c$ and $y\in c'$ have distance $\norm{x-y} \leq
r$.)
We will usually just write $\Gamma$ when the choice of $V, m, T, r$ is clear
from the
context.
Let us denote the components of $\Gamma$ by $\Gamma_1, \Gamma_2, \dots$ where
$\Gamma_i$ has at least as many cells as $\Gamma_{i+1}$ (ties are broken
arbitrarily).
For convenience we will also write $\Gamma_{\max} = \Gamma_1$.
We will often be a bit sloppy and identify $\Gamma_i$ with the union of its
cells, and
speak of $\diam(\Gamma_i)$ and the distance between $\Gamma_i$ and $\Gamma_j$
and so forth.

\begin{figure}[h!]
 \begin{center}
  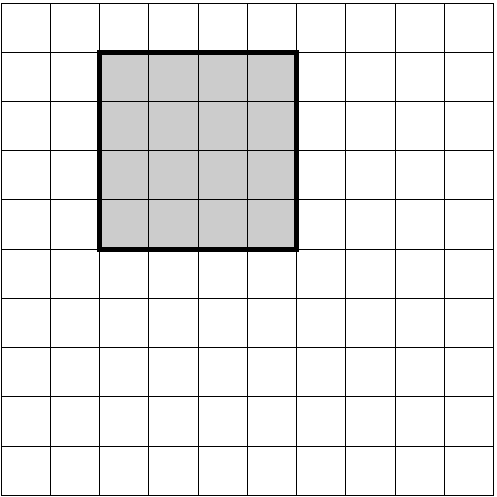
 \end{center}
\caption{The dissection $\Dcal(10)$, with a $4\times 4$ block of cells
highlighted.
\label{fig:D10}}
\end{figure}

Let us call a point $v\in V$ {\em safe} if there is a cell
$c\in\Gamma_{\max}$ such
that $| B(v;r) \cap V \cap c | \geq T$. (I.e.~in the geometric graph $G(V;r)$,
the point $v$ has at least $T$ neighbours inside $c$.)
If $v$ is not safe and there is a good cell $c \in \Gamma_i$, $i \geq 2$, such that
$| B(v;r) \cap V \cap c | \geq T$, we say that $v$ is {\em risky}.
Otherwise, if $v$ is neither safe nor risky, we call $v$ {\em dangerous}.
{\blue Every vertex in a cell of $\Gamma_{\max}$ is safe. Every vertex in a cell of $\Gamma_i$ for $i \geq 2$ is risky. Vertices in bad cells can be safe, risky or dangerous.}
\mnote{Changed semidangerous to `risky.' Risky is no longer a subset of safe.}

For $i \geq 2$ we let $\Gamma_i^+$ denote the set of all points of $V$ in
cells of
$\Gamma_i$, together with all risky points $v$ that satisfy $| B(v;r)
\cap V \cap c |
\geq T$
for at least one $c\in\Gamma_i$.

The following is a list of desirable properties that we would like $V$ and
$\Gamma(V,m,T,r)$ to have:

\begin{enumerate}
\item[\str{1}] $\Gamma_{\max}$ contains more than $0.99 \cdot |\Dcal|$ cells;
\item[\str{2}] $\diam(\Gamma_i^+)<r/100$ for all $i \geq 2$;
\item[\str{3}] If $u,v\in V$ are dangerous then
either $\norm{u-v} < r/100$ or $\norm{u-v} > r\cdot 10^{10}$;
\item[\str{4}] For all $i\ne j\geq 2$ the distance between $\Gamma_i^+$ and
$\Gamma_j^+$ is at least $r\cdot 10^{10}$;
\item[\str{5}] If $v\in V$ is dangerous and ${\red i} \geq 2$ then
the distance between $v$ and $\Gamma_{\red i}^+$ is at least $r \cdot 10^{10}$;
\item[\str{6}] If $c,c'\in\Gamma_{\max}$ are two cells at Euclidean distance
at most $10r$, then
there is a path in $\Gamma_{\max}$ between them of (graph-) length at most
$10^5$.
\end{enumerate}
See Figure \ref{fig:schematic} for a schematic of a geometric graph that satisfies~\str{1}--\str{6}.

\begin{figure}[h!]

\begin{center}
\begin{tikzpicture}[scale=.6]

\draw[white, pattern=dots] (-3,0) -- (11,0) -- (11,10) -- (-3,10) -- cycle;

\draw[fill=gray!20] (4.5,8.5) ellipse (.55 and .4);

\begin{scope}[shift={(8,8)}, rotate=100]
\draw[fill=gray!20] (0,0) circle (1 and 1.2);
\draw[fill=black] (0,0) circle (.5 and .6);
\end{scope}

\begin{scope}[shift={(7,3.5)}, rotate=-10]
\draw[fill=gray!20]  (0,0) ellipse (3 and 2);
\draw[fill=white]  (0,0) ellipse (2 and 4/3);
\draw[fill=gray!65]  (0,0) ellipse (2 and 4/3);
\draw[fill=white]  (0,0) ellipse (1 and 2/3);
\draw[pattern=vertical lines]  (0,0) ellipse (1 and 2/3);

\end{scope}

\begin{scope}[shift={(0,5)}, rotate=75]
\draw[fill=gray!20]  (0,0) ellipse (3.8 and 1.925);
\draw[fill=gray!20]  (0,0) ellipse (4.1 and 1.925);
\draw[fill=black]  (0,0) ellipse (3.3 and 1.375);
\draw[fill=white]  (0,0) ellipse (2.4 and 1);
\draw[fill=gray!65]  (0,0) ellipse (2.4 and 1);
\draw[fill=white]  (0,0) ellipse (1.2 and .5);
\draw[pattern=vertical lines]  (0,0) ellipse (1.2 and .5);
\end{scope}

\begin{scope}[shift={(12,0)}]

\draw[pattern=dots] (0, 9) -- (.5, 9) -- (.5,8.5) -- (0, 8.5) --cycle;
\node[right] at (.6,8.75) {good and safe};
\draw[fill=gray!20] (0, 8) -- (.5, 8) -- (.5,7.5) -- (0, 7.5) --cycle;
\node[right] at (.6,7.75) {bad and safe};
\draw[pattern=vertical lines] (0, 6) -- (.5, 6) -- (.5,5.5) -- (0, 5.5) --cycle;
\node[right] at (.6,5.75) {good and risky};
\draw[fill=gray!65] (0, 5) -- (.5, 5) -- (.5,4.5) -- (0, 4.5) --cycle;
\node[right] at (.6,4.75) {bad and risky};
\draw[fill=black] (0, 3) -- (.5, 3) -- (.5,2.5) -- (0, 2.5) --cycle;
\node[right] at (.6,2.75) {bad and dangerous};
\end{scope}

\node[fill=white] at (4,1) {$\Gamma_{\max}$};

\draw[-latex] (-4, 5) -- (0,5);
\node[left] at (-4,5) {$\Gamma_{1}$};

\draw[-latex] (11.5, 1) -- (7,3.5);
\node[right] at (11.5,1) {$\Gamma_{2}$};

\end{tikzpicture}

\end{center}

\caption{A schematic of part of a geometric graph that satisfies~\str{1}--\str{6}. Cells are characterized as good or bad. The smaller components $\Gamma_1, \Gamma_2$ of $\Gamma$ are surrounded by bad cells. Vertices are also characterized as safe, risky or dangerous.}
\label{fig:schematic}
\end{figure}
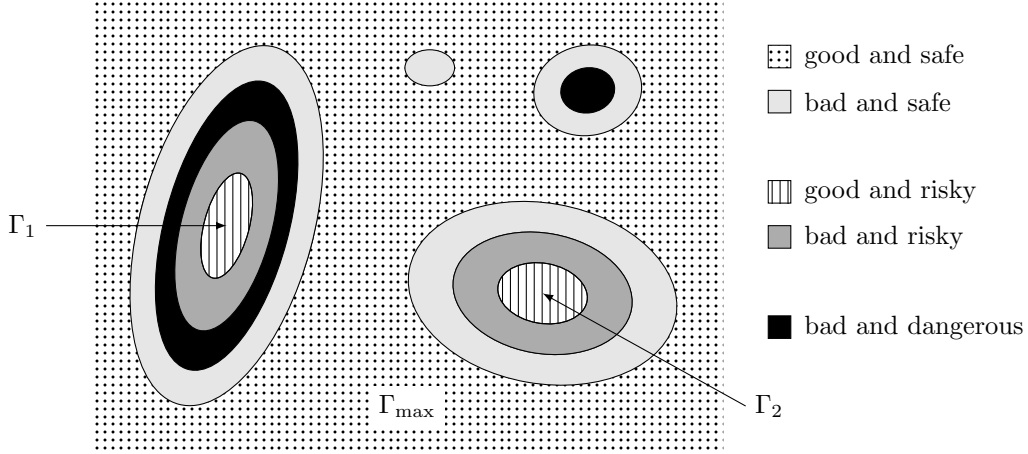

\mnote{Moved from p. 20}
{\blue Finally, we introduce some terminology for sets of dangerous and risky
points.}
Suppose that $V \subseteq [0,1]^2$ and $m, T, r$ are such that~\str{1}--\str{6}
above hold.
Dangerous points come in groups of points of diameter
$< r/100$ that are far apart.
We formally define a {\em dangerous cluster} (with respect to $V,m,T,r$)
to be an inclusion-wise maximal subset of
$V$ with the property that $\diam(A) < r \cdot 10^{10}$ and all elements of
$A$ are dangerous.

A set $A \subseteq V$ is an {\em obstruction} (with respect to $V,m,T,r$)
if it is either a dangerous cluster or $\Gamma_i^+$ for some $i\geq 2$. We
call $A$ an {\em $s$-obstruction} if $|A| = s$, and
we call it an $(\geq s)$-obstruction if $|A|\geq s$.
By \str{3}-\str{5}, obstructions are pairwise separated by distance
$r \cdot 10^{10}$. (One consequence: a vertex in a good cell is adjacent in
$G$ to at most one obstruction.)
A point $v\in V$ is {\em crucial} for $A$ if
\begin{enumerate}
\item[\cruc{1}] $A \subseteq B(v;r)$, and;
\item[\cruc{2}]  The vertex $v$ is safe: there is some cell $c \in \Gamma_{\max}$ such that
$|B(v;r)\cap c\cap V| \geq T$.
\end{enumerate}
We shall call the $T$ vertices in~\cruc{2} {\em important} for the crucial vertex
$v$ and  {\em important} for the obstruction $A$.
Note that this  crucial vertex  could be in the obstruction, or it could be a
nearby safe point.

\subsection{The Dissection Lemma}

For $n \in \eN$ and $\eta > 0$ a constant, let us define
\begin{equation}\label{eq:mdef}
 {\red m=m_n} := \left\lceil\sqrt{\frac{n}{\eta^2\ln n}}\right\rceil.
\end{equation}
The  goal of  this section is to prove that if
$r_n = \ln n + o (\ln n)$ then
 \str{1}--\str{6}  hold for $\Gamma(\Xcal_n,m_n,T,r_n)$ \emph{whp}. This is stated formally in
 the Dissection Lemma \ref{lem:structure} below. First,
we prove two intermediate lemmas.

\begin{lemma}\label{lem:cells}
Let $\eta, \eps, T, K > 0$ be arbitrary but fixed, where $\eta, \eps$ are
small and $T,K$ are large. Let
let $m$ be given by~\eqref{eq:mdef} and let $\Xcal_n$ be as in~\eqref{eq:Xdef}.
The following hold \emph{whp} for $\Dcal(m_n)$ with respect to $T$ and $\Xcal_n$:
\begin{enumerate}
\item\label{itm:cells.i} Out of every $K\times K$ block of cells, the area
of the bad cells inside the block is at most
$(1+\eps)\ln n/n$;
\item\label{itm:cells.ii} Out of every $K\times K$ block of cells
touching the boundary of the unit square,
the area of the bad cells inside the block is at
most $(1+\eps)\ln n / 2 n$.
\item\label{itm:cells.iii} Every $K\times K$ block of cells touching a corner
{\red contains only good cells}.
\end{enumerate}
\end{lemma}

\begin{proof}
The number of $K\times K$ blocks of cells is $(m-K+1)^2 = \Theta( n/\ln n)$.
The number of sets of cells of area $>(1+\eps)\ln n/n$ inside a given
$K\times K$ block is at most $2^{K^2} = O(1)$.
If $A$ is a union of cells inside some $K\times K$ block then the probability
that all its
cells are bad is at most

\[
\Pee( \Po( n \cdot \area(A) ) {\red <T K^2} )
\leq
\exp\left[ - n \cdot \area(A)
\cdot H\left(\frac{{\red TK^2}}{n\cdot\area(A)}\right) \right],
\]

\noindent
using Lemma~\ref{lem:chernoff}, where $H(x) = x\ln x - x + 1$.
If $\area(A) \geq (1+\eps) \ln n / n$ then $\frac{{\red TK^2}}{n\cdot\area(A)}
\leq \frac{{\red TK^2}}{(1{\red +}\eps)\ln n} = o(1)$ so that
$H\left(\frac{{\red T K^2}}{n\cdot\area(A)}\right) = 1+o(1)$ and hence

\[ \begin{array}{rcl}
\Pee( \Po( n \cdot \area(A) ) {\red <T K^2 })
& \leq &
\exp[ - (1+\eps)\ln n \cdot (1+o(1))  ] \\
& = &
n^{-(1+\eps)+o(1)}.
\end{array} \]

\noindent
Thus, the probability that there is a $K\times K$ block such that the area of
the
bad cells inside it
is at least $(1+\eps)\ln n / n$ is bounded above by:

\[
m^2 \cdot 2^{K^2} \cdot n^{-(1+\eps)+o(1)}
= n^{-\eps+o(1) } = o(1).
\]

\noindent
This proves part~\ref{itm:cells.i}.

The number of $K\times K$ blocks touching a side of the unit square is at most
$4m = \Theta( \sqrt{n/\ln n} )$,
and if $A$ is a union of cells with $\area(A) \geq (1+\eps) \ln n / 2n$,
consisting of no more than $K^2$
cells, then the probability that all its cells are bad is at most

\[ \begin{array}{rcl}
\Pee( \Po( n \cdot \area(A) ) {\red <T K^2} )
& \leq &
\exp[ - (1+\eps)\frac12 \ln n \cdot (1+o(1))  ] \\
& = &
n^{-(1+\eps)/2+o(1)}.
\end{array} \]

\noindent
Hence, the probability that there is a $K\times K$ block of cells touching a
side of the
unit square
such that the union of the bad cells inside the block has area at least
$(1+\eps) \ln n / 2n$ is at most:
\[
4m \cdot 2^{K^2} \cdot n^{-(1+\eps)/2 + o(1)} = n^{-\eps/2+o(1)} = o(1).
\]
This proves part~\ref{itm:cells.ii}.

Finally, there are only four $K\times K$ blocks that touch a corner.
The probability that a cell $c$ is bad is at most
\[
\Pee( \Po( n \cdot \area(c) ) {\red <T}) )
\leq
\exp[ - \Omega( \ln n ) \cdot (1+o(1))  ]
= o(1).
\]
Hence, the probability that there is a bad cell inside one of the
$K\times K$ blocks of cells touching a corner of the unit square
is at most:
\[
4 \cdot K^2 \cdot o(1) = o(1).
\]
This proves part~\ref{itm:cells.iii}.
\end{proof}

Recall that a {\em halfplane} is one of the two connected components
of $\eR^2 \setminus \ell$ with $\ell$ a line.
A {\em halfdisk} is the intersection of a disk $B(x,r)$ with a halfplane
whose defining line goes through the center $x$.
Also recall that a set $A \subseteq \eR^2$ is a {\em Boolean combination}
of the sets $A_1,\dots, A_n \subseteq \eR^2$ if
$A$ can be constructed from $A_1,\dots,A_n$ by means of any number of
compositions of the operations intersection, union and complement.
Our next lemma shows that the area of $A$ is well-approximated the area of the
cells that are entirely contained in $A$.

\begin{lemma}\label{lem:Sarea}
There exists a universal constant $C>0$ such that the following holds for
all $\eta > 0$ and sufficiently large $n \in\eN$.
Let $m_n$ be given by~\eqref{eq:mdef},
let $A \subseteq[0,1]^2$ be a Boolean combination of at most
$1000$ halfdisks with radii at most $1000 \cdot \sqrt{\ln n/n}$, and
let
\[
A' := \bigcup \{  c \in \Dcal(m_n) : c \subseteq A \},
\]
denote the union of all cells of $\Dcal(m_n)$ that are contained in $A$.
Then
\[
\area(A') \geq \area(A) - C \cdot \eta \cdot \left(\frac{\ln n}{n}\right),
\]
for $n$ sufficiently large.
\end{lemma}

\begin{proof}
Let $A, A'$ be as in the statement of the lemma. We use $\partial A$ to denote the boundary of $A$.
Let us define $\side := 1/m$ and
\[
A'' := \{ z \in \eR^2 : B(z; \side \sqrt{2}) \subseteq A \}
= A \setminus \left( B(0; \side \sqrt{2}) + \partial A\right).
\]
Then clearly $A'' \subseteq A' \subseteq A$.
If $A$ is a boolean combination of the halfdisks $A_1, \dots, A_k$ then clearly
$\partial A \subseteq \partial A_1 \cup \dots \cup \partial A_k$.
If $A_i$ is a halfdisk of radius $r \leq 1000 \cdot \sqrt{\ln n/n}$, then it is
easily seen that

\[\begin{array}{rcl}
\area\left( \partial A_i + B(0; \side \sqrt{2}) \right)
& \leq &
\frac12 \left( \pi(r+ \side \sqrt{2})^2 - \pi(r- \side \sqrt{2})^2 \right) +
2r \cdot 2 \side \sqrt{2} \\
& = &
r \cdot \side \dot (2 \pi \sqrt{2} + 4\sqrt{2}) \\
& \leq &
10^5 \cdot \eta \cdot \left(\frac{\ln n}{n}\right).
\end{array} \]
\mnote{Specific value for $r$ needed, to get above inequality?}
\noindent
This gives

\[\begin{array}{rcl}
 \area(A' )
& \geq &
\area(A'') \\
& \geq &
\area(A) - \sum_{i=1}^k \area( \partial A_i + B(0, \side \sqrt{2})  \\
& \geq &
\area(A) - 1000 \cdot 10^5 \cdot \eta \cdot \left(\frac{\ln n}{n}\right),
\end{array} \]
which proves the lemma with $C := 10^{8}$.
\end{proof}

We can now prove the Dissection Lemma: our random geometric graph satisfies
\str{1}--\str{6} with high probability.
{\red It will be convenient in the proof to introduce two slightly weaker
properties that will be proved
before their original counterpart:
\begin{enumerate}
\item[\strs{2}] $\diam(\Gamma_i)<r/100$ for all $i \geq 2$;
\item[\strs{4}] For all $i\ne j\geq 2$ the distance between $\Gamma_i$ and
$\Gamma_j$ is at least $r\cdot 10^{10}$;
\end{enumerate}
}
\begin{lemma}[Dissection Lemma]
\label{lem:structure}
Let $T > 0$ be arbitrary but fixed. For $\eta > 0$ sufficiently small, the
following holds.
Let ${\red m_n}$ be given by~\eqref{eq:mdef}, let $\Xcal_n$ be as
in~\eqref{eq:Xdef},
and let ${\red r_n}$ be such that $\pi n r_n^2 = \ln n + o(\ln n)$.
Then~\str{1}--\str{6} hold for $\Gamma(\Xcal_n,m_n,T,r_n)$ \emph{whp}.
\end{lemma}

\begin{proof}
We can assume that the conclusion of Lemma~\ref{lem:cells}
holds with $\eps := 10^{-5}$ and $K := \lceil (1+\eps) \cdot 10^{100} /
\eta^2 \rceil$.

\noindent
{\bf Proof of~\str{1}:}
Consider a $K\times K$ block of cells $\Bcal$. {\red By Lemma \ref{lem:cells},}
it contains at most
\[
N := {\red \frac{(1+\eps)\ln n}{ \side ^2n}=
\frac{1+\eps+o(1)}{\eta^2} }
\]
bad cells, since  $\side =(1+o(1))\sqrt{\eta^2 \ln n / n}$.)

Thus, at least $K - N > 0.99 K$ rows of the block do not contain any bad cell.
The cells of such a bad-cell-free row clearly belong to the same component of
$\Gamma$
(provided $\side  < r-\side \sqrt{2}$ which is certainly true for $\eta$
sufficiently small).
Since there is also at least one bad-cell-free column, we see that
all the bad-cell-free rows of the block belong to the same component of
$\Gamma$,
and this component
contains at least 99 percent of the cells in the block.
Let $\Ccal(\Bcal)$ denote the component of $\Gamma$ that contains more than
$0.99 K^2$
cells of
the block $\Bcal$.

Let us now consider two $K\times K$ blocks $\Bcal_1, \Bcal_2$, where
$\Bcal_2$ is
obtained by shifting $\Bcal_1$ to the left
by one cell. Then there are at least $K-2N > 0$ rows where both blocks don't
have any bad cells.
This shows that the component $\Ccal(\Bcal_1) = \Ccal(\Bcal_2)$.
Clearly the same thing is true if $B_2$ is obtained by shifting $B_1$ right,
down or up by one cell.

Now let $\Bcal_1, \Bcal_2$ be two arbitrary $K\times K$ blocks.
Since we can move from $\Bcal_1$ to $\Bcal_2$ by repeatedly shifting left,
right, down or up,
we see that in fact $\Ccal(\Bcal_1)=\Ccal(\Bcal_2)$ for {\em any} two
blocks $\Bcal_1, \Bcal_2$.
This proves that there is indeed a component of $\Gamma$ that contains
more than $0.99\cdot |\Dcal|$ cells.

{\bf Proof of~\strs{2}:}
Let $c$ be a cell that contains at least one point of $\Gamma_i$ with $i\geq 2$.
Let us first assume that $c$ is at least $K/2$ cells away from the
boundary of $[0,1]^2$.
In this case we can center a $K\times K$ block of cells $\Bcal$ on $c$
(If $K$ is odd we place $\Bcal$ so that $c$ is the middle cell, and if $K$ is
even we place $\Bcal$ so that a corner of $c$ {\red is} the center of $\Bcal$).
Reasoning as in the proof of~\str{1}, at least one row below,
one row above, at least one column to the left and at least one column to the
right of $c$
are bad-cell-free.
The cells in these bad-cell-free rows and columns must belong to $\Gamma_{\max}$
(by the proof of~\str{1}). Therefore  $\Gamma_i + B(0,r- \side \sqrt{2})$ is
completely
contained in the block $\Bcal$.
Let $p_L$ be a leftmost point of $\Gamma_i$, let $p_R$ be a rightmost point,
let $p_B$ be a lowest point and let $p_T$ be a highest point of $\Gamma_i$.
(These points need not be unique, but this does not pose a problem for the
remainder of the proof.)
Let $D_L$ denote the halfdisk

\[ D_L := B(p_L; r-2 \side \sqrt{2}) \cap \{ z\in\eR^2 : z_x < (p_L)_x \}. \]

\noindent
Then $D_L$ cannot contain any good cell, because that would contradict that
$p_L$
is the leftmost
point of $\Gamma_i$.
Similarly we define the halfdisks $D_R, D_B, D_T$ and observe that they each
cannot contain any good cell.
(see Figure~\ref{fig:halfdiskdrop}).

\begin{figure}[h!]
 \begin{center}
  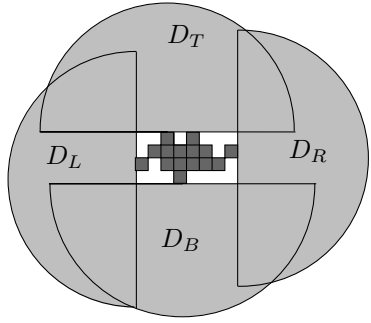
 \end{center}
\caption{Dropping halfdisks onto a non-giant component of $\Gamma$.
\label{fig:halfdiskdrop}}
\end{figure}

Now let us set $A := D_L \cup D_R \cup D_B \cup D_T$.
Assuming that $\diam(\Gamma_i) \geq r/100$ we have either $\norm{p_L-p_R} \geq
r / 100\sqrt{2}$
or $\norm{p_T-p_B} \geq r / 100\sqrt{2}$.
We can assume without loss of generality $\norm{p_L-p_R} \geq r / 100\sqrt{2}$.
Observe that, since $(p_L)_x \leq (p_T)_x \leq (p_R)_x$, we have

\mnote{Come back to this lower bound}
and similarly for $\area(D_B\setminus (D_L\cup D_R) )$.
It follows that
\[
\area( D_T \setminus (D_L\cup D_R) )
\geq \left( \frac{r/100\sqrt{2}}{2(r-2 \side \sqrt{2})}\right) \cdot \area(D_T),
\]
and
\[
\area(A) \geq \pi (r-2 \side \sqrt{2})^2 (1 + \frac{1}{200\sqrt{2}}).
\]
Thus, if $A'$ denotes the union of all cells that are contained in $A$
then, using Lemma~\ref{lem:Sarea}, we have
\[
\begin{array}{rcl}
 \area(A')
& \geq &
\area(A) - C \eta \left(\frac{\ln n}{n}\right) \\
& \geq &
{\blue \pi r^2 (1-6 \eta)^2 (1 + \frac{1}{200\sqrt{2}})} - C \eta
\left(\frac{\ln n}{n}\right) \\
& \geq &
(1+10^{-5}) \left(\frac{\ln n}{n}\right),
\end{array}
\]
for $n$ large enough, provided $\eta$ was chosen sufficiently small.
Since $A \subseteq \Bcal$ this contradicts {\red Lemma \ref{lem:cells}(i).}
Hence $\diam(\Gamma_i) < r/100$.

Let us now consider the case when some cell $c$ that is within $K/2$ cells of a
side of $[0,1]^2$ contains a point of $\Gamma_i$ with $i\geq 2$.
If $c$ is within $K/2$ cells of two sides of $[0,1]^2$ then it is in fact
in a $K\times K$ block touching a corner. Since there
are no bad cells in such a block, we have {\red from the proof of \str{1} that
$c\in\Gamma_{\max}$ and hence
$\Gamma_{\max} \cap \Gamma_\neq \emptyset$}, a contradiction.

Hence we can assume $c$ is within $K/2$ cells of one of the sides, but more
more than $K/2$ cells away from all other sides.
By symmetry considerations, we can assume the closest side is
the $y$-axis.
Let $p_L,p_R,p_T,p_B$ and $D_L, D_R, D_T, D_B$ and
$A := D_L \cup D_R \cup D_B \cup D_T$ be defined as before.
Observe that $D_R \subseteq [0,1]^2$ and
that at least half the area of $D_T, D_B$ falls inside
$[0,1]^2$.
Thus
\[
\area( A \cap [0,1]^2 ) \geq \frac12\area(A).
\]
Hence, if $A' \subseteq A$ again denotes the union of the cells contained in
$A$:
\[
\begin{array}{rcl}
\area(A') & \geq &
\area( A \cap [0,1]^2 ) - C \eta \left(\frac{\ln n}{n}\right) \\
& \geq & \frac12 \area(A) - C \eta \left(\frac{\ln n}{n}\right) \\
& \geq &
(1+10^{-5})\left(\frac{\ln n}{n}\right) / 2,
\end{array}
\]
for $\eta$ sufficiently small (using Lemma~\ref{lem:Sarea} to get the
{\red first} line,
and previous computations
to get the last line).
But this contradicts part~\ref{itm:cells.ii} of Lemma~\ref{lem:cells}.

It follows that $\diam(\Gamma_i) < r/100$, as required.

\noindent
{\bf Proof of~\str{3}:}
The proof is analogous to the proof {\red of~\strs{2}}.
If $u,v$ are two dangerous points with $\norm{u-v} < r \cdot 10^{10}$ then
we let $p_L, p_R, p_T, p_B$ be the leftmost, rightmost, top and bottom points of
$A = \{u,v\}$ and continue as before {\red to find a contradiction to Lemma
\ref{lem:cells}}.

\noindent
{\bf Proof of~\strs{4}:}
Again the proof is analogous to the proof of~\strs{2}.
If $\Gamma_i, \Gamma_j$ with $2 \leq i < j$ are two components of $\Gamma$
and the distance between
them is at most $r \cdot 10^{10}$, then we take
$p_L, p_R, p_T, p_B$ to be the leftmost, rightmost, top and bottom points of
$A = \Gamma_i \cup \Gamma_j$ and continue as before.

\noindent
{\red
{\bf Proofs of \str{2},\str{4}:} We run through the proofs replacing
$\Gamma_i$ by $\Gamma_i^+$.
It will still be possible to claim that $D_L$ contains no good cells. Before, it
would seem possible that such a cell was in $\Gamma_j,j\neq i$, disallowing
our contradiction.
Now we know that the components of $\Gamma$ are to far apart for this to happen.
}

\noindent
{\bf Proof of~\str{5}:}
Once again the proof is analogous to the proof of~\strs{2}.
If $v$ is dangerous and $\Gamma_i$ with $i \geq 2$ is a component of $\Gamma$
and the distance from $v$ to $\Gamma_i^+$ is
at most $r\cdot 10^{10}$, then we
take $p_L, p_R, p_T, p_B$ to be the leftmost, rightmost, top and bottom points
of
$A = \{v\} \cup \Gamma_{\red i}^+$ and continue as before.

\noindent
{\bf Proof of~\str{6}:}
We assume first that $c$ is at least $100 r'$ away from the boundary,
{\red where $r'=r- \side \sqrt{2}$}.
Let $p$ be the lower left corner of $c$, and for $x,y \in \Zed$ let us
{\red define
the {\em square}}
\[
S_{x,y} := p + \left[ \frac{(x-\frac12)r'}{\sqrt{5}},
\frac{(x+\frac12)r'}{\sqrt{5}}\right]\times
\left[ \frac{(y-\frac12)r'}{\sqrt{5}}, \frac{(y+\frac12)r'}{\sqrt{5}}\right], \]

\noindent
and for $k \in \eN$ let us set
$ R_k := \{ S_{x,y} : \max(|x|,|y|) = k \}$,
see Figure~\ref{fig:shortpath}.

\begin{figure}
\begin{center}
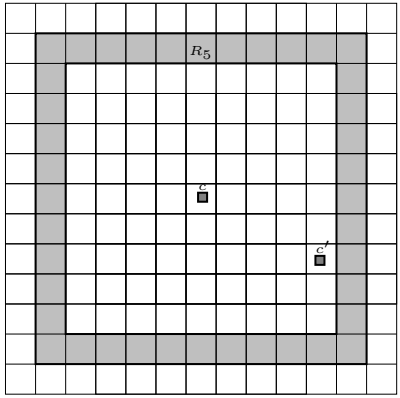
\end{center}
\caption{The squares $S_{x,y}$ and $R_5$.\label{fig:shortpath}}
\end{figure}

Observe that $c'$ is contained in some $S_{x,y}$ with $|x|,|y| \leq 11$
as $c,c'$ have distance at most {\blue $10r$}.
On the other, it cannot be that each of {\red $R_{12}, \dots, R_{27}$}
contains a square that
does not contain any good cell. This is because otherwise, if $Q_i \in R_i$ is a
square that does not contain any good cell, and we set $A :=
\bigcup_{i=12}^{27} Q_i$
then $A$ satisfies the conditions of Lemma~\ref{lem:Sarea}.
Hence setting $A' := \bigcup\{ c \in \Dcal(m_n) : c \subseteq {\red A} \}$,
we have that

\[\begin{array}{rcl}
   \area(A')
& \geq &
\area(A) - C \cdot \eta \cdot \left(\frac{\ln n}{n}\right) \\
& = &
\frac{16}{5} (r')^2 - C \cdot \eta \cdot \left(\frac{\ln n}{n}\right) \\
& \geq &
(1+10^{-5}) \pi r^2 - C \cdot \eta \cdot \left(\frac{\ln n}{n}\right) \\
& \geq &
(1+10^{-10}) \left(\frac{\ln n}{n}\right),
  \end{array} \]

\noindent
for $\eta > 0$ sufficiently small (here $C$ is the absolute constant
from Lemma~\ref{lem:Sarea}).
But this contradicts {\red Lemma \ref{lem:cells}}.

Hence, there is an $12 \leq k \leq 27$ such that all squares in $R_k$ contain
at least one good cell.
This implies that there is a cycle $C$ in $\Gamma_{\max}$ that is completely
contained in {\red $\bigcup_{k=12}^{27} R_k$}
(note that if $a \in S_{x,y}$ and $b \in S_{x+1,y}$ then $\norm{a-b} \leq r'$;
and similarly if
$a \in S_{x,y}$ and $b \in S_{x,y+1}$.)
Let us pick a $c'' \in \Gamma_{\max}$ that has distance
at least $100 r'$ from both $c$ and $c'$
(Such a $c''$ exists by~\str{1} if $n$ is sufficiently large since
$R_0, \dots, R_{27}
\subseteq B(p,100 r)$).
Consider a $cc''$-path in $\Gamma_{max}$. It must cross $C$ somewhere.
Hence, by Lemma~\ref{lem:cross},\mnote{Why invoke Lemma \ref{lem:cross}?}
we can find a path from $c$ to a cell of $C$
that stays inside squares of $R := R_0 \cup \dots \cup R_{27}$.
Similarly there is a path from $c'$ to a cell of $C$ that stays
inside squares of $R_0 \cup \dots \cup R_{27}$.
Hence, there is a $cc'$-path that stays entirely inside squares of $R$.

Now let $c=c_0, c_1, \dots, c_N=c'$ be a $cc'$-path that stays inside squares
of $R$.
{\red Let $p_i$ denote the lower left corner of the cell $c_i$.
By leaving out vertices if necessary, we can assume that
$c_ic_{i+2} \not\in E(\Gamma)$ for all $i=0, \dots, N-2$.
It follows} that the balls $\{ B(p_i,r'/2) : i \text{ even}\}$
are disjoint, {\red as are}  the balls $\{ B(p_i,r'/2) : i \text{ odd}\}$.
On the other hand we have that

\[ B(p_i, r' ) \subseteq B\left(p,r'\cdot\left(1+\frac{(\frac12+27)\sqrt{2}}
{\sqrt{5}}\right) \right)
 \subseteq B(p, 100 r'),
\]

\noindent
for all $i$.
Thus $\left\lceil\frac{N}{2}\right\rceil \pi (\frac{r'}{2})^2 \leq
\pi (100 r')^2$.
This gives that $N \leq 8 \cdot 10^4 \leq 10^5$, as required.

The case when $c$ is closer than $100r'$ to the boundary of the unit square
is analogous and we leave it to the reader.
\end{proof}


\subsection{The Obstruction Lemma}

{\blue In this section, we prove an important, technical lemma concerning the neighborhood of an obstruction.
In particular, we show that \emph{whp} every obstruction has a reasonable number of crucial vertices, along with their corresponding important vertices. This condition is essential for Maker: the obstructions are bottlenecks that require judicious play. These crucial and important vertices provide Maker with the flexibility he needs to overcome the local sparsity of the obstruction. We start by stating the main result of this section.

\begin{lemma}[Obstruction Lemma]
\label{lem:crucial}
For $\eta$ sufficiently small and $T$ sufficiently large, the following holds.
Let $(m_n)_n$ be given by~\eqref{eq:mdef} and let $V_n := \Xcal_n$
with $\Xcal_n$ as in~\eqref{eq:Xdef},
let $(r_n)_n$ be a sequence of positive numbers
such that
\[
\pi r_n^2 = \ln n + (2k-3)\ln\ln n + O(1),
\]
 with $k \geq 2$ fixed.
Then the following hold \emph{whp} with respect to $V_n, m_n, T, r_n$:
\begin{enumerate}
\item\label{crucial.i} For every $2 \leq s < T$, every $s$-obstruction has at
least
$k+s-2$ crucial vertices;
\item\label{crucial.ii} Every $(\geq T)$-obstruction has at
least $k+T-2$ crucial vertices.
\end{enumerate}
\end{lemma}

}

We prove the Obstruction Lemma via a series of intermediate, technical results. 
First, we prove a general lemma about   obstructions that are either near the boundary or have large diameter. 
The next lemma may seem like a bit of overkill at first. Let us therefore point out that it will be used 
in several places in the paper.

\begin{lemma}\label{lem:boundarydistancecount}
Let $k \geq 2$ and $0 < \delta \leq \frac{1}{100}$ and {\blue $C, D \geq 100$}
be arbitrary
but fixed, and let {\red $r_n$} be such that

\[ \pi n r_n^2 = (1+o(1))\ln n. \]

\noindent
Let $Z$ denote the number of $k$-tuples
$(X_{i_1},\dots,X_{i_k}) \in \Pcal_n^k$ such that
\begin{enumerate}
\item\label{itm:bdddist.i} $|(B(X_{i_j};r_n)\setminus B(X_{i_j};\delta r_n))
\cap \Pcal_n|
\leq D$ for all $1\leq j \leq k$, and;
\item\label{itm:bdddist.ii} $\norm{X_{i_j}-X_{i_{j'}}} \leq C r_n$ for all
$1\leq j < j'
\leq k$, and;
\item\label{itm:bdddist.iii} One of the following holds:
\begin{enumerate}
\item\label{itm:bdddist.a} There is a $1\leq j \leq k$ such that $X_{i_j}$ is
within $C r$ of a corner of the
unit square, or;
\item\label{itm:bdddist.b} There is a $1\leq j \leq k$ such that $X_{i_j}$ is
within $C r$ but no closer than $\delta r$
to the boundary of the unit square, or;
\item\label{itm:bdddist.c} There are $1\leq j < j' \leq k$ such that
$\norm{X_{i_j}-X_{i_{j'}}} \geq \delta r_n$.
\end{enumerate}
\end{enumerate}
Then $\Ee Z = O\left( n^{-c} \right)$ for some $c=c(\delta)$.
\end{lemma}

\begin{proof}
Let us set $C' := 2kC$.
Let $\Zcnr$ denote the number $k$-tuples $(u_1,\dots,u_k) \in \Pcal_n$
satisfying the demands from the lemma and $u_1$ is within $C'r_n$ of a corner
of the unit square; {\blue this includes all $k$-tuples that satisfy condition
(a).}
Let $\Zsde$ denote the number of such $k$-tuples satisfying the demands from
the lemma, and
$u_1$ is within $C'r_n$ of the boundary of the unit square, but not within
$C'r_n$ of a corner; {\blue this includes all $k$-tuples satisfying condition
(b) but not condition (a).}
Let $\Zmdl$ denote the number of such $k$-tuples with $u_1$ more than $C'r_n$
away from
the boundary of the unit square {\blue and satisfying condition (c) but not
counted by $\Zcnr$ or $\Zsde$.}

Using Theorem~\ref{thm:palm} above, we find that

\[
\Ee \Zcnr =
n^k \cdot \Ee h(Y_1,\dots,Y_k;\{Y_1,\dots,Y_k\}\cup\Pcal_n ),
\]

\noindent
where $h$ is the indicator function of the event that $Y_1$ is within
$C'r_n$ of a corner and
$(Y_1,\dots,Y_k)$ satisfy the demands from the lemma; and $Y_1,\dots, Y_k$ are
chosen uniformly at
random from the unit square and are independent of each other and of $\Pcal_n$.

Observe that, for every $u \in [0,1]^2$ we have:

\[ n \cdot \area( (B(u;r) \setminus B(u;\delta r))\cap [0,1]^2 )
 \geq \mucnr := n \cdot \frac14\pi(1-\delta^2){\red r^2}
\]

\noindent
Let $A \subseteq [0,1]^2$ be the set of points of the unit square that
are within $C'r$ of a corner, and denote $A(u) := B(u;Cr) \cap [0,1]^2$:
\mnote{Factor $k$ below for choice of $i_j$?}
\begin{equation}\label{eq:Zcnr}
\begin{array}{rcl}
\Ee\Zcnr
& \leq &
\displaystyle
n^{\red k} \int_{A}\int_{A(u_1)}\dots\int_{A(u_1)}
\Pee{\Big(} \Po\left( (B(u_1;r)\setminus B(u_1;\delta r)
\cap [0,1]^2\right)\leq D {\Big)}  {\dd}u_k\dots{\dd}u_1 \\
& \leq &
\displaystyle
n^k \cdot 4 \pi (C' r)^2 \cdot \left( \pi C^2 r^2 \right)^{k-1} \cdot
\Pee( \Po( \mucnr ) \leq D ) \\
& = &
\displaystyle
O\left( \ln^k n \cdot \exp[ - \mucnr \cdot  H(D/\mucnr) ] \right) \\
& = &
\displaystyle
O\left( \ln^k n \cdot \exp\left[ -(\frac14(1-\delta^2)+o(1))\cdot \ln n
\right] \right) \\
& = &
\displaystyle
O\left( n^{-\frac14(1-\delta^2) + o(1)} \right).
\end{array}
\end{equation}
Here we used Lemma~\ref{lem:chernoff} and that
$D /  \mucnr \to 0$ so that $H( D /  \mucnr ) \to 1$.

Let us now consider $\Zmdl$.
For $u,v\in \eR^2$, let us write
\[
A(u,v) := ((B(u;r)\setminus B(u;\delta r))\cup(B(v;r)\setminus B(v;\delta r))).
\]
If $u,v \in [0,1]^2$ are such that $\norm{u-v} \geq \delta r$
then~\eqref{eq:areadiff21} gives
\[
n \cdot  \area(A(u,v)) \geq
\mumdl := (\pi+\delta- 2\pi\delta^2)n{\red r^2}  .
\]
Using Theorem~\ref{thm:palm} again, we get

\begin{equation}\label{eq:Zmdl}
\begin{array}{rcl}
 \Ee \Zmdl
& \leq &
n^k \cdot \left(\pi C^2 {\red r^2}\right)^k \cdot
\Pee( \Po( \mumdl ) \leq D ) \\
& = &
O\left(
\ln^k n
\cdot \exp[ - \mumdl \cdot H(D/\mumdl) ]
\right) \\
& = &
O\left(
n^{-{\red 1-\delta/\pi+2\delta^2+o(1)}}
\right).
\end{array}
\end{equation}

\noindent

We now consider $\Zsde$.
Observe that if $u,v \in [0,1]^2$ are such that $v$ is more than $\delta r$
away from one side of the unit square, and more than
$r$ from all other sides, then~\eqref{eq:areaIntSq1} gives that

\[
\begin{array}{rcl}
n \cdot \area( A(u,v)\cap [0,1]^2 )
& \geq &
n \cdot \area\left( [0,1]^2 \bigcap B(v;r)\setminus
{\blue \left( B(v;\delta r) \cup B(u; \delta r) \right) } \right) \\
& \geq &
(\frac{\pi}{2}+\delta-2\pi\delta^2) n r^2.
\end{array}
\]
Similarly, if $u,v \in [0,1]^2$ are such that both $u,v$ are more than $r$
away from all but one side
of the unit square and $\norm{u-v} \geq \delta r$ then~\eqref{eq:areadiff21}
gives

\[
\begin{array}{rcl}
n \cdot \area( [0,1]^2 \cap A(u,v) )
& \geq &
n \cdot \frac12 \area(A(u,v)) \\
& \geq &
\musde := \left(\frac{\pi+\delta-2\pi\delta^2}{2}\right) n r^2.
\end{array}
\]
We see that
\begin{equation}\label{eq:Zsde}
\begin{array}{rcl}
 \Ee \Zsde
& \leq &
n^k \cdot 4 \cdot C' r_n \cdot \left(\pi C r_n^2\right)^{\red k-1} \cdot
\Pee( \Po(  \musde ) \leq D ) \\
& = &
O\left(
n^{\frac12} \ln^{\frac{2k-1}{2}} n
\cdot \exp[ - \musde \cdot H\left(D/\musde)\right]
\right) \\
& = &
O\left(
n^{-\delta/2\pi+\delta^2+o(1)}
\right).
\end{array}
\end{equation}
Since $\delta \leq 1/100$, we have $\delta/2\pi>\delta^2$.
Together with~\eqref{eq:Zcnr} and~\eqref{eq:Zmdl} this proves the bound on
$\Ee Z$
with $c(\delta) = \delta/100\pi$.
\end{proof}

Next, we prove a series of three technical lemmas about nearby pairs of points with restrictions on their common neighbours.
This argument culminates in Corollary~\ref{cor:paircnt} below, which illuminates the structure of the shared neighborhood of such pairs.
 We start with two definitions, shown visually in Figure \ref{fig:(a,b,c)}. Given $V \subseteq \eR^2$ and $r > 0$,
we will say that the pair $(u,v) \in V^2$
is an {\em $(a,b,c)$-pair} (with respect to $r$) if:

\begin{itemize}
\item[\pr{1}] $\norm{u-v} < r/100$;
\item[\pr{2}] $B(u, r-\norm{u-v}) \setminus B(u, \norm{u-v} )$ contains
exactly $a$
points of $V \setminus \{{\red u,v}\}$;
\item[\pr{3}] $\left( B(u,r)\cup B(v,r)\right) \setminus B(u, r-\norm{u-v})$
contains exactly $b$ points of $V \setminus\{{\red u,v}\}$ .
 \item[\pr{4}] $B(u,\norm{u-v})$ contains exactly $c$ points of
$V \setminus \{u,v\}$.
\end{itemize}

\noindent
We say that $(u,v)$ is an $(a,b,\geq c)$-pair if it is
an $(a,b,c')$-pair for some $c' \geq c$.

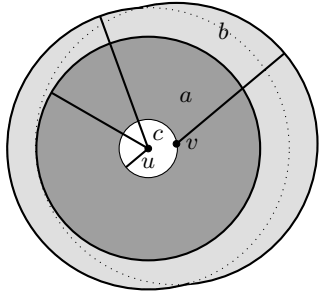
\begin{figure}[ht]
\begin{center}
\begin{tikzpicture}[scale=.125]

\draw[fill=gray!25, thick] (3,.5) circle (15);

\draw[fill=gray!25, gray!25] (0,0) circle (15);

\draw[fill=gray!25, gray!25] (0,0) circle (15);
\draw[thick] (0,0) ++ (92:15) arc (92:287:15);

\draw[fill=gray!75, thick] (0,0) circle (11.9);

\draw[fill=white] (0,0) circle (3.1);

\node[below] (u) at (0,0) {{\small $u$}};
\node[right] (v) at (3,.5) {{\small $v$}};

\draw[fill] (0,0) circle (10pt);
\draw[fill] (3,.5) circle (10pt);

\node at (4,5.5) {{\small $a$}};
\node at (8,12.5) {{\small $b$}};
\node at (1,1.5) {{\small $c$}};

\draw[dotted] (0,0) circle (15);
\draw[dotted] (3,.5) circle (15);

\draw[thick] (0,0) -- (220:3.1);
\draw[thick] (0,0) -- (150:11.9);
\draw[thick] (0,0) -- (110:15);

\begin{scope}[shift={(3,.5)}]
   \draw[thick] (0,0) -- (40:15);
\end{scope}

\end{tikzpicture}
\end{center}
\caption{Vertices $u,v$  with $\| u - v \| < r/100$ form an $(a,b,c)$-pair when the vertex counts in the three regions shown are $a$, $b$, and $c$, respectively. The definition of an $(a,b)$-pair ignores the vertex count in the innermost circle.}
\label{fig:(a,b,c)}
\end{figure}

\begin{lemma}\label{lem:pairexp}
For $k\geq 2$, let {\red $r_n$} be such that
\[
\pi n r_n^2 = \ln n + (2k-3)\ln\ln n + O(1),
\]
for some fixed $x\in\eR^2$, and let $\Pcal_n$ as in~\eqref{eq:Pdef}.
Let $a,b,c \in \eN$ be arbitrary but fixed, and let 
$R$ denote the number of $(a,b,\geq c)$-pairs in $\Pcal_n$ with respect to
$r_n$. Then $\Ee R = O\left( \ln^{a-(k+c)} n \right)$.
\end{lemma}

\begin{proof}
We start by considering $\Ee R$.
Let $\Rcnr$ denote the number of $(a,b,\geq c)$-pairs $(X_i,X_j)$ in $\Pcal_n$
for which $X_i$ is within $100 r$ of a corner of $[0,1]^2$;
let $\Rsde$ denote the number of $(a,b,\geq c)$-pairs $(X_i,X_j)$ in $\Pcal$
for which $X_i$ is within $100r$ of the boundary of $[0,1]^2$, but not within
$100 r$
of a corner, and let $\Rmdl$ denote the number of $(a,b,\geq c)$-pairs
$(X_i,X_j)$
for which $X_i$ is more than $100 r$ away from the boundary of $[0,1]^2$.

By Lemma~\ref{lem:boundarydistancecount} (with $k=2, \delta = 1/100,
D=\max(a+b,100), C=100$) we have

\begin{equation}\label{eq:RcnrFinal}
\Ee \Rcnr = O\left( n^{-{\red \Omega(1)}} \right).
\end{equation}

\noindent
For $0 < z < r/100$ let us write:

\[ \begin{array}{l}
\mu_1(z) := n \cdot \area(B(u;r-z) \setminus B(u;z)), \\
\mu_2(z) := n \cdot \area( (B(u;r) \cup B(v;r)) \setminus B(u;r-z)), \\
\mu_3(z) := n \cdot \area( B(u;z) ),
\end{array} \]

\noindent
where $u,v \in\eR^2$ are two points with $\norm{u-v}=z$.
Observe that 

\[ 
\mu_1(z)+\mu_2(z)+\mu_2(z) = n \cdot \area( B(u,r) \cup B(v,r) ) \geq
n \cdot (\pi r^2 + rz),  \]

\noindent
using Lemma~\ref{lem:areadiff}.
If $z \leq r/100$ then we also have

\begin{equation}\label{eq:papegaai}
\mu_1(z)+\mu_2(z) \geq n \cdot (\pi r^2 + rz - \pi z^2) \geq n \cdot ( \pi r^2 + \frac12 rz ).
\end{equation}

\noindent
Let us now consider $\Rmdl$. 
Using Theorem~\ref{thm:palm} and Lemma~\ref{lem:Poissontriviality}
we find

\begin{equation}\label{eq:RmdlFinal}
\begin{array}{rcl}
\Ee \Rmdl
& =  &
\displaystyle
\frac{n^2}{2} \int_{[100r,1-100r]^2} \int_{B(v;r/100)}
\frac{\mu_1(\norm{u-v})^{a}e^{-\mu_1(\norm{u-v})}}{a!} \cdot 
\frac{\mu_2(\norm{u-v})^{b}e^{-\mu_2(\norm{u-v})}}{b!} \cdot \\
& &
\displaystyle 
\quad \quad \quad \quad \quad \quad \quad \quad \quad \quad \quad \quad \quad \quad \quad
\quad \quad \quad \quad \quad \quad \quad \quad \quad \quad \quad \quad 
\mu_3(\norm{u-v})^{c}  {\dd}u{\dd}v \\
& =  &
\displaystyle
\frac{n^2}{2} (1-200 r)^2 \int_0^{r/100}
\frac{\mu_1(z)^{a}e^{-\mu_1(z)}}{a!} \cdot
\frac{\mu_2(z)^{b}e^{-\mu_2(z)}}{b!} \cdot
\mu_3(z)^{c} \cdot 
 2\pi z{\dd}z \\
& \leq  &
\displaystyle
n^2 \int_0^{r/100}
\left( \pi n r^2 \right)^{a}
\left( 4\pi n r z \right)^{b}
\left( \pi n z^2 \right)^{c}
e^{ - \pi n r^2 - n r z } 2\pi z{\dd}z \\
& = &
\displaystyle
O\left(
n^{2+c} \cdot \ln^a n \cdot (nr)^b \int_0^{r/100}
e^{- \ln n - (2k-3)\ln\ln n - n r z }  z^{b+2c+1}{\dd}z
\right) \\
& = &
\displaystyle
O\left(
n^{1+c} \cdot \left(\ln n\right)^{a-(2k-3)} \cdot (nr)^b \int_0^{r/100}
e^{- n r z }  z^{b+2c+1}{\dd}z
\right) \\
& = &
\displaystyle
O\left(
n^{1+c} \cdot \left(\ln n\right)^{a-(2k-3)} \cdot (nr)^{-(2+2c)}
\right) \\
& = &
O\left(
\left(\ln n\right)^{a-(2k-3)} \cdot (nr^2)^{-(1+c)}
\right) \\
& = &
\displaystyle
O\left(
\left(\ln n\right)^{a+2-2k-c} \right) \\
& = &
\displaystyle
O\left( (\ln n)^{a-(k+c)} \right).
\end{array}
\end{equation}

\noindent
Here we have used a switch to polar coordinates to get the second line;
we used~\eqref{eq:papegaai} to get the third line; the
change of variables $y = nrz$ to get the sixth line; and in the last
line we use $k\geq 2$.

Now we turn attention to $\Rsde$.
Let $\Rsde'$ denote the number of $(a,b,c)$-pairs $(u,v)$ with
$u$ at distance at least $r/100$ and at most $100 r$ from the boundary.
Then, again by Lemma~\ref{lem:boundarydistancecount}:

\begin{equation}\label{eq:RsdePrime}
\Ee \Rsde' = O\left( n^{-O(1)} \right).
\end{equation}

\noindent
On the other hand, by Theorem~\ref{thm:palm} and a switch to polar coordinates
(and where $z=||u-v||$ and $w$ is the distance of the nearest
of $u,v$ to the boundary):
\[
\begin{array}{rcl}
\Ee \left( \Rsde - \Rsde'\right)
& \leq  &
\displaystyle
8 n^2 \int_0^{r/100}  \!\!\! \int_0^{r/100} \!\!\!
\left( \pi n r^2 \right)^{a}
\left( 4 \pi n r z \right)^{b}
\left( \pi n z^2 \right)^{c}
e^{ - \frac{\pi}{2} n r^2 - \frac12 n r z - n r w + \pi z^2} \pi z{\dd}z{\dd} w \\
& \leq & 
\displaystyle
8 n^2 \int_0^{r/100}  \!\!\! \int_0^{r/100} \!\!\!
\left( \pi n r^2 \right)^{a}
\left( 4 \pi n r z \right)^{b}
\left( \pi n z^2 \right)^{c}
e^{ - \frac{\pi}{2} n r^2 - \frac14 n r z - n r w} \pi z{\dd}z{\dd} w \\
& = &
\displaystyle
O\left(
n^{2+c} \cdot \ln^a n \cdot (nr)^b \int_0^{r/100} \!\!\! \int_0^{r/100} \!\!\!
z^{b+2c+1}
e^{ - \frac{\pi}{2} n r^2 - \frac12 n r z - n r w} {\dd}z{\dd} w
\right) \\
& = &
\displaystyle
O\left(
n^{\frac32+c} \cdot \left(\ln n\right)^{a-(2k-3)/2} \cdot (nr)^b
\int_0^{r/100} \!\!\! \int_0^{r/100} \!\!\! z^{b+2c+1}
e^{ - \frac12 n r z - n r w} {\dd}z{\dd} w
\right) \\
& = &
\displaystyle
O\left(
n^{\frac32+c} \cdot \left(\ln n\right)^{a-(2k-3)/2} \cdot (nr)^{b-(b+2c+3)}
\right) \\
& = &
O\left(
\left(\ln n\right)^{a-(2k-3)/2} \cdot (nr^2)^{-(\frac32+c)}
\right) \\
& = &
O\left(
\left(\ln n\right)^{a-(2k-3)/2-\frac32-c}
\right) \\
& = &
O\left(
\left(\ln n\right)^{a-(k+c)}
\right).
\end{array}
\]

\noindent
Combining this with~\eqref{eq:RcnrFinal},~\eqref{eq:RmdlFinal}
and~\eqref{eq:RsdePrime} proves the lemma.
\end{proof}

Clearly the definition of $(a,b,\geq c)$-pairs can reformulated
in terms of a function $h_n$ as in Lemma~\ref{lem:configurations}.
As an immediate corollary of Lemmas~\ref{lem:pairexp}
and~\ref{lem:configurations} we now find:

\begin{corollary}\label{cor:paircnt}
For $k\geq 2$, let $(r_n)_n$ be such that
\[
 \pi n r_n^2 = \ln n + (2k-3)\ln\ln n + O(1),
\]
and let $\Xcal_n$ be as in~\eqref{eq:Xdef}.
Let $a,b,c \in \eN$ be fixed such that $a \leq c+k-1$. W.h.p.,
there are no $(a,b,\geq c)$-pairs for $\Xcal_n, r_n$.
\end{corollary}


Finally, we are ready to prove the Obstruction Lemma.

\begin{proofof}{Lemma \ref{lem:crucial}}
Let $\eta > 0$ be small and $T>0$ be large, to be determined in the proof.
All the probability theory needed for the current proof is essentially done.
For the remainder of the proof we can assume that $\Xcal_n = \{X_1, \dots, X_n\}$ is
and arbitrary point set of points in the unit square
for which the conclusions of Lemma~\ref{lem:cells}
with $\eps := 1/10^{10}$ and $K = 10^{10}$ hold, \str{1}--\str{5} hold and
the conclusions of Corollary~\ref{cor:paircnt} hold for
all $a,b,c\leq 1000 \cdot D$, where
\begin{equation}\label{eq:Ddef}
D := (T-1) \cdot  \left\lfloor \pi \left(r+\frac{\sqrt{2}}{m}\right)^2 /
\left(\frac{1}{m}\right)^2 \right\rfloor.
\end{equation}
(There are at most
$\left\lfloor \pi (r+\sqrt{2}/m)^2 / (1/m)^2 \right\rfloor$  cells that
intersect a disk of radius $r$, so
$D$ is an upper bound on the  degree of a dangerous vertex.)
That is, there are no $(a,b,\geq c)$-pairs with $b,c\leq 1000\cdot D$ and
$a \leq k+c-1$.
(Observe that $D$ is bounded because $r = O( \sqrt{\ln n / n} )$ {\red
and $m$ is defined as in \eqref{eq:mdef}}.)

Let $A$ be an arbitrary obstruction with $2 \leq s := |A| < T$.
Then $A$ is necessarily a dangerous cluster.
Pick $u,v \in A$ with $\diam(A) = \norm{u-v} =: z$.
Observe that $(B(u,r)\cup B(v,r)) \setminus B(u,r-z)$ cannot contain
more than $2D$ points, since this would imply either
$d(u) \geq D$ or $d(v) \geq D$ and this would in turn imply that one of
$u,v$ would be safe.
Thus $B(u,r-z) \setminus B(u,z)$ contains at least $m := k+s-2$ points
$X_{i_1},\dots,X_{i_m}$.
(Otherwise $(u,v)$ would be a $(a,b,c)$-pair for some $s-2 \leq c \leq D$, $a \leq c+k-1$ 
and $b \leq  2 D$, contradicting our assumptions.)

Observe that $X_{i_1},\dots,X_{i_m}$ are within $r$ of every element of $A$.
Also observe that $X_{i_1},\dots,X_{i_m}$ must be {\blue either safe or
risky}, because,
by definition of a dangerous cluster,
$A$ is an inclusion-wise maximal subset of
$\Xcal_n$ with the property that $\diam(A) < r \cdot 10^{10}$ and all
elements of $A$ are dangerous.
Thus each $X_{i_j}$ has at least $T$ neighbours in some good cell $c_j$.
By~\str{5} we must have $c_1,\dots,c_m \in \Gamma_{\max}$ so that
$X_{i_1}, \dots, X_{i_m}$ are indeed crucial.

Now suppose that $|A| \geq T$.
If $A$ is a {\red dangerous} cluster then we proceed as before.
Let us thus suppose that $A=\Gamma_i^+$ for some $i\geq 2$.
Let
\[
d_{V\setminus A}(v) := | N(v) \setminus A |,
\]
and let $W$ denote the set of those $X_i$ that are crucial for $A$.
If $|W|\geq k$ then we are done.
Let us thus assume that $|W| < k$.
Let us {\red observe:}

\begin{equation}\label{eq:vinAminW}
\text{Any $v\in A \setminus W$ has $d_{V\setminus A}(v) \leq D$.}
\end{equation}

\noindent
(This is because if $d_{V\setminus A}(v) > D$ then $v$ is adjacent
to at least $T$ vertices in some cell $c$
that does not belong to $\Gamma_i$.
This cell $c$ is therefore certainly good.
By~\str{4}, any good cell that is within $r \cdot 10^{10}$ of $A$,
but not part of $\Gamma_i$, must belong to $\Gamma_{\max}$.)

In that case $|A\setminus W| > T-k \geq k$ ($T$ being sufficiently large).
Pick an arbitrary $u \in A \setminus W$ and set
\[
A' := B(u,r/100) \cap {\red V}.
\]
Let $v$ be a point of $A'\setminus W$ of maximum distance to $u$, and write
$z := \norm{u-v}$.
Observe that $(B(u,r)\cup B(v,r)) \setminus B(u,r-z)$ cannot contain
more than $3D$ points, since this would imply either
$d_{V\setminus A}(u) {\red >} D$ or $d_{V\setminus A}(v){\red >} D$ and this
in turn
would imply that one of $u,v$ would be crucial.
Thus $B(u,r-z) \setminus B(u,z)$ contains at least $k$ points
$X_{i_1},\dots, X_{i_k}$.
(Otherwise $(u,v)$ would be {\red an} $(a,b)$-pair for some
$a \leq k-1, b \leq 1000 \cdot D$,
contradicting our assumptions.)
We claim these points must be crucial. Aiming for a contradiction, suppose
$X_{i_j} \not\in W$
for $1\leq j \leq k$.
What is more, by choice of $A'$ and $v$, we must have
$\norm{X_{i_j}-u} > r/100$.
Let us first suppose that $u$ is at least $2r$ away from the boundary of the
unit square.
Using the bound~\eqref{eq:areadiff21} and Lemma~\ref{lem:Sarea}, we see that the
total area of the cells that fall inside
$O := ( B(X_{i_j};r) \cup B(u;r) ) \setminus B(u;r/100)$ is at least

\[\begin{array}{rcl}
r^2 \left( \pi + \frac{1}{100} \right)
- \pi r^2 (\frac{1}{100})^2  - C \eta \left(\frac{\ln n}{n}\right)
& \geq &
(1+\frac{1}{1000}) \pi r^2 - C \eta \left(\frac{\ln n}{n}\right) \\
& \geq &
(1+10^{-5}) \frac{\ln n}{n},
\end{array} \]

\noindent
provided we have chosen $\eta$ small enough.
Similarly, if $u$ is within $2r$ of one of the sides of
the unit square and at least $2r$ from all other sides then
the total area of the cells inside $O$ is at least
$\frac12\area(O) - C \eta \left(\frac{\ln n}{n}\right)
\geq (1+10^{-5})\ln n/2n$;
and if $u$ is within $2r$ of two of the sides of the unit square
then the total area of the cells inside $O$ is at least
$\frac14\area(O) - C \eta \left(\frac{\ln n}{n}\right)
\geq {\red (1+10^{-5})\ln n/4n}$.
(Provided $\eta$ was chosen small enough in both cases).
Hence, by the conclusion of Lemma~\ref{lem:cells}, at least one of
these cells in $O$ must be good.
Since such a good cell is not part of the component of $\Gamma$ that
generates $A$,
and is closer than $r \cdot 10^{10}$ to $A$, it must belong to $\Gamma_{\max}$.
Hence either $X_{i_j}$ or $u$ is crucial. Since we chose $u \in A\setminus W$
we must have $X_{i_j} \in W$, a contradiction.

This proves that all of $X_{i_1},\dots, X_{i_k}$ are crucial, completing the
proof.
\end{proofof}

\section{The connectivity game}
\label{sec:connect}

To improve the presentation, we separate the proof into a
deterministic part and a probabilistic part.
The deterministic part is the following lemma:

\begin{lemma}\label{lem:conngamedet}
Suppose that $V \subseteq [0,1]^2, m \in \eN, T > 100, r > 0$ and
$r \leq \rho \leq 2r$ are such that
\begin{enumerate}
\item \str{1}--\str{5} hold with respect to $r$, and;
\item Every obstruction with respect to $\rho$ has at least 2 crucial vertices.
\end{enumerate}
Then Maker wins the connectivity game on $G(V;\rho)$.
\end{lemma}

\begin{proof}
Observe that $\Gamma=\Gamma(V;m,T,\rho)$
also satisfies~\str{1}--\str{5} if we modify them very slightly by replacing
the number $r\cdot 10^{10}$ in \str{3}--\str{5}
by $r \cdot 10^{10} / 2$.
All mentions of safe, dangerous, obstructions etc.~will be with respect to
$\rho$
in the rest of the proof.

{ \grn
From Theorem~\ref{t:lehman}, we know that Maker can win the game if
$G(V;\rho)$ contains two edge disjoint spanning trees.

In every cell $c \in \Gamma_{\max}$ there are at least $T$ vertices and they
form a clique, so we can find two edge disjoint trees on them. Since
$\Gamma_{\max}$ is connected, there is a pair of edges between every two
adjacent cells, and they complete two edge disjoint trees on all vertices in
cells of $\Gamma_{\max}$. Every vertex that is safe but not in any obstruction
has more than $T$ neighbours among the vertices in cells of $\Gamma_{\max}$,
so we can extend our pair of trees to such vertices as well. Finally, any
obstruction has at least two crucial vertices, so the vertices in the
obstruction can be connected to one of the trees via one crucial vertex, and
to the other one via another crucial vertex. This way, we obtain a pair of
edge disjoint spanning trees in our graph.
}
\end{proof}

\begin{proofof}{Theorem~\ref{thm:connhit}}
Observe that if there is a vertex of degree at most one then Breaker can
isolate this vertex
and win the connectivity game (recall that Breaker has the first move).
This implies:

\begin{equation}\label{eq:connhitLB}
\Pee\left( \rho_n(\text{Maker wins the connectivity game}) \geq
\rho_n(\text{min.\,deg. $\geq 2$}) \right) = 1.
\end{equation}

\noindent
It remains to see that $\rho_n(\text{Maker wins the connectivity game}) \leq$
$\rho_n(\text{min.\,deg. $\geq 2$})$ also  holds \emph{whp}.

Let $K>0$ be a (large) constant, and define:

\begin{equation}\label{eq:rUrLdef}
r_L(n) := \left(\frac{\ln n + \ln\ln n - K}{\pi n}\right)^{\frac12}, \quad
r_U(n) := \left(\frac{\ln n + \ln\ln n + K}{\pi n}\right)^{\frac12}.
\end{equation}

\noindent
By Theorem~\ref{thm:penrosemindeg}, there is a $K=K(\eps)$ such that

\[ \begin{array}{rcl}
\Pee{\Big(} \delta(G(n,r_L)) < 2, \, \delta(G(n,r_U)) \geq 2 {\Big)}
& \geq &
\Pee{\Big(}  \delta(G(n,r_L)) < 2 {\Big)} - \Pee{\Big(}
\delta(G(n,r_U)) < 2 {\Big)} \\
& = &
1 - e^{-(e^{K}+\sqrt{\pi e^{K}})}  - (1 - e^{-(e^{-K}+\sqrt{\pi e^{-K}})}) +
o(1) \\
& = &
e^{-(e^{-K}+\sqrt{\pi e^{-K}})}
- e^{-(e^{K}+\sqrt{\pi e^{K}})} + o(1).
\end{array}
\]

\noindent
Hence, for arbitrary $\eps > 0$, we can choose $K$ such that

\[
\Pee{\Big(} \delta(G(n,r_L)) < 2,\,  \delta(G(n,r_U)) \geq 2 {\Big)}
\geq 1 - \eps + o(1).
\]

\noindent
This shows that

\begin{equation}\label{eq:connhiteq}
\Pee{\Big(} r_L(n) \leq \rho_n(\text{min.\,deg. $\geq 2$}) \leq r_U(n) {\Big)}
\geq
1-\eps+o(1).
\end{equation}

\noindent
Now notice that, by Lemma~\ref{lem:structure} the properties~\str{1}--\str{5}
are satisfied with probability $1-o(1)$ by $V=\Xcal_n, m = m_n, T=10,
r=r_L(n)$ with
$\Xcal_n$ as given by~\eqref{eq:Xdef}, $m_n$ as given by~\eqref{eq:mdef}
and $r_L$ as above.
By Lemma~\ref{lem:crucial}, with probability $1-o(1)$,
$V=\Xcal_n, m = m_n, T=10, r=r_L(n)$ are such that every
$(\geq 2)$-obstruction with
respect to $r_L$ has at least
two crucial vertices. This is then clearly also true with respect to any
$\rho \geq r_L$.
Since a $1$-obstruction is just a vertex of low degree, by definition
every $1$-obstruction has at least two crucial vertices for every
$\rho \geq \rho_n(\text{min.\,deg. $\geq 2$})$.
Also observe that $r_L(n) \leq 2 r_U(n)$.
Hence, by Lemma~\ref{lem:conngamedet} and~\eqref{eq:connhitLB}:

\[ \Pee\left( \rho_n(\text{Maker wins the connectivity game}) =
\rho_n(\text{min.\,deg. $\geq 2$}) \right) \geq 1-\eps-o(1).
\]

\noindent
Sending $\eps \downarrow 0$ gives the theorem.
\end{proofof}


\section{The Hamilton cycle game}
\label{sec:hamilton}

In this section, we prove Theorem~\ref{thm:hamhit}. First, we prove that Maker
can win two different path-making games. These games will be useful in
constructing local paths that will be stitched together to create the Hamilton
cycle. Next, we introduce \emph{blob cycles} and prove some results about them. In
particular, we give some conditions under which blob cycles can be merged into
larger blob cycles. With these intermediate results in hand, we then
characterize the hitting radius for the Hamilton cycle game.

\subsection{Two helpful path games}
\label{sec:pathgames}

We begin by considering two auxiliary path games. Maker will use these local
path games to create parts of his Hamilton cycle.
Both games are played on a clique in which Maker tries to make a long path.
In the first game, Maker tries to make a path through all but one vertex.

\begin{lemma}\label{lem:pathplusiso}
Let $s \geq 3$ be arbitrary.
Maker can make a path of length $s-2$ in $K_s$.
\end{lemma}

\begin{proof}
\mnote{Improved proof flow, but did not change the argument.}
The cases  $s=3,4$ are straight-forward and left to the reader. Assume that
$s\geq 5$ and the statement is true for $s-1$.
Our strategy  for Maker  is as follows. He chooses an arbitrary vertex $u$.
If Breaker picks an edge incident with $u$ then Maker responds
by taking another edge incident with $u$. (If there is no such free edge, he claims an arbitrary free edge and 
forgets about it in the remainder of the game.)
If Breaker claims an edge between two vertices of $K_s\setminus u$ then
Maker responds by picking another edge between two vertices of $K_s\setminus u$
according to the winning strategy on $K_{s-1}$.

We now argue that this is a winning strategy.
Let $M$ denote Maker's graph at the end of the game.
If $M$ contains a path spanning $K_s\setminus u$ we are done.
Otherwise, $M$ contains a path $P = v_1, \ldots ,  v_{s-2}$ that contains all
but one
vertex $w$ of $K_s\setminus u$.
Observe that $u$ is adjacent to at least $\lfloor\frac{s-1}{2}\rfloor$ vertices
of $K_s\setminus u$.
If $u$ is adjacent to $v_1$ or to $v_{s-2}$ then we are done, so we can assume
this is not the case.

If $u$ is  not adjacent to $w$, then it has
at least $\lfloor\frac{s-1}{2}\rfloor > \lceil\frac{s-4}{2}\rceil$ neighbours
amongst the
$s-4$ interior vertices of $P$.
But then $u$ is adjacent to two consecutive vertices $v_i, v_{i+1}$ of $P$
and hence $P' := v_1, \dots, v_i, u, v_{i+1}, \dots, v_{s-2}$ is
a path of the required type.
Hence we can assume that $u$ is adjacent to $w$.
If $u$ is also adjacent to $v_{s-3}$ then
$P' := v_1, \dots, v_{s-3}, u, w$ is path of the required type.
Hence we can assume this is not the case.
Similarly we can assume $u$ is not adjacent to $v_2$.
If $s=5$ or $s=6$ then this in fact implies that $u$ is not adjacent to
any vertex of $P$, so that $w$ is the only neighbour of $u$.
But this contradicts the fact that $u$ has at least
$\lfloor\frac{s-1}{2}\rfloor\geq 2$
neighbours.

Considering $s>6$,
vertex $u$ has at least $\lfloor\frac{s-1}{2}\rfloor-1 >
\lceil\frac{s-6}{2}\rceil$
neighbours amongst $v_3, \dots, v_{s-4}$.
Again it follows that $u$ is adjacent to two consecutive vertices of
$P$,
so that we once again find a path of the required type.
This concludes the proof.
\end{proof}

\mnote{Tobias's updated path game since $B$ might not be a clique}
{\blue
Our second auxiliary game is the $(a,b)$ path game.
It played on the graph $G_{a,b}$ which has a vertex set of size $a+b$,
partitioned into two sets $A, B$ with $|A|=a, |B|=b$, where
$B$ induces a stable set and all edges not between two vertices of $B$ are
present. (So in particular, $A$ is a clique and $uv \in E(G_{a,b})$ for
all $u \in A, v \in B$.

Maker's objective is to create either a single path between two vertices
of $B$ that contains all vertices of $A$ (and possibly some other vertices
from $B$), or two vertex disjoint paths between vertices of $B$ that cover
all vertices of $A$.

\begin{lemma}\label{lem:abpathgame}
The $(a,b)$ path game is a win for Maker if one of the
following conditions is met
\begin{enumerate}
\item\label{itm:bgeq6} $b \geq 6$, or;
\item\label{itm:a=3} $a = 3$ and $b \geq 5$, or;
\item\label{itm:a=2} $a \in \{1,2\}$ and $b \geq 4$.
\end{enumerate}
\end{lemma}

\begin{proof}
We prove each of these three winning conditions in turn.

\smallskip

\noindent
{\bf Proof of part~\ref{itm:bgeq6}:}
We can assume that $a \geq 4$, because otherwise one of the other cases will apply.
If $b \geq 6, a \geq 4$ then Maker can play as follows.
\begin{itemize}
\item Whenever Breaker plays an edge between two vertices of $A$ then
Maker responds by claiming another edge inside $A$ according to the
strategy from Lemma~\ref{lem:pathplusiso}, guaranteeing a path in $A$ that
contains all
but one vertex of $A$.
\item If Breaker claims an edge between $u\in A$ and  $v\in B$ then
Maker claims an arbitrary unclaimed edge connecting $u$ to a vertex in $B$.
\end{itemize}
If this is not possible then Maker claims an arbitrary edge.

Consider the graph at the game's end.
Every vertex of $A$ will have at least three
neighbours in $B$.
Maker has also claimed a path $P$ in $A$ that contains $|A|-1$ vertices.
Let $a_1, a_2$ be the endpoints of $P$, and let $a_3$ be the sole vertex not on $P$.
Since $a_1, a_2$ both have at least three neighbours in $B$, we can extend
$P$ to a path $P'$ between two vertices $b_1, b_2 \in B$.
If $a_3$ has two neighbours distinct from $b_1, b_2$ then
have our two vertex disjoint paths between vertices of $B$ covering $A$.

Otherwise $a_3$ is adjacent to $b_1, b_2$ and a third vertex $b_3 \in B$.
But then we can extend $P'$ by going to $a_3$ and then to $b_3$ to
get a single path between two $B$-vertices that covers $A$.

\smallskip

\noindent
{\bf Proof of part~\ref{itm:a=3}:}
We now consider the case when $a=3, b\geq 5$.
Say $A = \{a_1, a_2, a_3\}, B = \{ b_1, \dots, b_\ell\}$ with $\ell \geq 5$.
The strategy for Maker is as follows.
For each $i=1,2,3$ he pairs the edges
$a_ib_1, a_ib_2$ (meaning that if Breaker claims one of them, then Maker responds by claiming
the other) and he pairs the edges $a_ib_3, a_ib_4$. This will make sure every
vertex of $A$ has at least two neighbours in $B$.
He also pairs $a_1b_5, a_2b_5$. This will make sure that one of $a_1, a_2$ has at least three neighbours in $B$.
Furthermore, he pairs $a_1a_2, a_1a_3$ and $a_2a_3, a_3b_5$.
This ensures Maker will either claim two edges in $A$, or he will have one edge in $A$ and
$a_3$ will have three neighbours in $B$.
This concludes the description of Maker's strategy.

To see that it is a winning strategy for Maker, let us assume first that Maker's graph
has two edges in $A$. That is, he has claimed a path $P$ through all three vertices of $A$.
Since both endpoints of $P$ have two neighbours in $B$, it extends to a path $P'$ between
two points of $B$.

Suppose then that Maker has claimed only a single edge of $A$, but every vertex of $A$ is incident with three vertices of
$B$. In this case we can reason as in the proof of part~\ref{itm:bgeq6} to see that
Maker has either a single, or two vertex disjoint paths between vertices of $B$ that cover all vertices of $A$.

\smallskip

\noindent
{\bf Proof of part~\ref{itm:a=2}:}
If $a=1$ and $b \geq 4$ then is it easy to see that Maker can claim
at least two edges incident with the unique vertex of $A$, which gives
a path of the required type.

Let us thus assume that  $A = \{a_1,a_2\}$ and $B = \{b_1,\dots, b_{\ell}\}$ with $\ell \geq 4$.
The winning strategy for Maker depends on Breaker's first move.

{\bf Case 1:} Breaker did not claim $a_1a_2$ in his first move.
Without loss of generality, Breaker claimed $a_1b_1$.
Maker responds by claiming $a_1a_2$, and for the remainder of the
game he pairs
$a_1b_3, a_1b_4$ and $a_2b_1, a_2b_2$.
This way, he will clearly end up with a single path between two
vertices of $B$ that covers $A$.

{\bf Case 2:} Breaker claimed the edge $a_1a_2$ in his first move.
In response, Maker  claims $a_1b_1$.
For the rest of the game, Maker plays as follows.
Let $e_B$ be the first edge incident with $A$ that Breaker claims after this
point. We distinguish three subcases.

\textbf{(2-a)} If $e_B=a_1 b_i$
for some $i\geq 1$, then
Maker responds by claiming $a_2 b_i$. Without loss of generality $i=2$.
For the rest of the game, Maker now pairs $a_1b_3, a_1 b_4$
and $a_2b_3, a_2b_4$.
This way, both $a_1$ and $a_2$ will have at least two neighbours in $B$, and
$a_2$ has at least one neighbour not adjacent to $a_1$.
But then there is either a single path  between $B$-vertices covering
$a_1, a_2$ or there are two vertex disjoint paths between $B$ vertices, one which
covers $a_1$ and one which covers $a_2$..

\textbf{(2-b)}
If $e_B= a_2b_1$, then Maker claims $a_2b_2$ in response and for the remainder
of the game he pairs $a_1b_3, a_1b_4$ and $a_2b_3, a_2b_4$.
Again $a_1$ and $a_2$ will have at least two neighbours in $B$, and
$a_2$ will have at least one neighbour not adjacent to $a_1$.
As in case {\bf (2-a)}, we then have the required path(s).

\textbf{(2-c)}
If $e_B=a_2b_i$ for $i\geq 2$, then without loss of generality we can assume
$i=2$.
Maker responds by claiming $a_2b_3$; and for the remainder of the game
he pairs $a_1b_2, a_1b_4$ and $a_2b_1, a_2b_4$.
Again $a_1$ and $a_2$ will have at least two neighbours in $B$, and
$a_2$ will have at least one neighbour not adjacent to $a_1$.
Again we can find the required path(s).
\end{proof}

}

\subsection{Blob Cycles}

\mnote{If we permute inner vertices of a blob, we get another blob cycle, we
can  reduce the blob by a couple of vertices. I made this change throughout.
Please read it carefully.}

{\blue In this section, we prove some results about \emph{blob cycles.}}
An $s$-blob cycle  $C$ is the union of a cycle with a clique on $s$
consecutive vertices of it; this $s$-clique is called the \emph{blob} of $C$.
See Figure \ref{fig:kblob} for a depiction.
{\blue When $C$ is an $s$-blob cycle with $|V(C)|=m$, we typically write
$V(C) = \{ u_0, \dots, u_{m-1} \}$ where $u_iu_{i+1}$ is an edge for all $i$,
 modulo $m$,
and $u_0, \dots, u_{s-1}$ are the blob vertices.}
A {\em $s$-blob Hamilton cycle} is the union of a Hamilton cycle with a
clique on $s$ consecutive vertices of it.

{\blue  Blob cycles will be the building blocks of our Hamilton cycle in
$G(n, r_n)$. In the next section,  we will construct our Hamilton cycle by
gluing together many blob cycles in a good cell. Most of these blob cycles
are used to connect to a sparser part of the graph. We then use a much larger
blob cycle to conglomerate the blob cycles in a good cell into one blob cycle.
The lemmas below concern adding vertices and/or edges to a blob cycle, and
then finally combining many blob cycles into one.}

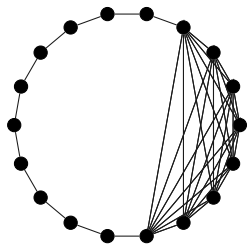
\begin{figure}[ht!]
 \begin{center}
\begin{tikzpicture}[scale=.5]

\tikzstyle{every node}=[circle, draw, fill=black,
                        inner sep=0pt, minimum width=5pt]

\foreach \i in {0,...,17}
{
\node (X\i) at (\i *20: 3) {};
}

{
\foreach \i/\j in {3/4,4/5,5/6,6/7,7/8,8/9,9/10,10/11,11/12,12/13,13/14}
{
\draw[black!90] (X\i) -- (X\j);
}

\foreach \i in {3,2,1,0,17,16,15,14}
{
  \foreach \j in {3,2,1,0,17,16,15,14}
  {
  \draw[black!90] (X\i) -- (X\j);
  }
}
}

\end{tikzpicture}
 \end{center}
\caption{An $8$-blob cycle on 18 vertices.\label{fig:kblob}}
\end{figure}

\begin{lemma}\label{lem:sblob}
For {\blue $k \geq 4$}, there is a $N=N(k)$ such that Maker
can make a $k$-blob Hamilton {\red cycle} in {\red an} $s$-clique for $s >N$.
\end{lemma}

\begin{proof}
By Theorem~\ref{thm:Hgamedet} there is an $s_0 = s_0(k)$ such
that Maker can make a $k$-clique on $K_s$ for all $s \geq s_0$.
Now let $s_1 = s_1(k)$ be such that Maker
can win the biased $(1:b)$-Hamilton cycle game on
$K_s$ for all $s\geq s_1$, where $b := {s_0\choose 2}$.
(Such a $s_1$ exists by Theorem~\ref{thm:kriv}.) \\
We set $N := \max\left(2 {s_0\choose 2}+1,s_1\right) + k$.

For $s \geq N$ Maker {\red can now} play as follows.
First he makes a $k$-clique on the first $s_0$ vertices of
$K_s$ by playing according to the strategy given by Theorem~\ref{thm:Hgamedet}.
(If Breaker plays an edge not between two vertices among the first $s_0$
then Maker pretends {\grn Breaker} played an arbitrary free edge between two
of the first
$s_0$ vertices and responds to that.)

Once Maker has succeeded in making a $k$-clique $C$, no more than
${s_0 \choose 2}$
edges have been
played so far.
He now continues playing as follows:

\begin{itemize}
 \item If Breaker claims an edge between two vertices of $V\setminus C$ then
Maker responds
according to the strategy given by Theorem~\ref{thm:kriv} that will ensure him a
cycle on $V\setminus C$;
\item If Breaker claims an edge between $u \in V\setminus C$ and $v \in C$ then
{\blue Maker} claims an unclaimed edge between $u$ and another vertex of $C$.
\end{itemize}

At the end of the game Maker's graph will contain a cycle $\tilde{C}$ on
$V\setminus C$.
Note that if a vertex $u \in V\setminus C$ was not incident with any edge
claimed {\red by} Breaker
at the time when Maker finished claiming $C$ then, by the end of the game,
$u$ will be incident in Maker's graph with at least {\blue $\lceil k/2\rceil
\geq 2$} vertices of $C$.
By choice of $N$, more than half of all the vertices of $V\setminus C$ were
not incident
with any of Breaker's edges just after Maker finished building his $k$-clique
$C$.
Hence there will be two vertices $u,v$ that are consecutive on $\tilde{C}$
and two vertices $x \neq y \in C$ such that Maker claimed the edges $ux, vy$.
This gives the required $k$-blob Hamilton cycle.
\end{proof}

\begin{lemma}\label{lem:blobcrap0}
For every $s\geq 3$ the following holds.
Let $G$ be a graph, $C \subseteq G$ an $s$-blob cycle, and $u, v \not\in V(C)$
with $uv \in E(G)$ such that $u,v$ each have at least {\blue $\lceil |V(C)|/2
\rceil$} neighbours on $C$.
Then there is an {\blue $s'$-blob cycle $C'$
with $s' \geq s-2$ and} $V(C') = \{u,v\} \cup V(C)$ where $uv \in E(C')$ and
the blob of $C'$ is contained
in the blob of $C$.
\end{lemma}

{\blue
\begin{proof}
In this proof (and the ones that follow), we continue to use the convention
that  $V(C) = \{ u_0, \dots, u_{m-1} \}$ where $u_iu_{i+1}$ is an edge for all
$i$, modulo $m$,
and $u_0, \dots, u_{s-1}$ are the blob vertices.

Since $u, v$ each have at least  $\lceil |V(C)|/2\rceil = \lceil m/2\rceil$
neighbours
on $C$, there
are two consecutive vertices $u_i, u_{i+1}$ such that  $uu_i, vu_{i+1}
\in E(G)$.
If $s-1 \leq i \leq m-1$, then $u_0,\dots,u_i,u,v,u_{i+1},\dots$,
$u_{m-1}$ is an $s$-blob  cycle.
 If $i=0$ then $u_0,u,v,u_1,\ldots,u_{m-1}$ is an $(s-1)$-blob cycle;
the case $i=s-2$ is similar.
If $1 \leq i \leq s-3$, we can relabel the blob vertices $u_1, u_2, \ldots ,
u_{s-3}$ so that $i=1$.
Then $u_0, u_1,  u, v, u_2, \dots, u_{m-1}$ is an $(s-2)$-blob  cycle. In all
cases, the new blob is a subset of the old blob.
\end{proof}
}

\mnote{Moved this lemma before the next one, since it is simpler.}
\begin{lemma}\label{lem:blobcrap1}
For every $\ell$ there is a $s=s(\ell)$ such that the following holds.
If $G$ is a graph, $C \subseteq G$ is an $s$-blob cycle, and $v \not\in V(C)$
has at least $|V(C)|/2-\ell$ neighbours on $C$, then
there exists an {\blue $s'$-blob cycle
$C'$ with $s' \geq s-2$ and}  $V(C') = \{v\}\cup V(C)$ and the blob of $C'$ is
contained in the blob of $C$.
\end{lemma}

\begin{proof}
{\blue
Let us set $s := 100 \cdot (\ell+1)$,\mnote{We can reduce $s=2\ell$ if we want
to.} and
 write $V(C) = \{ u_0, \dots, u_{m-1} \}$ and  $u_0, \dots, u_{s-1}$ are the
vertices of an $s$-clique.
First, suppose first that $v$ is adjacent to at least two vertices of $u_1,
\dots, u_{s-2}$. By reordering these vertices,
$vu_1, vu_2 \in E(G)$. Set $C' = u_0, u_1, v, u_2, \dots, u_{m-1}$.
Then $C'$ is clearly an $(s-2)$-blob cycle whose blob is inside the blob of $C$.
}

Next, suppose then that $v$ has at most one neighbour among $u_1,\dots,
u_{s-2}$.
Then $v$ must have $\lfloor m/2\rfloor - (\ell+1)$ neighbours on the path
$P = u_{s-1},
\dots, u_{m-1}, u_0$.
Observe that the biggest subset of $V(P)$ that does not contain two
consecutive vertices has cardinality
\[
\lceil |V(P)|/2\rceil = \lceil (m-s+2)/2\rceil <
m/2 - 50(\ell+1) + 1 < \lfloor m/2\rfloor - (\ell+1).
\]
Therefore $v$ is adjacent to two consecutive
vertices on the path $P$, say $vu_i, vu_{i+1} \in E(G)$.
This time $C' = u_0, \dots, u_i, v, u_{i+1}, \dots, u_{m-1}$ is clearly an
$s$-blob
cycle whose blob is identical
to the blob of $C$.
\end{proof}

\begin{lemma}\label{lem:blobcrap2}
For every $\ell$ there is an $s=s(\ell)$ such that the following holds.
Suppose that $G$ is a graph, $C_1, C_2 \subseteq G$ are vertex disjoint,
$C_1$ is an $s$-blob cycle,
$C_2$ is a 5-blob cycle, and every vertex of the blob of $C_2$ has at least
$\lfloor |V(C_1)|/2 \rfloor -\ell$ neighbours on $C_1$.  Then
there exists an {\blue $s'$-blob cycle $C$ where $s' \geq s-2$}
 with $V(C) = V(C_1)\cup V(C_2)$ and the blob of $C$
is contained in the blob of $C_1$.
\end{lemma}

{\blue
\begin{proof}
Let us set $s := 100 \cdot (\ell+1)$\mnote{We can reduce the value of
$s$ to $2 \ell +5$, if desired.},
and write $V(C_1) = \{ u_0, \dots, u_{m-1} \}$ where   $u_0, \dots, u_{s-1}$
are the blob vertices.
Similarly $V(C_2) = \{v_0, \dots, v_{n-1}\}$
where $v_0, \dots, v_4$ form a clique and $v_0,\dots, v_4$ each have at
least $\lfloor m/2\rfloor - \ell$ neighbours on
$C_1$.

First suppose that $v_1u_0, v_2u_1 \in E(G)$.
The graph  $C' = u_0$, $v_1, v_0$, $v_{n-1}, \dots, v_2$, $u_{1},
\dots, u_{m-1}$ is an $(s-1)$-blob cycle of the required type.
Next, when  $v_1u_1, v_2u_2 \in E(G)$  the graph $C' = u_0, u_1$, $v_1, v_0$,
$v_{n-1}, \dots, v_2$, $u_{2},
\dots, u_{m-1}$ is an $(s-2)$-blob cycle.
More generally, if any two among $v_1, \dots, v_3$ have distinct neighbours
among $u_0,\dots, u_{s-1}$ then we can just re-order the vertices
of $C_1$ and $C_2$ and apply the
above argument to get an $(s-1)$-blob cycle or an $(s-2)$-blob cycle.

If any two of $v_1,v_2,v_3$ both have at least two neighbours among
$u_1,\dots, u_{s-2}$ then we are done by the previous argument.
So assume this is not the case.
Relabeling $C_2$ if necessary, $v_1, v_2$ each have at most one neighbour
among $u_1,\dots, u_{s-2}$.
Then $v_1,v_2$ each have $\lfloor m/2 \rfloor -\ell-1$ neighbours on
$P = u_{s-1},\dots, u_{m-1}, u_0$.
In turn, there are at least {\red $\lfloor m/2 \rfloor -\ell-1$}
\mnote{Replaced -2 with -1 here. Aren't the worst 2 cases: (a) 2 endpoints,
or (b) a set of consecutive vertices?} points of $P$ {\em adjacent to a
neighbour of $v_1$}.
Since $(\lfloor m/2 \rfloor -\ell-1) + (\lfloor m/2 \rfloor -\ell-1) >
m-s+1 = |V(P)|$ there exists an
$s-1 \leq i \leq m-1$ such that $v_1u_i, v_2u_{i+1} \in E(G)$.
Clearly $C' = u_0, \dots, u_i$, $v_1, v_0$, $v_{n-1},
\dots, v_2$, $u_{i+1}, \dots, u_{n-1}$
is an $s$-blob cycle of the required type.
\end{proof}
}

Our final conglomeration lemma is the crucial tool for merging an ensemble of
blob cycles, along with some additional edges and vertices, into a single
spanning cycle. Ultimately, we will use this lemma to handle every good cell
in our dissection. The additional vertices and edges will be chosen so that
they allow us (1) to combine spanning cycles in neighboring cells and (2) to
add vertices in bad cells to our spanning cycle. See Section
\ref{sec:hamproof} for details.

\begin{lemma}\label{lem:blobcrap5}
For every $\ell_1, \ell_2$ there exists an $s=s(\ell_1, \ell_2)$ such that the
following holds.
Suppose $G$ is a graph and $C, C_1,\dots, C_{\ell_1} \subseteq G$ and
$u_1, \dots, u_{\ell_1}, v_1,\dots, v_{\ell_1}, w_1, \dots, w_{\ell_2} \in V(G)$
are such that:
\begin{itemize}
\item The vertices $u_1, \dots, u_{\ell_1}, v_1,\dots, v_{\ell_1}, w_1,
\dots, w_{\ell_2}$ are distinct;
\item $u_iv_i \in E(G)$ for $i=1,\dots,\ell_1$;
\item $C_1,\dots, C_{\ell_1}$ are 10-blob cycles and $C$ is an $s$-blob cycle;
\item $C, C_1,\dots, C_{\ell_1}$ are vertex disjoint and  do not contain any of
the {\red $u_i$'s, $v_i$'s or $w_i$'s};
\item $u_i, v_i$ both have at least $\lfloor |V(C_i)|/2\rfloor$ neighbours on
$C_i$
for $i=1,\dots,\ell_1$;
 {\blue
\item Every vertex of $\{w_1,\dots,w_{\ell_2}\}$
has at least
$\lfloor|V(C)|/2\rfloor$ neighbours on $C$;
\item For every $1\leq i\leq \ell_1$, every vertex of the blob of $C_i$ has at
least $\lfloor |V(C)|/2\rfloor$
neighbours on $C$.
}
\end{itemize}
Then there is an  {\blue $(s-2\ell_1-2\ell_2)$}-blob cycle $C'$ such that

\[  \begin{array}{l}
V(C') = V(C) \cup V(C_1)\cup\dots\cup V(C_{\ell_1}) \cup \{u_1, \dots,
u_{\ell_1},
v_1,\dots, v_{\ell_1}, w_1, \dots, w_{\ell_2}\},  \\
u_1v_1, \dots, u_{\ell_1}v_{\ell_1} \in E(C'),
\end{array} \]

\noindent
and the blob of $C'$ is contained in the blob of $C$.
\end{lemma}

{\blue
\begin{proof}
We prove the statement for $s := \max \{ s_1(\ell_1), s_2(\ell_2) \}  + 2
\ell_1 + 2 \ell_2$
where $s_1(\cdot)$ is the function provided by Lemma~\ref{lem:blobcrap1} and
$s_2(\cdot)$ is
as provided by Lemma~\ref{lem:blobcrap2}.

First, we apply Lemma~\ref{lem:blobcrap0} to each triple $u_i, v_i, C_i$.
 {\red For $1 \leq i \leq \ell_1$, this yields an 8-blob cycle $C_1', \dots,
{\red C_{\ell_1}'}$ containing the edge $u_iv_i$ and}
such that each vertex of the blob of $C_i'$ has at least $\lfloor
|V(C)|/2\rfloor$
neighbours
on $C$. {\red Also,  the
edge $u_iv_i$ will not be part of the assumed 8-blob of $C_i'$}.

Next, we join all of the $w_i$'s with the $s$-blob cycle $C$. We apply
Lemma~\ref{lem:blobcrap1} repeatedly, with $\ell = 2\ell_2$. This value of
$\ell$ lets us pick distinct $u_j, u_{j+1}$ for each $w_i$, one after another.
We then add all the $w_i$ to $C$ in one fell swoop to get the
$(s-2\ell_1)$-blob cycle $C'$ that contains all of the $w_i$'s.

Finally, we merge $C'$ with the $C_1', \ldots , C_{\ell_1}'$ in a similar
manner. We apply Lemma~\ref{lem:blobcrap2} repeatedly, with
$\ell = 2 \ell_1 + 3 \ell_2$. This choice of $\ell$ allows us to pick distinct
$u_j, u_{j+1}$ for each $C_i'$.
We then merge $C_1', \ldots , C_{\ell_1}$ with  $C$ in one fell swoop to get
the desired $(s-2\ell_1-2\ell_2)$-blob cycle $C''$.
\end{proof}
}

\subsection{Proof of the Hamilton cycle game} \label{sec:hamproof}

We again divide the remainder of the proof of  Theorem~\ref{thm:hamhit}
into a deterministic and a probabilistic part.
The following lemma is the deterministic part:

\mnote{This proof has been reworked, but I did not make it all blue.
Please read this whole section carefully!}
\begin{lemma}\label{lem:hamgamedet}
There exists a value of $T$ such that the following holds.
If $V \subseteq [0,1]^2, m \in \eN, r > 0$ and $r \leq \rho \leq 2r$ are
such that
\begin{enumerate}
\item \str{1}--\str{6} hold with respect to $r$, and;
\item\label{itm:hamgamedet.ii} For all $2 \leq s \leq T$, every
$s$-obstruction with
respect to $r$ has
at least $s+2$ crucial vertices, and;
\item\label{itm:hamgamedet.iii} Every $(\geq T)$-obstruction  with respect to
$r$
has at least six crucial vertices, and;
\item $G(V,\rho)$ has minimum degree at least four.
\end{enumerate}
Then Maker wins the Hamilton cycle game on $G(V;\rho)$.
\end{lemma}

Before launching into the proof, we give a high-level overview of the
argument. 
We use the spanning tree $\Tcal$ of $\Gamma_{\max}$ from Lemma~\ref{lem:span}
as the skeleton of the Hamilton cycle.
Maker plays many local mini-games, and then stitches together the local structures to create the Hamilton cycle.  Figure \ref{fig:ham-strategy} shows a simplified version of the mini-games in a good cell.
We  \emph{mark} a large,
but finite, number of vertices in good cells. The majority of the marked
vertices are \emph{important vertices} in good cells that are associated to
\emph{critical vertices} for bad cells. Some marked vertices are reserved for
joining the blob cycles in good cells that are adjacent in $\Tcal$.   Maker
plays the game so that (1) in each cell, he can create a family of blob cycles
spanning the unmarked vertices (Lemma~\ref{lem:sblob}, multiple times), (2)
there are two independent edges between good cells that are adjacent in
$\Tcal$, and (3) he can construct a family of paths through the bad cells.
In particular, every vertex in a bad cell will be on exactly one path between
a pair of marked vertices that are in the same good cell.  At this point, we
place the marked vertices into two categories. Pairs of marked vertices that
are endpoints of a special path are temporarily considered \emph{marked edges}
$u_i v_i$. Note that the path games create one or two paths through each obstruction. The unused marked vertices $w_j$ are considered as \emph{marked singletons}.

\begin{figure}[h!]
\begin{center}
\begin{tabular}{cc}

\begin{tikzpicture}[scale=.32]

\tikzstyle{myarrows}=[line width=0.25mm,draw=black,-triangle 45,postaction={draw, line width=1mm, shorten >=2mm, -}]

\tikzstyle{myarrows2}=[line width=0.25mm,draw=gray,-triangle 45,postaction={draw=gray, line width=1mm, shorten >=2mm, -}]


\draw[gray!25,fill] (0,10) -- (10,10) -- (10,12) -- (0,12) -- cycle;
\draw[thick, fill=gray!40] (0,0) -- (10,0) -- (10,10) -- (0,10) -- cycle;
\draw[thick] (0,10) -- (0,12);
\draw[thick] (10,10) -- (10,12);

\draw[dashed] (0,8) -- (10,8);
\draw[dashed] (0,5) -- (10,5);

\draw[dashed] (3.33,5) -- (3.33,10);

\draw[dashed] (6.66,5) -- (6.66,10);

\foreach \i in {1.6, 5, 8.4}
{

\begin{scope}[shift={(\i,0)}]

\foreach \j in {-1, 0, 1}
{
\node[circle, draw, fill=black,inner sep=0pt, minimum width=2pt] at (\j, 9.25) {};
\node[circle, draw, fill=black,inner sep=0pt, minimum width=2pt] at (\j, 8.75) {};
}

\end{scope}


\draw[thick]  (-3,9) ellipse (.75 and 1.5);

\node at (-3, 6.75) {\scriptsize{obstruction}};

\node[circle, draw, fill=black,inner sep=0pt, minimum width=2pt] (V1) at (-3, 9.5) {};
\node[circle, draw, fill=black,inner sep=0pt, minimum width=2pt] (V2) at (-3, 9) {};
\node[circle, draw, fill=black,inner sep=0pt, minimum width=2pt] (V3) at (-3, 8.5) {};


\draw[thick]  (13,9) circle (1.5);

\node[circle, draw, fill=black!70,inner sep=0pt, minimum width=2pt] (Z1) at (12.5, 9.5) {};
\node[circle, draw, fill=black!70,inner sep=0pt, minimum width=2pt] (Z3) at (13.5, 9.5) {};
\node[circle, draw, fill=black!70,inner sep=0pt, minimum width=2pt] (Z4) at (12.5, 8.5) {};
\node[circle, draw, fill=black!70,inner sep=0pt, minimum width=2pt] (Z2) at (13.5, 8.5) {};

\node at (13, 7) {\scriptsize{safe cluster}};


{
\draw[myarrows] (.25,9) -- (-2.5,9);
\draw[myarrows] (5,9.75) -- (5,11);
\draw[myarrows] (9.75,9) -- (12,9);
}

{
\foreach \k in {-3.5, 0, 3.5}
{
\begin{scope}[shift={(\k,0)}]
\draw[myarrows] (4.5,8.25) -- (4.5,6.25);
\draw[myarrows] (4.5,5.5) -- (4.5,2.5);
\draw[myarrows] (5.5,8.25) -- (5.5,2.5);
\end{scope}
}
}
}

\end{tikzpicture}

&

\begin{tikzpicture}[scale=.32]

\tikzstyle{myarrows}=[line width=0.25mm,draw=black,-triangle 45,postaction={draw, line width=1mm, shorten >=2mm, -}]

\tikzstyle{myarrows2}=[line width=0.25mm,draw=gray,-triangle 45,postaction={draw=gray, line width=1mm, shorten >=2mm, -}]


\draw[gray!25,fill] (0,10) -- (10,10) -- (10,12) -- (0,12) -- cycle;
\draw[thick, fill=gray!40] (0,0) -- (10,0) -- (10,10) -- (0,10) -- cycle;
\draw[thick] (0,10) -- (0,12);
\draw[thick] (10,10) -- (10,12);

\draw[dashed] (0,8) -- (10,8);
\draw[dashed] (0,5) -- (10,5);

\draw[dashed] (3.33,5) -- (3.33,10);

\draw[dashed] (6.66,5) -- (6.66,10);

\foreach \i in {1.6, 5, 8.4}
{

\begin{scope}[shift={(\i,0)}]

\foreach \j in {-1, 0, 1}
{
\node[circle, draw, fill=black,inner sep=0pt, minimum width=2pt] at (\j, 9.25) {};
\node[circle, draw, fill=black,inner sep=0pt, minimum width=2pt] at (\j, 8.75) {};
}

\end{scope}


\draw[thick]  (-3,9) ellipse (.75 and 1.5);

\node at (-3, 6.75) {\scriptsize{obstruction}};

\node[circle, draw, fill=black,inner sep=0pt, minimum width=2pt] (V1) at (-3, 9.5) {};
\node[circle, draw, fill=black,inner sep=0pt, minimum width=2pt] (V2) at (-3, 9) {};
\node[circle, draw, fill=black,inner sep=0pt, minimum width=2pt] (V3) at (-3, 8.5) {};


\draw[thick]  (13,9) circle (1.5);

\node[circle, draw, fill=black!70,inner sep=0pt, minimum width=2pt] (Z1) at (12.5, 9.5) {};
\node[circle, draw, fill=black!70,inner sep=0pt, minimum width=2pt] (Z3) at (13.5, 9.5) {};
\node[circle, draw, fill=black!70,inner sep=0pt, minimum width=2pt] (Z4) at (12.5, 8.5) {};
\node[circle, draw, fill=black!70,inner sep=0pt, minimum width=2pt] (Z2) at (13.5, 8.5) {};

\node at (13, 7) {\scriptsize{safe cluster}};

}

\draw[very thick] (5,2.5) ellipse (4 and 1.5);

\foreach \i in {1.6, 5, 8.4}
{

\begin{scope}[shift={(\i,0)}]

\draw[very thick] (0, 6.5) circle (1);

\end{scope}

\draw[thick] (.6, 8.75) -- (V3) -- (V2) -- (V1) -- (.6, 9.25);

\draw[thick] (9.4, 8.75) -- (Z4) -- (Z3) -- (Z2) -- (Z1) -- (9.4, 9.25);

\draw[thick] (4, 9.25) -- (4,10.75);
\draw[thick] (5, 9.25) -- (5,10.75);

\node[circle, draw, fill=black,inner sep=0pt, minimum width=2pt] at (4, 10.75) {};
\node[circle, draw, fill=black,inner sep=0pt, minimum width=2pt] at (5, 10.75) {};

}

\draw[fill] (5,3.75) ellipse (2.5 and .6);

\foreach \i in {1.6, 5, 8.4}
{

\begin{scope}[shift={(\i,0)}]

\draw[fill] (0, 7.3) ellipse (.6 and .4);

\end{scope}

}

\end{tikzpicture}

\\
(a) & (b)
\end{tabular}

\caption{A schematic for the mini-games for Maker's  Hamilton strategy in a good cell. (a) The marked vertices (top row) are designated to make paths through nearby obstructions and safe clusters, and to connect to good cells that are adjacent in the tree $\Tcal$. The unmarked vertices (bottom two rows) are partitioned into subsets for blob cycle creations. Maker claims half the edges from each vertex to each  lower level. (b) After all edges have been claimed, Maker has paths through obstructions and two independent edges to nearby good cells, and some unused marked vertices, and a soup of blob cycles. The blob cycles  absorb the paths and unused vertices (using the many edges down to the lower levels in the good cell), culminating in the Hamilton cycle.}
\label{fig:ham-strategy}

\end{center}
\end{figure}
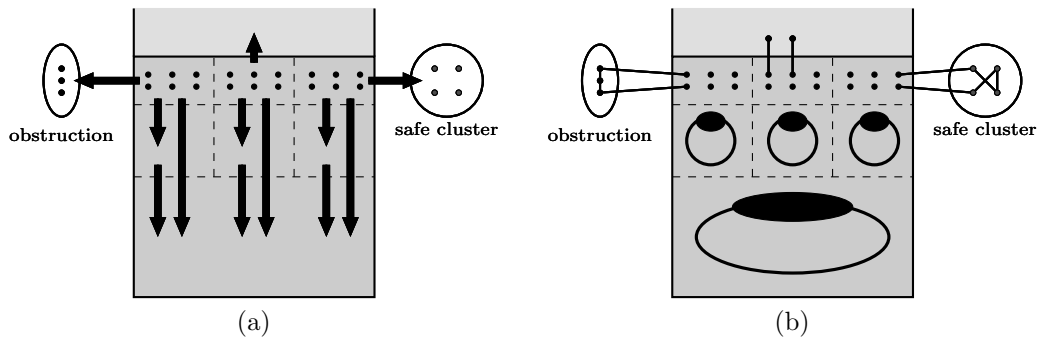

{\blue
 We construct our Hamilton cycle as follows. For each good cell, we
conglomerate its vertices via  Lemma \ref{lem:blobcrap5} using the family of
 blob cycles and the marked edges $u_i v_i$ and marked singletons $w_j$.
This creates a cycle that spans all  unmarked vertices, all marked edges
and all marked singletons. Finally, we replace each marked edge with the graph
structure it represents. This adds all of the bad vertices to our cycle, and
connects  every cell in $\Gamma_{\max}$.  The result is the
 desired Hamilton cycle. With this outline in mind, we proceed with the proof.
 }

\begin{proof}
{\blue
Let $T$ be a (large but finite) number, to be made explicit later on in the
proof. For each good cell of our dissection, we will identify at most $T$
vertices to help us with nearby vertices in bad cells. Collectively, we will
refer these vertices as \emph{marked}. The remaining vertices are
\emph{unmarked}, and these unmarked vertices will be used to create blob
cycles within the good cell.
}

\mnote{Really: holds for $2r$ and therefore for $r$. I think we can say this
more effectively.}
As in the proof of Lemma~\ref{lem:conngamedet}, observe that
$\Gamma( V; m, T, \rho)$
satisfies~\str{1}--\str{6} if we modify~\str{3}--\str{5}
very slightly by replacing the number $r \cdot 10^{10}$ by $r\cdot 10^{10} / 2$.
Also observe that items~\ref{itm:hamgamedet.ii} and~\ref{itm:hamgamedet.iii}
clearly
also hold with respect to $\rho$.
Again all mention of safe, dangerous, obstructions and so on will be with
respect
to $\rho$ from now on.


{\blue
Before the game starts, Maker identifies many local games that he will play.
These games fall into four categories, according to the types of vertices
involved. We have
games between good vertices in neighboring good cells;
games for bad vertices and nearby marked good vertices;
 games for unmarked vertices within a good cell;
and games between marked and unmarked vertices within a good cell.
We now describe these games in detail.}

Lemma~\ref{lem:span} implies  that $\Gamma_{\max}$ has a spanning tree
$\Tcal$ of maximum degree at most five.
Before the game, we fix such a spanning tree.
\mnote{I have changed how we use the spanning tree to make the Hamilton cycle.}
{\blue
Maker will use this spanning tree as the skeleton for the Hamilton cycle.
For each edge $c c'$ in the spanning tree, we identify four vertices in each
cell that are important for the edge $cc'$ and are considered marked.
Since $\Tcal$ has maximum degree 5, we mark at most 20 vertices in each good
cell. With $T$ being
large, we can and do take all these marked vertices distinct.
}

Maker keeps track of each obstruction. By Lemma \ref{lem:crucial}
(with $k=4$), every $s$-obstruction has at least $2+s \geq 4$ crucial vertices
for when $2 \leq s < T$, or at least $4$ crucial vertices when $s \geq T$. Note that a 2-obstruction must have at least 4 crucial vertices. 
This for instance means we cannot have two vertices of degree four that are joined by an edge (which  would have given Breaker a winning strategy).

To each obstruction, we assign either all of its crucial vertices if there
are fewer than six, or six of its crucial vertices, if there are more.
(No vertex can be crucial for multiple obstructions: obstructions are
well-separated by \str{3} -- \str{5}, so a cell contains crucial vertices for
at most one obstruction.)
Each of these crucial vertices has at least $T$ important vertices inside
some cell of
$\Gamma_{\max}$.
We assign four important vertices (all  in the same cell) to each crucial
vertex. (As above, we choose all important vertices to be distinct.)
{\blue Every  important vertex and every crucial vertex in a good cell  is
considered marked.
In $G$, vertices in a good cell are adjacent to at most one obstruction, so
 this adds at most $6+24=30$ marked vertices to each cell.}

 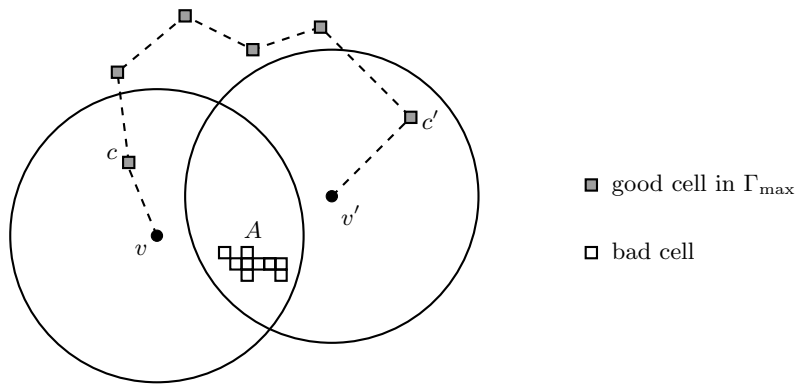
\begin{figure}[h!]
\begin{center}
\begin{tikzpicture}[scale=.3]

\begin{scope}[shift={(.25,.25)}]
\draw[thick, dashed] (-5.5,5) -- (-6,9) -- (-3,11.5) -- (0,10) -- (3, 11) -- (7,7);
\end{scope}

\draw[thick, dashed] (3.75, 3.75) -- (7.25, 7.25);

\draw[fill] (3.75, 3.75) circle (.25);
\node[below right] at (3.75, 3.75) {\small{$v'$}};
\draw[thick] (3.75, 3.75) circle (6.5);

\node[right] at (7.33, 7.15) {\small{$c'$}};

\draw[thick, dashed] (-4, 2) -- (-5.25, 5);

\draw[fill] (-4, 2) circle (.25);
\node[below left] at (-4, 2) {\small{$v$}};
\draw[thick] (-4, 2) circle (6.5);

\node[above] at (-6, 5) {\small{$c$}};

\foreach \x/\y in {
-6/9,-3/11.5,0/10,3/11, -5.5/5, 7/7, 15/4
}
{
 \draw[thick, fill=gray!75] (\x,\y) -- ( $(\x,\y) + (1/2, 0)$) --  ( $(\x,\y) + (1/2, 1/2)$) --  ( $(\x,\y) + (0, 1/2)$) -- cycle;
}

\begin{scope}[shift={(-1.25,1)}]

\foreach \x/\y in {
0/0,1/0,2/0
-.5/-.5,.5/-.5,1/-.5,1.5/-.5,2/-.5,2.5/-.5,
1/-1,2.5/-1,
16.25/0
}
{
 \draw[thick] (\x,\y) -- ( $(\x,\y) + (1/2, 0)$) --  ( $(\x,\y) + (1/2, 1/2)$) --  ( $(\x,\y) + (0, 1/2)$) -- cycle;
}

\node at (1.5, 1.25) {\small{$A$}};

\node[right] at (17,3.25) {\small{good cell in $\Gamma_{\max}$}};
\node[right] at (17,.33) {\small{bad cell}};

\end{scope}

\end{tikzpicture}
 \end{center}
 \caption{An obstruction $A$ with crucial vertices $v, v'$. These vertices have important vertices in distinct cells $c$ and $c'$. There is a short path between $c$ and $c'$ in $\Gamma_{\max}$, and we mark two vertices in each cell of this path.}
 \label{fig:two-crucial}
\end{figure}

We must also consider pairs of important vertices assigned to the obstruction.
Let  $c, c' \in \Gamma_{\max}$
be the cells that the two quadruples of important vertices lie in (these
important vertices  belong to different crucial vertices, so we might have
$c \neq c'$), as shown in Figure \ref{fig:two-crucial}.
This gives an upper bound on their distance, so by~\str{6} there is a (short)
path $\Pi$ between them. We orient all edges of the path
towards one of the endpoints, say $c$. Inside each cell of the path we assign
two vertices to $\Pi$. These vertices are considered marked.
A given good cell $c''$ is near at most one obstruction and the path is short,
so there are at most  ${6 \choose 2}$ such paths through $c''$. Thus at most,
we mark no more than $2 {6 \choose 2} = 30$ additional vertices in $c''$ (choosing
distinct vertices every time because $T$ is large). This completes our marking
of vertices for obstructions.


Next, we deal with safe vertices that are not in a cell of $\Gamma_{\max}$.
Each such vertex $v$ has at least $T$ neighbours inside some cell
$c \in\Gamma_{\max}$.
We assign $v$ arbitrarily to one  such $c$.
Next, for each cell $c$, we consider the set of safe points assigned to it.
We partition these safe vertices into at most {\grn 36} cliques by  centering a
$3r\times3r$ square
on the center of $c$, and dividing it into $(r/2)\times(r/2)$ squares.
 We will refer to each of these cliques as a {\em safe cluster}.
For each safe cluster we do the following. If it has more than six members,
we declare
six of them ``crucial''. 
For each such crucial vertex $w$, we pick four of
its neighbours in
$c$ (different from all marked vertices in $c$, since  $T$ is large).
We
declare these vertices to be \emph{important for $w$}.
If a safe cluster has at most six members, then we consider these vertices
to be singleton safe clusters. We assign four \emph{important} vertices in $c$
to each of these vertices. These important vertices are considered marked.
In the worst case (where every safe cluster size is at most six), we mark
$36 \cdot 6 \cdot 4 = 864$ vertices in $c$.

We have accounted for all vertices outside of $\Gamma_{\max}$, and marked
fewer than 1000 vertices in each good cell. The path games associate one or two paths to each obstruction or safe cluster, so we end up with at most 200 marked edges.
Finally, we address all unmarked vertices in cells of $\Gamma_{\max}$.
Inside each cell $c \in \Gamma_{\max}$, we partition the unmarked points
into sets $C_0(c)$, $C_1(c), \dots, C_{\ell}(c)$ where
$|C_i(c)| = N(10)$ for $i=1,\dots, \ell$,  and $|C_0(c)| > N(s)$, with
$N(10)$ and  $N(q)$ as in Lemma~\ref{lem:sblob}, where $\ell = 1000$ and $s=s(200,1000)$
 from Lemma \ref{lem:blobcrap5}  are both constants.
We can now specify our constant $T$, which must allow us to mark our distinct
vertices, and to make the blob cycles. Choosing $T = 1000 + \ell N(10) + N(s)$
is sufficient.

This completes Maker's organization of the graph. During the game, Maker plays
as follows:

\begin{enumerate}
\item Every edge $cc'$ {\red of $\cal T$}
has important vertices $x_1, x_2, x_3, x_4 \in c$ and
$y_1, y_2, y_3, y_4 \in c'$.
We pair the edges $x_1 y_1$ and $x_2 y_2$.  When Breaker claims one of them,
Maker responds by claiming the other one. Likewise, we pair edge $x_3 y_3$ and
$x_4 y_4$. Therefore Maker claims two independent edges between $c$ and $c'$.

\item For every crucial vertex, we pair the edges to the four important vertices
assigned to it. When Breaker claims one of them, Maker responds by claiming
the other one.
Therefore,  Maker claims at least two of these edges.

\item If Breaker claims an edge inside an obstruction $O$  together with the
crucial vertices $C$ assigned to that obstruction, then Maker responds
according to his winning strategy for
the corresponding $(a,b)$-path game (Lemma~\ref{lem:abpathgame}), where
$A=O$ and $B=C$.

\item If Breaker claims an edge inside a safe cluster $S$  together with its
important vertices $I$, Maker again responds according to his winning strategy
for the $(a,b)$-path game, this time with $A=S$ and $B=I$.

\item Suppose $c_1c_2$ is a directed edge of a path $\Pi$ between two cells
$c, c'$ containing
important vertices for an obstruction.
Let $u_1,u_2 \in c_1$ and $v_1, v_2\in c_2$ be the vertices assigned to $\Pi$.
Maker pairs $u_1v_1, u_1v_2$ and  pairs $u_2v_1, u_2v_2$.

\item Similarly if $u_1, u_2, u_3, u_4 \in c$ are important for an
obstruction, and
$v_1,v_2 \in c$ are assigned to a path $\Pi$ in $\Gamma_{\max}$ then
Maker pairs $u_iv_1, u_iv_2$ for $i=1,\dots,4$.

\item If Breaker claims an edge between two vertices of $C_i(c)$ for some
$1\leq i\leq \ell$ and $c\in\Gamma_{\max}$ then
Maker claims another edge in $C_i(c)$ according to his winning strategy for the
10-blob game (Lemma~\ref{lem:sblob}).

\item Likewise, if Breaker claims an edge inside $C_0(c)$ for some
$c\in\Gamma_{\max}$ then
Maker responds according to his winning strategy for the $s$-blob game
(Lemma~\ref{lem:sblob}).

\item If Breaker claims an edge $uv$ with $u \in c
\setminus (C_0(c)\cup\dots\cup C_{\ell}(c))$ for some $c\in\Gamma_{\max}$ and
$v \in C_i(c)$ then Maker claims another edge between $u$ and a vertex of
$C_i(c)$.

\item If Breaker claims an edge between $u \in C_0(c)$ and $v\in C_i(c)$ with
$c\in\Gamma_{\max}$ and
$1\leq i \leq \ell$, then Maker claim another edge $wv$ with $w \in C_0(c)$.

\item For any other Breaker move, Maker responds arbitrarily.

\end{enumerate}

We now prove that this is a winning strategy for Maker. We must show that
after the game ends, Maker's graph contains a Hamilton cycle.
\begin{enumerate}[{\bf P1}]
\item
First, by (iv),  Maker will have won the $(a,b)$-game for each safe cluster
$S$.
There will be either two important vertices (in the same good cell $c$) connected
by a  path that spans the safe cluster, or two pairs of such vertices connected by 
vertex disjoint paths that span the cluster.

\item
Similarly,  by (iii), Maker will have won the  $(a,b)$-game for every
obstruction.
In other words,  there will be one or two paths between pairs of crucial vertices that
span the obstruction. By (ii), each of these path extends to a path between important
vertices in two possibly different cells $c, c' \in \Gamma_{\max}$. By (v)
and (vi), these paths  also extend to a paths between two
marked vertices in the same cell $c\in\Gamma_{\max}$.
Also, every vertex that got marked crucial but is not part of such a path,
will be part of a path of length
two between two of its important vertices.

\item
Let $u,v$ be a pair of important vertices in the same good cell $c$ that are
the endpoints of a path through an obstruction or safe cluster. We treat $uv$
as a marked edge. We also have a pair of vertices marked to connect to each adjacent cell in $\Tcal$. The total number of marked edges is at most 200.
Meanwhile any marked vertex $w$ that is not part of one of
these paths (as either an endpoint or an interior point) is treated as a
marked singleton.

\item
Next, we consider the  unmarked vertices in a good cell $c$. Strategies (vii)
and (viii) ensure that Maker will have created the blob cycles
$C_0(c), C_1(c), \ldots, C_{\ell}(c)$. These cycles span all unmarked vertices
in $c$.
By (ix) and (x), the family of blob cycles and the marked edges and the marked
singletons satisfy the conditions of Lemma \ref{lem:blobcrap5}. Therefore, we
can construct one cycle that spans all of the vertices in $c$, and includes
every marked edge. At this point, we have a spanning cycle in every good cell,
where this cycle used marked edges.

\item
Let $uv$ be a marked edge in $c$. Recall that this marked edge $uv$ is a
placeholder for another structure in Maker's graph (in fact, Breaker may have
claimed the actual edge between these vertices). First, if the vertices $u,v$
were important for the spanning tree $\Tcal$ of $\Gamma_{\max}$, then there is
a corresponding marked edge $u'v'$ in the spanning cycle of $c'$. We replace
these two marked edges by edges $uu'$ and $vv'$, which merges the spanning
cycles in $c$ and $c'$.  Second, if the vertices $u,v$ were important for an
obstruction $O$, then we replace the marked edge $uv$ with the spanning path
through $O$ whose endpoints are $u$ and $v$. Third, we replace any marked edge
associated with a safe cluster with the analogous path through that cluster.
Once we have replaced all the marked edges, we have our Hamilton cycle.
\end{enumerate}
\end{proof}

\begin{proofof}{Theorem~\ref{thm:hamhit}}
The proof is very similar to that of Theorem~\ref{thm:connhit}.
It is clear that Breaker wins if there is a vertex of degree at most three
(Breaker starts).
Hence

\begin{equation}\label{eq:HamhitLB} \Pee\left( \rho_n(\text{Maker wins the
Hamilton cycle game}) \geq
\rho_n(\text{min.\,deg. $\geq 4$}) \right) = 1.
\end{equation}

\noindent
We now define:

\[
r_L(n) := \left(\frac{\ln n + 5\ln\ln n - K}{\pi n}\right)^{\frac12}, \quad
r_U(n) := \left(\frac{\ln n + 5\ln\ln n + K}{\pi n}\right)^{\frac12},
\]

\noindent
for $K$ a (large) constant. By Theorem~\ref{thm:penrosemindeg}, we can
choose  $K=K(\eps)$ such that

\begin{equation}\label{eq:hamhiteq}
\Pee{\Big(} r_L(n) \leq \rho_n(\text{min.\,deg. $\geq 4$}) \leq r_U(n) {\Big)}
\geq
1-\eps+o(1).
\end{equation}

\noindent
By Lemma~\ref{lem:structure} the properties~\str{1}--\str{6}
are satisfied with probability $1-o(1)$ by $V=\Xcal_n, m = m_n, T=O(1),
r=r_L(n)$ with
$\Xcal_n$ as given by~\eqref{eq:Xdef}, $m_n$ as given by~\eqref{eq:mdef}
and $r_L$ as above.
By Lemma~\ref{lem:crucial}, with probability $1-o(1)$,
the remaining conditions of Lemma~\ref{lem:hamgamedet} are met
for any $r\leq \rho\leq 2r$ with $\delta( G(V;\rho) )\geq 4$.
Hence:

\[ \Pee\left( \rho_n(\text{Maker wins the Hamilton cycle game}) =
\rho_n(\text{min.\,deg. $\geq 4$}) \right) \geq 1-\eps-o(1).
\]

\noindent
Sending $\eps \downarrow 0$ gives the theorem.
\end{proofof}

\section{The perfect matching game}
\label{sec:perfect}

We start by considering the obvious obstructions preventing Maker-win:
vertices $v$ with degree $d(v) \leq 1$, and edges $uv$ with edge-degree
$d(uv) = |(N(v) \cup N(u)) \backslash \{ u,v \} | \leq 2$. Indeed, with
Breaker first to move, he can isolate a vertex of degree one. For an edge with
at most two neighbouring vertices, Breaker can ensure that one vertex
remains unmatched by Maker.

\begin{lemma}\label{lem:perfecthitExpPois}
Let $(r_n)_n$ be such that
\[
\pi n r_n^2 = \ln n + \ln\ln n + x + o(1),
\]
for some $x\in \eR$.
Let $Z_n$ denote the number of vertices of degree exactly one plus the number
of edges of
edge-degree exactly two in $\GPo(n,r_n)$.
Then
\[
\Ee Z_n {\red \rightarrow} (1+\pi^2/8)e^{-x} + \sqrt{\pi}(1+\pi) e^{-x/2}
\]
as $n\to\infty$.
\end{lemma}

\begin{proof}
Let $Y_n$ denote the number of vertices of degree exactly one and let $W_n$
denote the
number of edges of edge-degree exactly two.
From Theorem~\ref{thm:penrosemindeg} and Lemma~\ref{lem:penrosemindeg2}, we
have
\begin{equation}\label{eq:EY}
\Ee Y_n = e^{-x} + \sqrt{\pi e^{-x}} + o(1).
\end{equation}

\noindent
It thus remains to compute $\Ee W$.
Let $\Wcnr$ denote the number of edge-degree two edges $\{X_i,X_j\}$
with both vertices $X_i, X_j$ within $100 r$ of two of the sides of  $[0,1]^2$.
By Lemma~\ref{lem:boundarydistancecount} {\red using $k=2$ and condition (iii)
(a)}, we have

\begin{equation}\label{eq:WcnrExp}
\Ee\Wcnr = O\left( n^{-O(1)} \right).
\end{equation}

\noindent
Let $\Wmdl$ denote the number of edge-degree two edges $\{X_i,X_j\}$
for which both $X_i X_j$ are at least $r$ removed from the boundary of
the unit square.
For $0 \leq z \leq r$, let us write
\[
\mu(z) = n \cdot \area( B(u;r) \cup B(v;r) ),
\]
where $u,v \in \eR^2$ are such that $\norm{u-v}=z$.
Let $\eps > 0$ be arbitrary.
By {\red equation \eqref{eq:areadiff22} of} Corollary~\ref{cor:areadiff2},
there is a $\delta=\delta(\eps)$ such that

\begin{equation}\label{eq:muzbds}
\begin{array}{c}
\pi n r^2 + (2-\eps)nrz \leq \mu(z) \leq \pi n r^2 + (2+\eps)nrz,   \\
\text{ and } \\
(1-\eps) \pi n r^2 \leq \mu(z) \leq (1+\eps) \pi n r^2,
\end{array}
\end{equation}

\noindent
for all $0 \leq z \leq \delta {\red r}$.
Let $\Wmdl^\delta$ denote the number of edge-degree two edges $\{X_i,X_j\}$
for which both $X_i, X_j$ are at least $r$ removed from the boundary of
the unit square, and $\norm{X_i-X_j} \leq \delta r$.
By Lemma~\ref{lem:boundarydistancecount} {\red using $k=2$ and condition (iii)
(c)},  we have

\begin{equation}\label{eq:WmdlMinWmdlEps}
 \Ee \left(\Wmdl-\Wmdl^\delta\right)
= O\left( n^{-O(1)}\right).
\end{equation}

\noindent
Let us now compute $\Wmdl^\delta$. We have

\begin{equation}\label{eq:WmdlEpsUB}
\begin{array}{rcl}
\Ee \Wmdl^\delta
& \leq &
\displaystyle
\frac12 \cdot n^2 \int_{[0,1]^2}\int_{B(v;\delta r)}
\frac{\mu(\norm{u-v})^{2}e^{-\mu(\norm{u-{\red v}})}}{2}
 {\dd}u{\dd}v \\
& = &
\displaystyle
\frac12 \cdot n^2 \int_0^{\delta r}
\frac{\mu(z)^{2}e^{-\mu(z)}}{2}
 2\pi z {\dd}z \\
& \leq  &
\displaystyle
\frac12 n^2 \int_0^{\delta r}
\frac{\left( (1+\eps) \pi n r^2 \right)^{2}
e^{ - \pi n r^2 - (2-\eps) n r z }}{2} 2\pi z{\dd}z \\
& = &
\displaystyle
(1+o(1)) \cdot \frac{(1+\eps)^2\pi}{2} \cdot n^2 \cdot \ln^2 n \cdot
e^{-\pi n r^2}
\int_0^{\delta r}
e^{ - (2-\eps) n r z } z{\dd}z \\
& = &
\displaystyle
(1+o(1)) \cdot \frac{(1+\eps)^2\pi e^{-x}}{2} \cdot n \cdot \ln n
\int_0^{\delta r}
e^{ - (2-\eps) n r z } z{\dd}z \\
& = &
\displaystyle
(1+o(1)) \cdot \frac{(1+\eps)^2\pi e^{-x}}{2} \cdot n \cdot \ln n \cdot
\left(\frac{1}{(2-\eps)nr}\right)^2 \int_0^{(2-\eps)\delta nr^2}
e^{ - y } y{\dd}y \\
& = &
\displaystyle
(1+o(1)) \cdot \frac{(1+\eps)^2\pi e^{-x}}{2(2-\eps)^2} \cdot
 {\red \frac{\ln n}{n r^2}} \cdot \int_0^{\infty}
e^{ - y } y{\dd}y \\
& = &
\displaystyle
(1+o(1)) \cdot \frac{(1+\eps)^2\pi^2 e^{-x}}{2(2-\eps)^2}.
\end{array}
\end{equation}

\noindent
Here the factor $\frac12$ in the first line comes from the fact that
Theorem~\ref{thm:palm}
applies to {\em ordered} pairs; to get the second line we switched to
polar coordinates; to get the third line
we applied the bounds~\eqref{eq:muzbds}; to get the fourth line we used
that $\pi n r^2 = (1+o(1)) \ln n$;
to get the fifth line we used that $\pi n r^2 = \ln n + \ln\ln n + x + o(1)$ so
that $e^{-\pi n r^2} = (1+o(1)) e^{-x} / n\ln n$; to get the sixth line we used
the change of variables $y = (2-\eps) n r z$; and in the last two lines we
used that $n r^2\to \infty$
so that $\int_0^{(2-\eps)\delta nr^2} e^{ - y } y{\dd}y \to
\int_0^{\infty} e^{ - y } y{\dd}y = 1$.

By analogous computations we get

\begin{equation}\label{eq:WmdlEpsLB}
\begin{array}{rcl}
\Ee \Wmdl
& \geq &
\Ee \Wmdl^\delta \\
& \geq &
\displaystyle
\frac12 \cdot n^2 \int_{[100r,1-100r]^2}\int_{B(v;\delta r)}
\frac{\mu(\norm{u-v})^{2}e^{-\mu(\norm{u-{\red v}})}}{2}
 {\dd}u{\dd}v \\
& = &
\displaystyle
{\red (1-o(1))} \cdot \frac{n^2}{2} \int_0^{\delta r}
\frac{\mu(z)^{2}e^{-\mu(z)}}{2}
 2\pi z {\dd}z \\
& \geq  &
\displaystyle
{\red (1-o(1))}  \cdot \frac{n^2}{2} \int_0^{\delta r}
\frac{\left( (1-\eps) \pi n r^2 \right)^{2}
e^{ - \pi n r^2 - (2+\eps) n r z }}{2} 2\pi z{\dd}z \\
& = &
\displaystyle
{\red (1-o(1))} \cdot \frac{(1-\eps)^2\pi^2 e^{-x}}{2(2+\eps)^2}.
\end{array}
\end{equation}

\noindent
Combining~\eqref{eq:WmdlMinWmdlEps},~\eqref{eq:WmdlEpsUB}
and~\eqref{eq:WmdlEpsLB} and sending $\eps\downarrow 0$, we find

\begin{equation}\label{eq:WmdlFinal}
 \Ee \Wmdl = (1+o(1))\cdot \frac{\pi^2 e^{-x}}{8}.
\end{equation}

Finally, we must consider the remaining four rectangles of width $r$ that
are adjacent to exactly one border.\mnote{Moved definition of
$\mu(w,z,\alpha)$ down.}
Let $\Wsde$ denote the number of edges $\{X_i,X_j\}$ with edge-degree two
such that at least one
of $X_i,X_j$ is no more than $r$ away from some side of the unit square, and
both are at least $100r$ away from all other sides;
and let $\Wsde^\delta$ denote all such pairs for which in addition
at least one of $X_i, X_j$ is no more than $\delta r$ away from a side of the
unit square,
and $\norm{X_i-X_j} \leq \delta r$.
By Lemma~\ref{lem:boundarydistancecount} {\red using $k=2$ with the union of
conditions (iii)(b) and (iii)(c)}, we have

\begin{equation}\label{eq:WsdeMinWsdeEps}
 \Ee\left(\Wsde-\Wsde^\delta\right) = O\left( n^{-O(1)} \right).
\end{equation}

For $0 \leq w, z \leq r$ and $-\pi/2 \leq \alpha \leq \pi/2$, let us write:

\[
\mu(w,z,\alpha) := n\cdot \area( [0,1]^2 \cap (B(u;r)\cup B(v;r))),
\]

\noindent
where $u,v\in [0,1]^2$ are such that $u_x = w < v_x, \norm{u-v} = z$
and the angle between $v-u$ and the
positive $x$-axis is $\alpha$ (see Figure~\ref{fig:dhalpha}).
Fix $\eps> 0$. By Lemma~\ref{lem:areafinal} there is a $\delta=\delta(\eps)$
such that
for all $0 \leq w, z\leq \delta r$ and all $-\pi/2\leq\alpha\leq\pi/2$:

\begin{equation}\label{eq:bounds23}
\begin{array}{c}
 \frac{\pi}{2} n r^2 + (1+\cos\alpha-\eps) nrz + (2-\eps)nwr \\
\leq \\
\mu(w,z,\alpha) \\
\leq \\
\frac{\pi}{2} n r^2 + (1+\cos\alpha+\eps) nrz + (2+\eps)nwr,
\end{array}
\end{equation}
and
\begin{equation}\label{eq:bounds24}
(1-\eps)\frac{\pi}{2}nr^2 \leq \mu(w,z,\alpha) \leq (1+\eps)\frac{\pi}{2}nr^2.
\end{equation}

Let $B^+(v;r)$ denote the set of all $p \in B(v;r)$ with $p_x \geq v_x$.
We can write

\[
\begin{array}{rcl}
 \Ee \Wsde^\delta
& {\red \leq } &
\displaystyle
4 \cdot n^2 \int_{[0,\delta r]\times [100r,1-100r]}\int_{B^+(v;r)}
\Pee\left(\Po([0,1]^2 \cap (B(u;r)\cup B(v;r)))=2\right) {\dd}u{\dd}v \\
& = &
\displaystyle
(1+o(1)) \cdot 4 n^2 \cdot \int_{-\pi/2}^{\pi/2}\int_0^{\delta r}
\int_0^{\delta r}
\frac{\left(\mu(w,z,\alpha)\right)^2}{2} e^{-\mu(w,z,\alpha)} {\red z} \,
{\dd}w{\dd}z{\dd}\alpha \\
& \leq &
\displaystyle
(1+o(1)) \cdot 4 n^2 \cdot  
\int_{-\pi/2}^{\pi/2} \int_0^{\delta r}\int_0^{\delta r}
\frac{\left((1+\eps)\frac{\pi}{2}nr^2\right)^2}{2}e^{-\frac{\pi}{2}nr^2
-(1+\cos\alpha-\eps)nrz -(2-\eps)nrw} {\red z} \,  {\dd}w{\dd}z{\dd}\alpha \\
& = &
\displaystyle
(1+o(1)) \cdot \frac{(1+\eps)^2}{2} \cdot n^2 \cdot \ln^2 n \cdot
e^{-\frac{\pi}{2} nr^2} \cdot 
\int_{-\pi/2}^{\pi/2}\int_0^{\delta r}\int_0^{\delta r}
e^{-(1+\cos\alpha-\eps)nrz - (2-\eps)nrw} {\red z} \,
{\dd}w{\dd}z{\dd}\alpha \\
& = &
\displaystyle
(1+o(1)) \cdot \frac{(1+\eps)^2}{2} \cdot n^{\frac32} \cdot \ln^{\frac32} n
\cdot e^{-x/2} \cdot 
\int_{-\pi/2}^{\pi/2}\int_0^{\delta r}\int_0^{\delta r}
e^{-(1+\cos\alpha-\eps)nrz - (2-\eps)nrw} {\red z} \,
{\dd}w{\dd}z{\dd}\alpha \\
& = &
\displaystyle
(1+o(1)) \cdot \frac{(1+\eps)^2e^{-x/2}}{2} \cdot n^{\frac32}
\cdot \ln^{\frac32} n \cdot
\left(\frac{1}{(2-\eps)nr}\right) \cdot
\left[ 1 - e^{-(2-\eps)\delta n r^2} \right] \cdot
\\
& &
\displaystyle
\int_{-\pi/2}^{\pi/2}\int_0^{\delta r}
e^{-(1+\cos\alpha-\eps)nrz} {\red z} \,  {\dd}z{\dd}\alpha \\
& = &
\displaystyle
(1+o(1)) \cdot \frac{(1+\eps)^2e^{-x/2} }{2(2-\eps)} \cdot n^{\frac32} \cdot
\ln^{\frac32} n \cdot
(nr)^{-1} \cdot
\\
& &
\displaystyle
\int_{-\pi/2}^{\pi/2} \left( \frac{1}{(1+\cos\alpha-\eps)nr} \right)^2
\int_0^{(1+\cos\alpha-\eps)\delta nr^2}
 {\red y } \, e^{-y}{\dd}y{\dd}\alpha \\
& = &
\displaystyle
(1+o(1)) \cdot \frac{(1+\eps)^2e^{-x/2}}{2(2-\eps)} \cdot n^{\frac32}
\cdot \ln^{\frac32} n \cdot
(nr)^{-3} \cdot
\int_{-\pi/2}^{\pi/2} \left(\frac{1}{1+\cos\alpha-\eps}\right)^2 {\dd}\alpha \\
\end{array}
\]

Here we have once again used Theorem~\ref{thm:palm} in the first line,
together with symmetry considerations.
{\red The first line gives an upper bound on $\Wsde^\delta$ since a vertex in
$B(u;r) \cap B(v;r)$ would increase the edge degree by 2 instead of 1.}
In the second line we switched to polar coordinates; in the third line we
used the bounds~\eqref{eq:bounds23} and~\eqref{eq:bounds24};
in the fourth line we used $\pi n r^2 = (1+o(1))\ln n$; in the fifth
line we used $\pi n r^2 = \ln n + \ln\ln n + x + o(1)$;
in the sixth line we integrated with respect to $w$; in the seventh
line we applied the substitution
$y = (1+\cos\alpha-\eps)nrz$; in the eight line we used that
$ \int_0^{(1+\cos\alpha-\eps)\delta nr^2}
{\red y} \, e^{-y}{\dd}y\to \int_0^{\infty} {\red y} \, e^{-y}y{\dd}y = 1$.
We get:

\begin{equation}\label{eq:WsdeEpsUB}
\begin{array}{rcl}
 \Ee \Wsde^\delta
& \leq &
\displaystyle
(1+o(1)) \cdot \frac{(1+\eps)^2e^{-x/2}}{2(2-\eps)} \cdot \ln^{\frac32} n \cdot
(nr^2)^{-\frac32} \cdot
\int_{-\pi/2}^{\pi/2} \left(\frac{1}{1+\cos\alpha-\eps}\right)^2 {\dd}\alpha \\
& = &
\displaystyle
(1+o(1)) \cdot \frac{(1+\eps)^2e^{-x/2}\pi^{\frac32}}{2(2-\eps)} \cdot
\int_{-\pi/2}^{\pi/2} \left(\frac{1}{1+\cos\alpha-\eps}\right)^2 {\dd}\alpha \\
\end{array}
\end{equation}

\noindent
Reversing the use of the upper and lower bounds from~\eqref{eq:bounds23}
and~\eqref{eq:bounds24}, and repeating the
computations giving~\eqref{eq:WsdeEpsUB} we find

\[
\begin{array}{rcl}
 \Ee\Wsde
& \geq &
\Ee\Wsde^\delta \\
& \geq &
\displaystyle
(1+o(1)) \cdot 4 n^2 \cdot  
\int_{-\pi/2}^{\pi/2} \int_0^{\delta r}\int_0^{\delta r}
\frac{\left((1-\eps)\frac{\pi}{2}nr^2\right)^2}{2}e^{-\frac{\pi}{2}nr^2
-(1+\cos\alpha+\eps)nrz -(2+\eps)nrw}{\dd}w{\dd}z{\dd}\alpha \\
& = &
\displaystyle
(1+o(1)) \cdot \frac{(1-\eps)^2e^{-x/2}\pi^{\frac32}}{2(2+\eps)} \cdot
\int_{-\pi/2}^{\pi/2} \left(\frac{1}{1+\cos\alpha+\eps}\right)^2 {\dd}\alpha \\
\end{array}
\]

\noindent
Combining this with~\eqref{eq:WsdeMinWsdeEps} and ~\eqref{eq:WsdeEpsUB} and
sending $\eps\downarrow 0$, we find (employing the dominated convergence
theorem to
justify switching limit and integral):
\[
 \Ee \Wsde
=
(1+o(1)) \cdot \frac{e^{-x/2}\pi^{\frac32}}{4} \cdot
\int_{-\pi/2}^{\pi/2} \left(\frac{1}{1+\cos\alpha}\right)^2 {\dd}\alpha.
\]
We compute
\[
 \int_{-\pi/2}^{\pi/2} \left(\frac{1}{1+\cos\alpha}\right)^2 {\dd}\alpha
=
\left[
\frac{\sin(x)}{2(\cos(x) + 1)} + \frac{\sin^3(x)}{6(\cos(x) + 1)^3}
\right]_{-\pi/2}^{\pi/2} = \frac{4}{3}.
\]
Hence

\begin{equation}\label{eq:WsdeFinal}
 \Ee\Wsde = (1+o(1)) \cdot \frac{e^{-x/2}\pi^{\frac32}}{3}.
\end{equation}

Combining~\eqref{eq:WcnrExp},~\eqref{eq:WmdlFinal} and~\eqref{eq:WsdeFinal}
shows

\[
 \Ee W =
\frac{\pi^2 e^{-x}}{8} + \frac{e^{-x/2}\pi^{\frac32}}{3} + o(1)
\]

Together with~\eqref{eq:EY} this proves the Lemma.
\end{proof}

The following lemma can be proved via
a straightforward adaptation of the proof of Theorem 6.6 in~\cite{PenroseBook}.
For completeness we provide a proof in the appendix.

\begin{lemma}\label{lem:headache}
Let $n \in \eN$ and $r>0$ be arbitrary.
Let $Z$ denote the number of vertices of degree exactly one in $G_P(n,r)$
plus the number of edges of edge-degree exactly two in $G_P(n,r)$.
Then
\[
\dtv( Z, \Po(\Ee Z) ) \leq 6 \cdot (I_1+I_2+I_3+I_4+I_5+I_6)
\]
where
\[ \begin{array}{l}
\displaystyle
I_1 := n^2 \int_{[0,1]^2}\int_{B(x;100r)} \Pee(E_x)\Pee(E_y){\dd}y{\dd}x, \\
\displaystyle
I_2 := n^2 \int_{[0,1]^2}\int_{B(x;100r)} \Pee(E_x^{y}, E_y^{x}){\dd}y{\dd}x, \\
\displaystyle
I_3 := n^4 \int_{A_3} \Pee(F_{x_1,y_1})\Pee(F_{x_2,y_2})%
{\dd}y_2{\dd}x_2{\dd}y_1{\dd}x_1, \\
\displaystyle
I_4 := n^4 \int_{A_3} \Pee(F_{x_1,y_1}^{x_2,y_2},F_{x_2,y_2}^{x_1,y_1})%
{\dd}y_2{\dd}x_2{\dd}y_1{\dd}x_1, \\
\displaystyle
I_5 := n^3 \int_{A_5}\Pee(F_{x_1,y_1})\Pee(E_{x_2})%
{\dd}x_2{\dd}y_1{\dd}x_1, \\
\displaystyle
I_6 := n^3 \int_{A_5}\Pee(F_{x_1,y_1}^{x_2}, E_{x_2}^{x_1,y_1})%
{\dd}x_2{\dd}y_1{\dd}x_1.
\end{array}\]

\noindent
Here

\[ \begin{array}{rcl}
A_3 := \{ (x_1,y_1,x_2,y_2) \in \left([0,1]^2\right)^4
& : &
y_1\in B(x_1,r), y_2\in B(x_2;r), \\
& &
x_2,y_2\in B(x_1;100r)\cap B(y_1;100r) \},
\end{array} \]

\[
A_5 := \{ (x_1,y_1,x_2) \in \left([0,1]^2\right)^3 :
y_1\in B(x_1,r), x_2\in B(x_1;100r)\cap B(y_1;100r) \},
\]

and $E_x$ denotes the event that one point of $\Pcal$ falls inside
$B(x;r)$, and $E_x^y$ the event that one point of $\{y\}\cup\Pcal$ falls
inside $B(x;r)$ and
$E_x^{y,z}$ denotes the event that one point of $\{y,z\}\cup\Pcal$ falls
inside $B(x;r)$; $F_{x,y}$ denotes the event that
two points of $\Pcal$ fall
inside $B(x;r)\cup B(y;r)$, and $F_{x,y}^{w}, F_{x,y}^{w,z}$ are defined
similarly
as $E_x^y, E_x^{y,z}$
\end{lemma}

\begin{lemma}\label{lem:perfecthitprobPois}
Let $(r_n)_n$ be such that
\[
\pi n r_n^2 = \ln n + \ln\ln n + x + o(1),
\]
for some $x\in \eR$.
Let $Z_n$ denote the number of vertices of degree {\bf exactly} one, plus
the number of edges of edge-degree {\bf exactly} two in $\GPo(n,r_n)$.
Then
\[
\Pee( Z_n = 0 ) \to  \exp\left[ - (1+\pi^2/8)e^{-x} + \sqrt{\pi}(1+\pi)
e^{-x/2} \right],
\]
as $n\to\infty$.
\end{lemma}

\begin{proof}
It suffices to show that if $I_1,\dots, I_6$ is as in
Lemma~\ref{lem:headache}, then
$I_1,\dots, I_6 \to 0$ as $n\to\infty$.
Observe that $I_1, I_5$ is most

\[ \begin{array}{rcl}
I_1, I_5
& \leq &
\Ee Z_n \cdot \pi n (100 r_n)^2 \cdot \Pee( \Po( \pi n r_n^2 / 4 ) \leq 1 ) \\
& = &
O\left( \ln n \cdot \exp[ -(\frac14+o(1)) \ln n \cdot H( 4/\pi n r_n^2 ) ]
\right) \\
& = &
O\left( n^{-\frac14+o(1)} \right),
\end{array} \]

\noindent
where we used Lemma~\ref{lem:chernoff} and the fact that for $r$ sufficiently
small
at least one quarter of $B(v;r)$ is contained in the unit square for all
$v\in[0,1]^2$.

Similarly

\[ \begin{array}{rcl}
I_3
& \leq &
\Ee Z_n \cdot \left( \pi n (100 r_n)^2 \right)^2 \cdot
\Pee( \Po( \pi n r_n^2 / 4 ) \leq 1 ) \\
& = &
O\left( n^{-\frac14+o(1)} \right).
\end{array} \]%

\noindent
Notice that $I_2$ is the expected number of (ordered) pairs of points $(u,v)$
with
$\norm{u-v} \leq 100 r_n$ and $d(u)=d(v)=1$.
Observe that if moreover $\norm{u-v} \leq r/100$, then $(u,v)$ are in fact an
$(0,0)$-pair.
Thus, using Lemma~\ref{lem:pairexp} and Lemma~\ref{lem:boundarydistancecount}:

\[
I_2  \leq
O\left( \ln^{-1} n \right)
+
O\left( n^{-c} \right)
 =
o(1).
\]

\noindent
Similarly, $I_4$ is equals the expected number of $4$-tuples of points
$(u_1,\dots,u_4)$
with all distances $\norm{u_i-u_j} \leq 100r$ and
$\norm{u_1-u_2}, \norm{u_3-u_4}\leq r$, and
$u_1u_2, u_3u_4$ each having edge-degree equal to two.
Observe if in such 4-tuple that $\norm{u_i-u_j} \leq r/100$ for all
$1\leq i,j\leq 4$,
then one of the pairs $(u_i,u_j)$ will in fact be a $(0,0,2)$-pair.
Also observe that for each $(0,0,2)$-pair contributes at most six 4-tuples to
the number of
our 4-tuples with all distances $\leq r/100$.
Thus, using again Lemma~\ref{lem:pairexp} and
Lemma~\ref{lem:boundarydistancecount}:

\[
I_4
\leq
6 \cdot O\left( \ln^{-4} n \right) + O\left( n^{-c}\right)
 =
o(1).
\]

Finally, $I_6$ equals the expected number of 3-tuples $(u_1,u_2,u_3)$
with all distances at most $100 r$, $\norm{u_1-u_2} \leq r$, the degree of $u_3$
equal to two, and
the edge-degree of $u_1u_2$ equal to two.
Observe that, if all distances in such a 3-tuple are $\leq r/100$
then one of $(u_1,u_2), (u_1,u_3), (u_2,u_3)$ will be
a $(0,1,1)$-pair.
Also, for every $(0,1,1)$-pair, there are at most three of our 3-tuples
with all distances $\leq r/100$.
Hence, again Lemma~\ref{lem:pairexp} and Lemma~\ref{lem:boundarydistancecount}:

\[
I_6 \leq
3 \cdot O\left( \ln^{-3} n \right) + O\left( n^{-c} \right)
 =
o(1).
\]

This shows that

\[ \dtv( Z, \Po(\Ee Z) ) = o(1). \]

So in particular

\[
 \Pee( Z = 0 ) = e^{-\Ee Z} + o(1) = e^{-(1+\frac{\pi^2}{8})e^{-x} +
 \sqrt{\pi}(1+\pi) e^{-x/2}} + o(1),
\]
using Lemma~\ref{lem:perfecthitExpPois}.
\end{proof}

\begin{corollary}\label{cor:perfecthitprobPo2}
Let $(r_n)_n$ be such that
\[
\pi n r_n^2 = \ln n + \ln\ln n + x + o(1),
\]
for some $x\in \eR$.
Let $Z_n$ denote the number of vertices of degree {\bf at most} one, plus
the number of edges of edge-degree {\bf at most} two in $\GPo(n,r_n)$.
Then
\[
\Pee( Z_n = 0 ) \to  \exp\left[ - ((1+\pi^2/8)e^{-x} +
\sqrt{\pi}(1+\pi) e^{-x/2}) \right],
\]
as $n\to\infty$.
\end{corollary}

\begin{proof}
Let $Y$ denote the number of vertices of degree exactly one; let $Y'$ denote
the number of vertices of degree exactly zero; let $W$ denote the number
of edges of
edge-degree exactly two, let $W'$ denote the number of edges
of edge-degree at most one, all in $\GPo(n,r_n)$.
The probability that $Y'>0$ can be read {\red off} from
Theorem~\ref{thm:penrosemindeg}.
By our choice of $r_n$, we have

\begin{equation}\label{eq:Yprime}
\Pee( Y' > 0 ) = o(1).
\end{equation}

Suppose that $uv$ is an edge of edge-degree at most one.
Then either $\norm{u-v} > r/100$, or $(u,v)$ is either a $(1,0,0)$-pair, a
$(0,1,0)$-pair,
a $(0,0,1)$-pair, or a $(0,0,0)$-pair.
By Lemmas~\ref{lem:boundarydistancecount} and~\ref{lem:pairexp}, we have

\begin{equation}\label{eq:Wprime}
\Pee( W' > 0 ) \leq \Ee W' = O\left( n^{-c} \right) +
O\left( \ln^{-1} n \right) = o(1).
\end{equation}

\noindent
Let $Z$ be as in the statement of Corollary.
By~\eqref{eq:Yprime},~\eqref{eq:Wprime} we have
\[
\Pee( Y+ W=0) - \Pee(Y'>0) - \Pee(W'>0) \leq  \Pee( Z = 0 ) \leq
\Pee( Y + W = 0 ),
\]
and hence $\Pee( Z=0 ) = \Pee( Y+W = 0 ) + o(1)$.
The corollary now follows directly from Lemma~\ref{lem:perfecthitprobPois}.
\end{proof}

Using Lemma~\ref{lem:configurations}, this last corollary immediately transfers also to the binomial case.
(We can rephrase $Z_n$ in terms of a measurable function $h_n(u,v,V)$ which equals one
only if either $u\neq v$ and the pair forms an edge of edge-degree at most two or $u=v$ and the degree is
at most one.)

\begin{corollary}\label{cor:perfecthitprobBin}
Let $(r_n)_n$ be such that
\[
\pi n r_n^2 = \ln n + \ln\ln n + x + o(1),
\]
for some $x\in \eR$.
Let $\Ztil_n$ denote the number of vertices of degree {\bf at most} one, plus
the number of edges of edge-degree {\bf at most} two in $G(n,r_n)$.
Then
\[
\Pee( \Ztil_n = 0 ) \to  \exp\left[ - ((1+\pi^2/8)e^{-x} +
\sqrt{\pi}(1+\pi) e^{-x/2}) \right],
\]
as $n\to\infty$.
\end{corollary}

\subsection{The proof of Theorem~\ref{thm:perfecthit}}

We will again introduce an auxiliary game that will be helpful for the
analysis of the
perfect matching game on the random geometric graph. As in the (a,b) path game, 
The {\em $(a,b)$ matching game} is played on the same graph $G_{a,b}$ as the $(a,b)$ path game in Section \ref{sec:pathgames}.
The vertices of  $G_{a,b}$ are
partitioned into  sets $A, B$ with $|A|=a, |B|=b$, where the only missing edges are the internal edges of $B$.
Maker's objective is to create a matching that saturates all vertices in $A$
(he does not care about the vertices in $B$). 
When we use this lemma for the perfect matching game, the $a$ vertices will belong to an obstruction $A$ and the $b$ vertices will be important for $A$.

\begin{lemma}\label{lem:abmatchinggame}
The $(a,b)$ matching game is a win for Maker if one of the
following conditions is met
\begin{enumerate}
\item\label{itm:match.bgeq4} $b \geq 4$, or;
\item\label{itm:match.a=3} $a \in \{2,3\}$ and $b \geq 3$, or;
\item\label{itm:match.a=1} $a=1$ and $b\geq 2$.
\end{enumerate}
\end{lemma}

\begin{proofof}{part~\ref{itm:match.bgeq4}}
If $b \geq 4$ then a winning strategy for Maker is as follows:

\begin{itemize}
\item Whenever Breaker plays an edge between two vertices of $A$ then
Maker responds claiming another edge inside $A$ according to the
strategy from Lemma~\ref{lem:pathplusiso} that will guarantee him
that by the end of the game he will have a path in $A$ that contains all
but one vertex of $A$.
\item If Breaker claims an edge between a vertex $u\in A$ and a vertex
$v\in B$ then
Maker claims an arbitrary unclaimed edge connecting $u$ to a vertex in $B$.
\end{itemize}

\noindent
If Maker cannot claim such an edge, then he claims an arbitrary edge (and we forget about it for the remainder of the game).
At the end of the game, Maker's graph
contains a path $P$ through all but one vertex of $A$, and every vertex of
$A$ will have
at least two neighbours in $B$.
Thus, the path $P$ contains a matching that covers all but at most two
points of $A$, and
the remaining (up to) two points can be covered by (at most two) vertex
disjoint edges to
vertices in $B$.

\noindent
{\bf Proof of part~\ref{itm:match.a=3}:}
If $a=3$ then the strategy just outlined in the proof of
part~\ref{itm:match.bgeq4} also
works.
This time, at the end of the game, Maker's graph contains an edge between two
vertices of $A$, and
the remaining vertex of $A$ has at least one neighbour in $B$.

Let us thus consider the case when $a=2$.
We can write $A = \{a_1, a_2\}$ and $B = \{b_1,\dots,b_\ell\}$ with
$\ell\geq 3$.
First suppose that in Breaker's first move he does not claim
$a_1a_2$. In that case Maker can claim $a_1a_2$ in his first move and win
the game.

Hence we can assume that in his first move, Breaker claims $a_1a_2$.
Maker now responds by claiming $a_1b_1$ in his first move and for the
remainder
of the game pairs the edges $a_2b_2, a_2b_3$, meaning that
if Breaker plays one of these two edges then he claims the other
(otherwise he plays arbitrarily).
This way, Maker will clearly end up with a matching of the required type
in the end.

\noindent
{\bf Proof of part~\ref{itm:match.a=1}:} Maker wins in the obvious way.
\end{proofof}

\begin{lemma}\label{lem:matchgamedet}
There exists a value of $T$ such that the following holds.
If $V \subseteq [0,1]^2, m \in \eN, r > 0$ and $r \leq \rho \leq 10^5r$ are
such that
\begin{enumerate}
\item \str{1}--\str{5} hold for with respect to $r$, and;
\item For all $3 \leq s \leq T$, every $s$-obstruction with respect to $r$ has
at least $s$ crucial vertices, and;
\item Every $(\geq T)$-obstruction with respect to $r$ has at least four
crucial vertices,
and;
\item Every edge of $G(V,\rho)$ has edge-degree at least three, and;
\item $G(V,\rho)$ has minimum degree at least two.
\end{enumerate}
Then Maker wins the perfect matching game on $G(V;\rho)$.
\end{lemma}

\begin{proof}
The proof is a relatively straightforward adaptation of the proof of
Lemma~\ref{lem:hamgamedet}.
In fact it is slightly simpler.
This time, on each obstruction or safe cluster we play the corresponding
$(a,b)$-matching game
(which we win by Lemma~\ref{lem:abmatchinggame}).
This will give a {\red matching} $M_0$ that saturates all vertices outside cells of
$\Gamma_{\max}$ and some vertices of $\Gamma_{\max}$.
A small change in {\bf P4} will show that we will have a Hamilton cycle $H$
through the remaining vertices $R$ of $\Gamma_{\max}$.
Indeed, we only mark the edges corresponding to vertices that are important 
for the spanning tree ${\cal T}$ of $\Gamma_{\max}$. Second, we only mark the 
singletons corresponding to important vertices that are not matched in
the various $(a,b)$ games. Then when we apply Lemma \ref{lem:blobcrap5} to merge
the blob-cycles and the marked singletons viz. 
the important vertices not covered by $M_0$.

$|R|$ is even, 
since the total number of points $n$ is even.  Thus the Hamilton cycle
$H$ yields a matching $M_1$ covering all of $R$. We take $M_0\cup M_1$
as our perfect matching. 
\end{proof}

\begin{proofof}{Theorem~\ref{thm:perfecthit}} This is again a straightforward
adaptation of the proof
of Theorem~\ref{thm:connhit}. 
It is clear that Breaker wins if there is a vertex of degree at most one
or an edge of degree at most two.
(Recall that Breaker starts the game.)
Hence

\begin{equation}\label{eq:PMhitLB} 
\Pee\left[ \rho_n(\text{Maker wins}) \geq 
\rho_n(\text{min.deg. $\geq 2$ and 
min.edge-deg. $\geq 3$}) \right] = 1.
\end{equation}

\noindent
We now define:

\[
r_L(n) := \left(\frac{\ln n + \ln\ln n - K}{\pi n}\right)^{\frac12}, \quad
r_U(n) := \left(\frac{\ln n + \ln\ln n + K}{\pi n}\right)^{\frac12},
\]

\noindent
for $K$ a (large) constant. By Lemma~\ref{lem:perfecthitExpPois}, we can
choose  $K=K(\eps)$ such that

\begin{equation}\label{eq:pmmhiteq}
\Pee{\Big[} r_L(n) \leq \rho_n(\text{min.deg. $\geq 2$ and 
min.edge-deg. $\geq 3$}) \leq r_U(n) {\Big]}
\geq
1-\eps+o(1).
\end{equation}

\noindent
By Lemma~\ref{lem:structure} the properties~\str{1}--\str{5}
are satisfied with probability $1-o(1)$ by $V=\Xcal_n, m = m_n, T=O(1),
r=r_L(n)$ with
$\Xcal_n$ as given by~\eqref{eq:Xdef}, $m_n$ as given by~\eqref{eq:mdef}
and $r_L$ as above.
By Lemma~\ref{lem:crucial},  with probability $1-o(1)$,
the remaining conditions of Lemma~\ref{lem:matchgamedet} are met
for any $r\leq \rho\leq 2r$ with minimum vertex degree at least
two and minimum edge degree at least three.
Hence:

\[ \begin{array}{c}
\Pee{\Big [} \rho_n(\text{Maker wins the perfect matching game}) =
\rho_n(\text{min. deg. $\geq 2$ and 
min. edge-deg. $\geq 3$}) {\Big]} \\
\geq \\
1-\eps-o(1).
\end{array} \]

\noindent
Sending $\eps \downarrow 0$ gives the theorem.
\end{proofof}

\section{The $H$-game}
\label{sec:H-game}

\begin{proofof}{Theorem~\ref{thm:Hgame}}
The proof of Theorem~\ref{thm:Hgame} is a bit different from the previous proofs.
Let $H$ be a fixed connected graph and let $k$ denote the least $\ell$ for which the $H$-game is Makers
win on an $\ell$-clique; let $\Fcal$ denote the family of all non-isomorphic graphs on $k$ vertices for which
the game is Maker-win.  Let $H_1, \dots, H_m \in \Fcal$ be those graphs in $\Fcal$ that can actually be realized as a 
geometric graph.
Observe that, since $H$ is connected and Maker cannot win on any graph on $<k$ vertices, each $H_i$ is connected.

Let $\eps>0$ be arbitrary and let $K>0$ be a large constant, to be chosen later.
Let us set

\[ r_U := K \cdot n^{-k/2(k-1)}. \]

\noindent
It follows from Theorem~\ref{thm:smallsubgraphs} that

\[ \begin{array}{rcl}
\Pee( G(n, r_U) \text{ contains a subgraph $\in\Fcal$}  )
& = & 
1 - \exp\left[ - K^{2(k-1)} \cdot \sum_{i=1}^m \mu(H_i) \right] + o(1) \\
& \geq & 
1 - \eps + o(1),
\end{array} \]

\noindent
where $\mu(.)$ is as defined in~\eqref{eq:muGdef} 
and the last inequality holds if we assume (without loss of generality) that $K$ was chosen sufficiently large.

Observe that if $G(n,r_U)$ contains a component of order $>k$, then 
there exist $k+1$ points $X_{i_1}, \dots, X_{i_{k+1}}$ such that 
$\norm{ X_{i_1} - X_{i_j} } \leq (k+1)r_U$ for all
$2 \leq j \leq k+1$.
This gives

\[ 
\begin{array}{rcl}
\Pee( G(n, r_U) \text{ has a component of order $> k$} ) 
& \leq & 
n^{k+1} \cdot {\big(} (k+1)r_U {\big)}^{2k} \\
& = & 
O\left( n^{k+1 - \frac{k^2}{k-1}} \right) \\
& = & 
O\left( n^{-k/(k-1)} \right) \\
& = & 
o(1). 
\end{array} \]

\noindent
Let $E$ denote the event that $G(n,r_U)$ contains a subgraph $\in\Fcal$ but no component on $>k$ vertices.
Then it is clear that 

\[ \Pee\left[ \rho_n(\text{Maker wins the $H$-game}) = 
\rho_n(\text{contains a subgraph $\in\Fcal$}) \right]
\geq \Pee( E ) \geq 1 - \eps -o(1), \]

\noindent
since, if all components have order $\leq k$ then Maker only wins if and only if there
is a component $\in\Fcal$.
Sending $\eps\downarrow 0$ proves the theorem.
\end{proofof}

\begin{proofof}{Corollary~\ref{cor:Hprob}}
This follows immediately from Theorem~\ref{thm:Hgamedet} and Theorem~\ref{thm:smallsubgraphs}.
\end{proofof}

\section{Conclusion and further work}
\label{sec:conclusion}

In the present paper, we explicitly determined the hitting radius for the games of connectivity, perfect matching and Hamilton cycle, all played on the edges of the random geometric graph. As it turns out in all three cases, the hitting radius for $G(n,r)$ to be Maker-win coincides exactly with a simple, necessary minimum degree condition. For the connectivity game it is the minimum degree two, in the case of the perfect matching game it is again the minimum degree two accompanied by the minimum edge degree three, and for the Hamilton cycle game we have the minimum degree four. Each of these characterizations engenders an extremely precise description of the behavior at the threshold value for the radius. We also state a general result for the $H$-game, for a fixed graph $H$, where the hitting radius can be determined by finding the smallest $k$ for which the $H$-game is Maker-win on the $k$-clique edge set and finding the list of all connected graphs on $k$ vertices which are Maker-win.

These results are curiously similar to the hitting time results obtained for the Maker-win in the same three games played on the Erd\H{o}s-R\'enyi random graph process. In that setting, in the connectivity game the hitting time for Maker-win is the same as for the minimum degree two~\cite{StojakovicSzabo05}, the same holds for the perfect matching game~\cite{BFHK11}, and the condition changes to minimum degree four for the Hamilton cycle game~\cite{BFHK11}. As we can see, in the case of the connectivity game and the Hamilton cycle game the conditions are exactly the same as for the random geometric graph. The difference in the condition for the perfect matching game is not a surprise, as the existence of an induced 3-path (which clearly prevents Maker from winning) at the point when the graph becomes minimum degree two is an unlikely event in the Erd\H{o}s-R\'enyi random graph process, but it does happen with positive probability in the random geometric graph.

With respect to the $H$-game much less is known on the Erd\H{o}s-R\'enyi random graph process. The only graph for which we have a description of the hitting time is the triangle---Maker wins exactly when the first $K_5$ with one edge missing appears~\cite{MuSt}. On the random geometric graph we basically have the same witness of Maker's victory, as the smallest $k$ for which Maker can win the triangle game on edges of $K_k$ is $k=5$. Interestingly, for most other graphs $H$ it is known that a hitting time result for Maker-win in the Erd\H{o}s-R\'enyi random graph process cannot involve the appearance of a finite graph on which Maker can win---Maker-winning strategy must be of ``global nature''~\cite{MuSt,NeStSt}. This is in contrast to the results we obtained in Theorem~\ref{thm:Hgame}, showing that on the random geometric graph Maker can typically win the $H$-game by simply spotting a copy of one of some finite list of graphs and restricting his attention to that subgraph.

\medskip

Playing a game on a random graph instead of the complete graph can be seen as a help to Breaker, as the board of the game becomes sparser, there are fewer winning sets and consequently Maker finds it harder to win. A standard alternative approach is to play the biased $(1:b)$ game on the complete graph, where Breaker again gains momentum when $b$ is increased. Naturally, one can combine the two approaches, playing the biased game on a random graph. This has been done in~\cite{StojakovicSzabo05,FGKN} for the Erd\H{o}s-R\'enyi random graph, where the threshold probability for Maker-win in the biased $(1:b)$ game, with $b=b(n)$ fixed, was sought for several standard positional games on graphs. The same question can be asked in our random geometric graph setting.

\begin{question}
 Given a bias $b$, what can be said for the smallest radius $r$ at which Maker can win the $(1:b)$ biased game, for the games of connectivity, perfect matching, Hamilton cycle, and the $H$-game?
\end{question}

In this paper we have only considered the random geometric graph constructed on points taken uniformly at random on the unit square, using
the euclidean norm to decide on the edges.
So other obvious directions for further work are:

\begin{question}
What happens for Maker-Breaker games on 
random geometric graphs in dimensions $d\geq 3$?
What happens if we use other probability distributions or norms in two dimensions?
\end{question}

\bibliographystyle{plain}
\bibliography{ReferencesMakerBreakerRGG}

\appendix

\section{The proof of Theorem~\ref{thm:palm}\label{sec:Palmapp}}

\begin{proofof}{Theorem~\ref{thm:palm}}
We condition on $N=m$. For convenience, let us write
$\Xcal_m := \{ X_i : 1 \leq i \leq m\}$.
We have
\[\begin{array}{rcl}
\Ee Z
& = &
\displaystyle
\sum_{m=k}^\infty \Ee\left[ Z | N=m \right] \cdot \Pee( N = m ) \\
& = &
\displaystyle
\sum_{m=k}^\infty
(m)_k \cdot \Ee\left[ h( X_1,\dots,X_k; \Xcal_m ) \right] \cdot \Pee( N=m ) \\
& = &
\displaystyle
\sum_{m=k}^\infty
(m)_k \cdot \Ee\left[ h( Y_1,\dots,Y_k; \{Y_1,\dots,Y_k\} \cup \Xcal_{m-k} )
\right]
\cdot \frac{n^m e^{-n}}{m!} \\
& = &
\displaystyle
n^k \cdot \sum_{j=0}^\infty
\Ee\left[ h( Y_1,\dots,Y_k; \{Y_1,\dots,Y_k\} \cup \Xcal_j ) \right] \cdot
\frac{n^j
e^{-n}}{j!} \\
& = &
\displaystyle
n^k \cdot \Ee\left[ h( Y_1,\dots,Y_k; \{Y_1,\dots,Y_k\} \cup \Pcal ) \right],
  \end{array}\]
as required.
\end{proofof}

\section{The proof of Lemma~\ref{lem:configurations}\label{sec:configurations}}

\begin{proofof}{Lemma~\ref{lem:configurations}}
Let $\eps>0$ be arbitrary.
By Markov's inequality and the fact that $\Ee Z_n = O(1)$, there exists a
constant
$K>0$ such that

\[
 \Pee( Z_n > K ) \leq \Ee Z_n / K \leq \eps.
\]

\noindent
Let $N$ be the $\Po(n)$-distributed random {\red variable} used in the
definition of $\Pcal_n$.
By Chebyschev's inequality we have
\[
\Pee\left[ |N-n| > K\sqrt{n} \right] \leq \Var(N) / (K\sqrt{n})^2 = 1/K^2 <
\eps.
\]
(We can assume without loss of generality that $1/K^2 < \eps$.)

Also observe that, since $\pi n r^2 = o(\sqrt{n})$, there exists a sequence $f(n) = o(\sqrt{n})$ such that  

\[ \Pee{\big[}\Delta( G(n+K\sqrt{n}, r_n) ) > f(n) {\big]} = o(1), \]

\noindent
where $\Delta(.)$ denotes the maximum degree. 
(This can for instance be seen from 
known results such as Theorem 2.3 in~\cite{cmcdplane}, or by a first moment argument using the Chernoff bound.)

For $m \in \eN$ let us call a tuple $(X_{i_1},\dots,X_{i_k}) \in \Xcal_m^k$ a {\em configuration} if
$h_n( X_{i_1},\dots,X_{i_k}; \Xcal_m) = 1$ and let $Y_m$ denote the number of configurations in 
$\Xcal_m$.
Let pick an arbitrary $t \leq K$ and an arbitrary pair $n-K\sqrt{n} \leq m < m' \leq n+K\sqrt{n}$.

We have

\begin{equation}\label{eq:ZOZP}
\begin{array}{rcl}
\Pee( \Ztil_n \neq Z_n )
& = &
\sum_{m=0}^\infty\sum_{t=0}^\infty \Pee( \Ztil \neq t, Z = t | N=m )
\Pee( N = m ) \\
& \leq &
\sum_{m=n-K\sqrt{n}}^{n+K\sqrt{n}}\sum_{t=0}^K
\Pee( \Ztil \neq t, Z = t | N=m ) \Pee( N = m ) 
\\
& &
+ \Pee( |N-n| > K\sqrt{n} ) + \Pee( Z > K )\\
& \leq &
\sum_{m=n-K\sqrt{n}}^{n+K\sqrt{n}}\sum_{t=0}^K \Pee( \Ztil \neq t, Z=t | N=m )
\Pee( N=m ) + 2\eps  \\
& = & 
\sum_{m=n-K\sqrt{n}}^{n+K\sqrt{n}}\sum_{t=0}^K \Pee( Y_n \neq t, Y_m=t ) 
\Pee( N=m ) + 2\eps
\end{array}
\end{equation}

\noindent
Let us now fix some $n-K\sqrt{n} \leq m \leq n$ and $t \leq K$.
Note that if $Y_m = t$ and $Y_n < t$ then we can fix $t$ configurations in $\Xcal_m$, and
at least one of the points $X_{m+1}, \dots, X_n$ must fall within $r$ of one of the $k \cdot t$ 
points of these configurations.
This gives:

\begin{equation}\label{eq:mleqn1} 
\Pee( Y_m = t, Y_n < t ) \leq  \Pee( Y_n < t | Y_m = t ) \leq k \cdot t \cdot r_n^2 \cdot (n-m) = o(1), 
\end{equation}

\noindent
since $r_n = o( n^{-1/2} )$ and $n-m = O(\sqrt{n})$.

Similarly, if $Y_m=t$ and $Y_n > t$ then we can fix $t+1$ configurations in $\Xcal_n$, and
at least one of the points that is either part of one of these configurations, or
within distance $r$ of a point of one of these configurations must be among
$X_{m+1}, \dots, X_n$.
This gives

\begin{equation}\label{eq:mleqn2} 
\begin{array}{rcl}
\Pee( Y_m = t, Y_n > t )
& \leq &
\Pee( \Delta( G(n+K\sqrt{n}, r_n) ) > f(n) ) \\
& & 
+ \Pee( Y_m = t, Y_n > t, \Delta( G(n+K\sqrt{n}, r_n) ) \leq f(n)  ) \\
& \leq &  
o(1)
+
\Pee( Y_m = t | Y_n > t, \Delta( G(n+K\sqrt{n}, r_n) ) \leq f(n)  ) \\
& \leq & 
o(1)
+ k \cdot (t+1) \cdot (f(n)+1) \cdot \frac{n-m}{n} \\
& = & 
o(1) ,
\end{array}
\end{equation}

\noindent
using $n-m = O(\sqrt{n})$ and $f(n) = \sqrt{n}$ for the last line. 
Combining~\eqref{eq:mleqn1} and~\eqref{eq:mleqn2} shows that
$\Pee( Y_m = t, Y_n \neq t ) = o(1)$ for all $t \leq K$ and $n-K\sqrt{n} \leq m \leq n$.

Similarly, we find that $\Pee( Y_m = t, Y_n \neq t ) = o(1)$ for all $t \leq K$ and $n\leq m \leq n+K\sqrt{n}$.
Using~\eqref{eq:ZOZP} we now find that

\[ 
\Pee( \Ztil_n \neq Z_n )
\leq 
\sum_{m=n-K\sqrt{n}}^{n+K\sqrt{n}}\sum_{t=0}^K \Pee( Y_n \neq t, Y_m=t ) 
\Pee( N=m ) + 2\eps  
 = 
o(1) + 2\eps.
\]

\noindent
Sending $\eps\downarrow 0$ completes the proof.
\end{proofof}

\section{The proof of Lemma~\ref{lem:span}\label{sec:spanapp}}

\begin{proofof}{Lemma~\ref{lem:span}}
Consider a spanning tree $T$ of $G$ that minimizes the sum of the edge lengths.
Then $T$ does not have any vertex of degree $\geq 7$.
This is because, if $v$ were to have degree $\geq 7$, then there are two
neighbours
$u,w$ of $v$
such that the angle between the segments $[v,u]$ and $[v,w]$ is strictly
less than
60 degrees.
We can assume without loss of generality that $[v,u]$ is shorter than $[v,w]$.
Note that if we remove the edge $vw$ and add the edge $uw$ then we obtain
another
spanning tree but with strictly smaller total edge-length, a contradiction.
Hence $T$ has maximum degree at most 6.

Similarly, we see that if $v$ has degree 6 in $T$ then the neighbours of $v$
have
exactly the same distance \onote{Corrected typo}
to $v$ and the angle between a pair of {\orange consecutive} neighbours is exactly 60
degrees.

Let us now pick a spanning tree $T'$ which minimizes the total edge-length
(and hence has no degree $\geq 7$ vertices)
and, subject to this, has as few as possible degree 6 vertices and, subject
to these two demands, the maximum over all degree 6 vertices of
their $x$-coordinate is as large as possible.
Let $v$ be a degree 6 vertex with largest $x$-coordinate. It has two neighbours
$u,w$ with strictly
large $x$-coordinates.  Observe that, as seen above the segments
$[v,u], [v,w], [u,w]$ all have the same
length. Let us remove the edge $uv$ and add $uw$.
This results in another spanning tree, with the same total edge length.
The degree of $u$ has not changed, the degree of $v$ has dropped, and the
degree of $w$ has increased by one.
If the degree of $w$ has become 6 then
our spanning tree is not as we assumed (there is a degree 6 vertex with
$x$-coordinate strictly
larger than that of $v$). Thus the degree of $w$ after this operation is $<6$.
But then the number of degree 6 vertices has decreased, contradiction.
\end{proofof}

\section{The proof of Lemma~\ref{lem:headache}}

\noindent
If $G = (V,E)$ is a graph and $\overline{Z} = (Z_v : v \in V)$ are random
variables, then
$G$ is a {\em dependency graph} for $\overline{Z}$ if, whenever
there is no edge between $A, B \subseteq V$, the
random vectors $\overline{Z_A} := (Z_v : v\in A)$ and
$\overline{Z_B} := (Z_v : v \in B)$
are independent.
The proof of Lemma~\ref{lem:headache} relies on the following result of
Arratia et al.~\cite{twomoments}:

\begin{theorem}[Arratia et al.~\cite{twomoments}]\label{thm:twomoments}
Let $(Z_v : v \in V)$ be a collection of Bernouilli random variables with
dependency graph
$G=(V,E)$. Set $p_v := \Ee Z_v = \Pee( Z_v = 1 )$ and
$p_{uv} := \Ee Z_uZ_v = \Pee( Z_v=Z_u=1 )$.
Set $W := \sum Z_v$, and $\lambda := \Ee W$.
Then
\[
\dtv( W, \Po(\lambda) )
\leq
\min(3,1/\lambda) \cdot \left( \sum_{v\in V(G)} \sum_{u \in N(v)} p_{uv} +
\sum_{v\in V(G)} \sum_{u\in N(v)\cup\{v\}} p_up_v \right).
\]
\end{theorem}

\begin{proofof}{Lemma~\ref{lem:headache}}
We will use Theorem~\ref{thm:twomoments}.
We consider the dissection $\Dcal(m)$ for some $m\in\eN$.
For convenience, let us order the cells of $\Dcal(m)$ in some (arbitrary) way as
$c_1,c_2,\dots, c_{m^2}$.
For $1\leq i \leq m^2$, let us denote by $x_i$ the lower left-hand corner of
$c_i$, and let
$\xi_i$ denote the indicator variable defined by;
\[
\xi_i := 1_{\left\{ \Pcal(c_i)=1, \Pcal(B(x_i;r)\setminus c_i) = 1 \right\}}.
\]
For $1\leq i < j \leq m^2$ with $\norm{x_i-x_j}\leq r$, let us write
\[
\xi_{(i,j)} := 1_{\left\{ \Pcal(c_i)=\Pcal(c_j)=1,
\Pcal((B(x_i;r)\cup B(x_j;r))\setminus (c_i\cup c_j)) = 2 \right\}}.
\]
Let us set
\[
Y^m := \sum_{i=1}^{m^2} \xi_i, \quad
W^m := \sum_{1\leq i < j \leq m^2, \atop \norm{x_i-x_j}\leq r} \xi_{i,j},
\]
and $Z^m := Y^m + W^m$.
By construction, we have%
\begin{equation}\label{eq:PoConverge}
 \lim_{m\to\infty} Y^m = Y, \quad, \lim_{m\to\infty} W^m = W,
 \quad \lim_{m\to\infty} Z^m = Z.
\end{equation}

\noindent
For convenience let us write $\Ical^m := [m^2], \Jcal^m := {[m^2]\choose 2}$ and
$\Vcal^m := \Ical^m \cup \Jcal^m$.
We now define a graph $\Gcal^m$ with vertex set $\Vcal_m$ and edges:
\begin{itemize}
 \item $ij\in E(\Gcal^m)$ if $i,j\in\Ical^m, \norm{x_i-x_j}\leq 100r$;
\item $uv\in E(\Gcal^m)$ if $u=\{i_1,i_2\}, v=\{i_3,i_4\} \in \Jcal^m$ and
$\norm{x_{i_a}-x_{i_b}} \leq 100r$ for all $1\leq a,b\leq 4$;
\item $uv\in E(\Gcal^m)$ if $u = i_1 \in\Ical^m, v=\{i_2,i_3\}\in \Jcal^m$
and $\norm{x_{i_a}-x_{i_b}} \leq 100r$ for all $1\leq a,b\leq 3$;
\end{itemize}
By the spatial independence properties of the Poisson process, this defined
a dependency graph on the random indicator variables $(\xi_v : v \in \Vcal^m )$.
By Theorem~\ref{thm:twomoments}, we therefore have that

\begin{equation}\label{eq:TwoMoments}
\dtv( Z^m, \Ee( Z^m ) ) \leq
3 \cdot \left( \sum_{v \in \Vcal^m}\sum_{u \in N(v)} \Ee \xi_v\xi_u
+  \sum_{v \in \Vcal^m}\sum_{u \in N(v)} \Ee \xi_v\Ee \xi_u \right).
\end{equation}

\noindent
Let us write

\[ \begin{array}{l}
\displaystyle
S_1^m := \sum_{v \in \Ical^m}\sum_{u \in N(v)\cap\Ical^m} \Ee \xi_v\Ee \xi_u,
\quad
S_2^m := \sum_{v \in \Ical^m}\sum_{u \in N(v)\cap\Ical^m} \Ee \xi_v\xi_u, \\
\displaystyle
S_3^m := \sum_{v \in \Jcal^m}\sum_{u \in N(v)\cap\Jcal^m} \Ee \xi_v\Ee \xi_u,
\quad
S_4^m := \sum_{v \in \Jcal^m}\sum_{u \in N(v)\cap\Jcal^m} \Ee \xi_v\xi_u, \\
\displaystyle
S_5^m := \sum_{v \in \Jcal^m}\sum_{u \in N(v)\cap\Ical^m} \Ee \xi_v\Ee \xi_u,
\quad
S_6^m := \sum_{v \in \Jcal^m}\sum_{u \in N(v)\cap\Ical^m} \Ee \xi_v\xi_u, \\
 \end{array} \]

\noindent
Then

\begin{equation}\label{eq:ISrelate}
\begin{array}{l}
\sum_{v \in \Vcal^m}\sum_{u \in N(v)} \Ee \xi_v\xi_u
=
S_2^m + S_4^m + 2S_6^m, \\
\sum_{v \in \Vcal^m}\sum_{u \in N(v)} \Ee \xi_v\Ee \xi_u
=
S_1^m + S_3^m + 2S_5^m.
\end{array}
\end{equation}

\noindent
For $x \in [0,1]^2$, let us define $\varphi^m(x) := m^{-2} \Ee \xi_i$
where $i\in\Ical$ is such that
$x \in c_i$.
Then
\[
S_1^m = \int_{[0,1]^2}\int_{B'(x;100r)} \varphi^m(x)\varphi^m(y){\dd}y{\dd}x,
\]
where $B'(x;100r)$ is the union of all cells $c_i$ with $\norm{x-c_i}
\leq 100r$.
Next we claim that $\lim_{m\to\infty} S_j^m \to I_j$ for all $1\leq j \leq 6$.
Now notice that

\[
0 \leq \varphi^m(x) \leq m^{-2} (n m^2) e^{-n m^2} \leq n,
\]
and, for every $x\in[0,1]^2$:
\[
\lim_{m\to\infty} \varphi_m(x) = n \cdot \Pee( E_x ).
\]
We can thus apply the dominated convergence theorem to show that
$\lim_{m\to\infty} S_1^m \to I_1$.
Similarly, we can show that
$S_j^m \to I_j$ for all $2\leq j\leq 6$.
Together with~\eqref{eq:PoConverge},~\eqref{eq:TwoMoments}
and~\eqref{eq:ISrelate}
this proves the Lemma.
\end{proofof}

\end{document}